\documentclass[11pt,a4paper]{article}

\usepackage[utf8]{inputenc}
\usepackage[T1]{fontenc}
\usepackage{lmodern}
\usepackage[margin=1.1in]{geometry}
\usepackage{microtype}

\usepackage{amsmath,amssymb,amsthm}
\usepackage[numbers,square,comma]{natbib}
\usepackage{tikz}
\usetikzlibrary{arrows.meta,positioning}
\usepackage{prooftree}
\usepackage[hidelinks]{hyperref}

\hypersetup{
  pdftitle={The Uniform Functional Interpretation with Informative Types},
  pdfauthor={Fernando Ferreira and Paulo Oliva}
}

\theoremstyle{plain}
\newtheorem{theorem}{Theorem}[section]
\newtheorem{corollary}[theorem]{Corollary}
\newtheorem{lemma}[theorem]{Lemma}
\newtheorem{proposition}[theorem]{Proposition}
\theoremstyle{definition}
\newtheorem{definition}[theorem]{Definition}
\newtheorem{example}[theorem]{Example}
\newtheorem{notation}[theorem]{Notation}
\newtheorem{remark}[theorem]{Remark}

\newcommand{\I}{{\mathcal I}}
\newcommand{\U}{{\mathcal U}}

\newcommand{\SourceT}{{\mathbf S}}
\newcommand{\TargetT}{{\mathbf T}}
\newcommand{\SourceTI}{\SourceT_{\I}}

\newcommand{\Ipp}{\I^{\mathrm{pp}}}
\newcommand{\Icp}{\I^{\mathrm{cp}}}
\newcommand{\Ucp}{\U^{\mathrm{cp}}}
\newcommand{\Imb}{\I^{\mathrm{mb}}}

\newcommand{\Umb}{\U^{\mathrm{mb}}}
\newcommand{\If}{\I^{\mathrm{f}}}

\newcommand{\Icb}{\I^{\mathrm{cb}}}
\newcommand{\Ucb}{\U^{\mathrm{cb}}}

\newcommand{\uInter}[3]{\langle#1\rangle^{#2}_{#3}}
\newcommand{\iInter}[3]{|#1|^{#2}_{#3}}
\newcommand{\rel}[1]{#1^\I}
\newcommand{\witness}[2]{\I_{#2}(#1)}
\newcommand{\uwitness}[4]{\langle \I_{#2}(#1) \rangle^{#3}_{#4}}
\newcommand{\iwitness}[4]{|\I_{#2}(#1)|^{#3}_{#4}}

\newcommand{\BB}{\mathbb{B}}
\newcommand{\NN}{\mathbb{N}}
\newcommand{\RR}{\mathbb{R}}

\newcommand{\Sw}{\mathcal{S}^\omega}
\newcommand{\Sws}{\mathcal{S}^\omega_*}
\newcommand{\Bw}{\mathcal{B}^\omega}

\newcommand{\Bbounded}{\triangleleft}

\newcommand{\Lang}[1]{{\mathcal L}_{#1}}
\newcommand{\LangI}[1]{{\mathcal L}_{#1}^{\mathcal{I}}}

\newcommand{\HAomega}{\mathbf{HA}^\omega}
\newcommand{\NHAomega}{\operatorname{\mathbf{N-HA}}^\omega}
\newcommand{\GodelT}{\mathsf{T}}

\newcommand{\Rec}{{\mathsf R}}
\newcommand{\Suc}{{\mathsf S}}
\newcommand{\Zero}{{\mathsf 0}}

\newcommand{\foralltilde}{\tilde{\forall}}
\newcommand{\existstilde}{\tilde{\exists}}

\newcommand{\pvec}[1]{\boldsymbol{#1}}
\newcommand{\proves}{\vdash}

\newcommand{\eqleft}[1]{\begin{itemize} \item[] $#1$ \end{itemize}}

\newcommand{\Bmaj}{\leq^B}
\newcommand{\iBmaj}{\trianglelefteq^B}

\newcommand{\WSorts}{{\cal W}}

\title{\textbf{The Uniform Functional Interpretation\\ with Informative Types}}

\author{%
Fernando Ferreira\thanks{This work is funded by FCT - Funda\c{c}\~ao para a Ci\^encia e a Tecnologia, I.P., through national funds, under the project UID/04561/2025 (\url{https://doi.org/10.54499/UID/04561/2025}).}\\[1mm]
{\small Departamento de Ci\^encias Matem\'aticas, Faculdade de Ci\^encias,}\\
{\small Universidade de Lisboa, 1749-016 Lisboa, Portugal}\\
{\small \texttt{fjferreira@fc.ul.pt}}
\and
Paulo Oliva\\[1mm]
{\small School of Electronic Engineering and Computer Science,}\\
{\small Queen Mary University of London, Mile End Road, London, E1 4NS, UK}\\
{\small \texttt{p.oliva@qmul.ac.uk}}
}

\date{}

\begin{document}

\maketitle

\begin{abstract}
We discuss a new approach to functional interpretations based on uniform quantification and relativization. The uniform quantification in the background permits a more penetrating analysis of principles related to collection and contra-collection. Relativization comes from a computationally informative notion of being an element of a given type. The approach is flexible. When the information takes the shape of bounds, we can recapture a form of the combination of G\"odel's functional {\em dialectica} interpretation with majorizability. When the information is ``canonical'' in function types, we obtain new functional interpretations and new models of G\"odel's theory $\GodelT$.
\end{abstract}

\noindent\textbf{Keywords:} Functional interpretations; Uniformities; Computational extraction; Collection and contra-collection; Semi-intuitionism; Proof mining.

\medskip

\noindent\textbf{Mathematics Subject Classification (2020):} Primary 03F10; Secondary 03F25, 03F35, 03F50, 03F55.

\section{Introduction}

Functional interpretations of arithmetic were introduced by Kurt G\"odel in his seminal {\em dialectica} paper \cite{Goedel(58)}. Since then, the original interpretation has been refined, modified and studied in different aspects and for various purposes. It was also extended to analysis \cite{Spector(62)} and admissible set theory \cite{AvigadTowsner(09), Ferreira(14), Ferreira(17)}, and  studied in purely logical settings \cite{FerreiraFerreira(17),OlivaFerreira(11),GerhardyKohlenbach(05),Oliva(06)}. A good reference, up to the end of the millennium, is the paper \cite{AvigadFeferman(98)}. In the last three decades, G\"odel's interpretation in combination with majorizability has been successfully used as an applied tool for extracting computational information from ordinary mathematical proofs (especially in the domain of iterative arguments in nonlinear functional analysis). We are referring here to Ulrich Kohlenbach's program of proof mining, conveniently explained in the book \cite{Kohlenbach(08)}. More recently, the book \cite{Diller(20)} appeared, authored by one of the early students of G\"odel's interpretation and covering material from a traditional perspective. It is not the aim of these brief lines to review the literature on functional interpretations, far too long and intricate, but rather to convey to the reader the idea that functional interpretations remain, to this day, an active area of study, with both theoretical and applied components.

The present paper rests on two main ideas. The first is that we regard functional interpretations as based on a background interpretation, dubbed the uniform interpretation, whose main feature is that quantifications are not computationally relevant. The second idea is that computational relevance is acquired via so-called type-informative predicates. Being an element of a type is contingent upon falling under a given type-informative predicate, whose interpretation is computationally relevant. Let us briefly discuss these two aspects in turn.

Uniform quantifications already appear in the works of Jean-Louis Krivine \cite{Krivine(94)} and Ulrich Berger \cite{Berger(04)}, but within a realizability framework. Let us write $\uInter{A}{\pvec a}{}$ for saying that witness $\pvec a$ realizes $A$. The quantifier clauses for uniform quantification are:
\eqleft{
\begin{array}{lcl}
	\uInter{\forall x^\sigma A(x)}{\pvec a}{} & :\equiv & \forall x^\sigma \uInter{A(x)}{\pvec a}{} \\[1mm]
	\uInter{\exists x^\sigma A(x)}{\pvec a}{} & :\equiv & \exists x^\sigma \uInter{A(x)}{\pvec a}{}.
\end{array}	
}
If one judges computational relevance by the realizing witnesses, uniform quantifications are not computationally relevant. In the framework of functional interpretations, with witnesses and counter-witnesses (written as a superscript and a subscript, respectively), the situation is subtler. The correct clauses for uniform quantification are
\eqleft{
\begin{array}{lcl}
	\uInter{\forall x^\sigma A(x)}{\pvec a}{\pvec b} & :\equiv & \forall x^\sigma \uInter{A(x)}{\pvec a}{\pvec b} \\[1mm]
	\uInter{\exists x^\sigma A(x)}{\pvec a}{\pvec B} & :\equiv & \exists x^\sigma \forall \pvec b \in \pvec B \, \uInter{A(x)}{\pvec a}{\pvec b},
\end{array}	
}
where $\pvec B$ above is a non-empty finite set. In Section \ref{uniform-interpretation-section}, we discuss these two clauses and prove the uniform soundness theorem. 

The clauses of a functional interpretation for the various connectives and quantifiers are intimately related with the principles that the very interpretation can realize. Let us see what can be said about the clauses for the uniform interpretation of the quantifiers. If one reads the functional interpretation of a formula $A$ as recasting its meaning computationally as $\exists \pvec a \forall \pvec b \, \uInter{A}{\pvec a}{\pvec b}$, then the clauses above convey the following informal reading:
\eqleft{
\begin{array}{lcl}
	\forall x^\sigma \exists \pvec a \forall \pvec b \, \uInter{A(x)}{\pvec a}{\pvec b} & \mathrm{is} & \exists \pvec a \forall \pvec b \, \forall x^\sigma  \uInter{A(x)}{\pvec a}{\pvec b} \\[1mm]
	 \exists x^\sigma \exists \pvec a \forall \pvec b \, \uInter{A(x)}{\pvec a}{\pvec b} & \mathrm{is} & \exists \pvec a \forall \pvec B \, \exists x^\sigma \forall \pvec b \in \pvec B \uInter{A(x)}{\pvec a}{\pvec b}.
\end{array}	
}
We find very interesting and illuminating the switch of the universal quantifier $\forall x^\sigma$ from the outside to the inside of the formula, as one reads the top clause from left to right, and the (partial) switch of the existential quantifier $\exists x^\sigma$ from the inside to the outside as one reads the second clause from right to left. We see in the raw, so to speak, principles of \emph{collection} and \emph{contra-collection} (respectively). The terminology for these principles was introduced in \cite{FerreiraOliva(05)}. Collection is related to the FAN principle of intuitionistic mathematics and to Kohlenbach's uniform boundedness principle \cite{Kohlenbach(06)}. Contra-collection is related to weak K\"onig's lemma \cite{FerreiraOliva(05)}.  

We claim that the quantification clauses of the uniform interpretation are the root explanation of why functional interpretations are, in the right circumstances, able to interpret non-constructive principles like weak K\"onig's lemma or even principles that are not set-theoretically true. The uniform interpretation is, however, not our goal but rather a stepping stone to computationally meaningful interpretations: the so-called $\I$-interpretations. When we want to analyze the $\I$-interpretations of (versions of) principles of collection or contra-collection, we can ``open the box'' and bring the background uniform interpretation to the fore, and look at the root source of the uniformity in these principles. This can be seen in the analysis of uniform or semi-uniform restricted quantifiers in Section \ref{uniform-boundedness-section}, especially Subsection \ref{semi-uniform-section}. 

If the uniform interpretation of a formula is devoid of computational relevance (i.e., if it has no witnesses, nor counter-witnesses), it is plain that the same is true of its quantifications. Lest we fall into triviality, a possibility for a computationally relevant account based on the uniform interpretation is to anchor to computationally relevant interpretations of the primitive relation symbols. This is easy to illustrate. Consider the ground type $\RR$ of the reals. If $<_\RR$ is a primitive (binary) relation symbol of the language of our theory of real numbers, we may define $\langle x <_\RR y\rangle^n$ as $x + \frac{1}{n+1} \leq y$. We may also define $\langle x =_\RR y\rangle_k$ as $|x-y| \leq \frac{1}{k+1}$. In other words, formulas derive their computational content from their atomic subformulas.

The second idea of the paper is, roughly, saying that being an element of a certain type is like falling under a given primitive unary relation symbol, and that the interpretations of these (so-called, type-informative) predicate symbols may be computationally relevant. For the purposes of this introduction, we denote by $x:\sigma$ the unary predicate saying that $x$ is of type $\sigma$ (in the body of the text, we use a different notation $\witness{x}{\sigma}$ for the type-informative predicates). For instance, we can have $\uInter{x:\NN}{n}{}$ as being $x \leq n$. This is the so-called bounding interpretation at the ground type $\NN$ (Section \ref{bounding-section}). Alternatively, we could have  $\uInter{x:\NN}{n}{}$ as being $x = n$, the precise interpretation at the ground type $\NN$ (Section \ref{precise-section}). 

The interpretation of the type-informative predicates at the ground types does {\em not} determine the interpretation of the type-informative predicates at the function types. We study several possibilities in Section \ref{precise-section} and in Section \ref{bounding-section}. There is, however, always a canonical way of interpreting the type-informative predicates at the function types, given the interpretations of the type-informative predicates at the ground types. We define and study these canonical extensions in general in Section \ref{canonical-section} and, later on, we present two concrete canonical interpretations of arithmetic, one in Subsection \ref{canonical-precise-section}, the other in Subsection \ref{canonical-bounding-section}. These two canonical extensions define new functional interpretations of Heyting arithmetic. Let us briefly highlight some interesting features of the canonical interpretations. In contrast to G\"odel's {\em dialectica} interpretation, the principles of pointwise continuity and extensionality for type 2 functionals are realized by our first canonical interpretation (see Proposition \ref{prop-cont-2} and its corollary). The second canonical interpretation is able to realize new collection and contra-collection principles (see Subsection \ref{analytic-applications-section}). Both canonical interpretations lead to natural definitions of models of G\"odel's theory $\mathsf T$ (end of Subsection \ref{canonical-precise-section} and end of Subsection \ref{canonical-bounding-section}).

Some modern functional interpretations enjoy a monotonicity property, in the sense that if a witness validates a formula (with a given counter-witness), then a ``bigger'' witness also does. A primeval version of this property appeared with the monotone functional interpretation of Kohlenbach in \cite{Kohlenbach(96)}. The property itself appeared, for the first time, in \cite{FerreiraOliva(05)}. Many more examples have been considered meanwhile, not only in the context of arithmetic (as, for instance, the interpretations of non-standard arithmetic in \cite{BergBriseidSafarik(12)} or of subsystems of second-order arithmetic in \cite{Ferreira(20PM)}), but also in admissible set theory (as in \cite{AvigadTowsner(09)}, in the form of inductive definitions, or in \cite{Ferreira(14), Ferreira(17)}) and in classical logic in \cite{FerreiraFerreira(17)}. The latter paper introduced the terminology of ``cumulative interpretation'' for a functional interpretation that enjoys a monotonicity property as described above. 

The framework of this paper offers an explanatory account of cumulative vis-\`a-vis non-cumulative functional interpretations. The crux of the matter -- as it happens -- lies with the interpretation of disjunction. We take disjunction not as a primitive logical symbol but rather as defined with the aid of an existential quantification over the Boolean type $\BB$. Section \ref{disjunction-section} considers two possibilities for the type-informative predicate of $\BB$. One possibility takes $\uInter{x:\BB}{b}{}$ as $x=b$, where $b$ is a Boolean. This choice gives the {\em dialectica} clause of disjunction and the interpretation is precise, not cumulative. Else, we can have $\uInter{x:\BB}{}{}$, without witnesses or counter-witnesses, as simply being the true statement that $x$ is a Boolean. This alternative requires a so-called notion of joining of bounds and enjoys a monotonicity property. It gives rise to cumulative interpretations. It also interprets an interesting form of semi-intuitionism. This is worked out and explained in Subsection \ref{uniform-booleans-sec}.

\section{The uniform interpretation}
\label{uniform-interpretation-section}

Our interpretations always involve two theories framed in the language of finite types: a \emph{source theory} $\SourceT$, the theory being interpreted, and a \emph{target theory} $\TargetT$, the theory where the interpretation takes place and where the conclusion of the soundness theorem is verified. The source theory is always intuitionistic in this paper, but the target theory need not be. We will be quite informal, and sometimes vague, regarding the target theories because, for the studies of this paper, one need not be too precise about them. That being said, one should be precise about the terms that occur in the so-called witnessing types of the target theory.

In this section, we present the uniform interpretation of (intuitionistic) theories of finite types. We make three design choices concerning the source language and its theory of finite types.

\begin{enumerate}
    \item There are many (potentially infinitely many) ground types, as opposed to the more common approach, according to which there is only one ground type. We use capital letters $X, Y, Z, \ldots$ for the ground types of the source theory and assume that, for each ground type $X$, there is a constant $e_X$ of type $X$ (see Remark \ref{terms-source} for the necessity of this assumption).
    The finite types are built as usual: the ground types are finite types and, if $\sigma$ and $\tau$ are finite types, then $\sigma \to \tau$ is a finite type.
    
    \item There is always a logical equality sign at a given type $\sigma$, denoted by $\equiv_\sigma$. The source theory has the usual axioms of (logical) equality:
    \[ 
        \forall x^\sigma (x \equiv_\sigma x); \quad\quad \forall x^\sigma, y^\sigma (A \wedge \, x \equiv_\sigma y \, \to \, A');
    \]
    where $A$ is a quantifier-free formula (of the source language) and $A'$ is obtained from $A$ by replacing one or more variables $x$ by $y$. As is well known, the above axioms entail that $\equiv_\sigma$ is an equivalence relation. 

    The equality signs are convenient for formulating the usual axioms for the combinators:
    \eqleft{
    \begin{array}{lcl}
        \Pi_{\rho, \sigma} x^\rho y^\sigma 
            & \equiv_{\rho} & x \\[1mm]
        \Sigma_{\rho, \sigma, \gamma} x^{\gamma \to \sigma \to \rho} y^{\gamma \to \sigma} z^{\gamma} 
            & \equiv_{\rho} & x z (y z)
    \end{array}
    }
    so that $\lambda$-terms can be defined. 
    
    Logical equality has a fixed interpretation, as we shall see. It is often useful to work with other kinds of equality (e.g. extensional equality), having different interpretations. Unlike logical equality, these other equalities will be treated like any other (binary) relation symbol. 
    
    \item The last design choice is the most peculiar. Until Section \ref{disjunction-section}, we do not admit disjunction in our language. There are very good reasons for the absence of the connective of disjunction in the source language at this stage of the discussion. Disjunction will be added (and discussed) in Section \ref{disjunction-section}.
\end{enumerate}

Otherwise, the language is as usual. The intuitionistic logic of the source theory $\SourceT$ is formulated in a Gentzen style sequent calculus (with cut) in which, of course, the rules of disjunction are absent (see system $\mathbf{G1i}$ of \cite[Section 3.1]{Troelstra(96)}).

\subsection{Witnessing types and their elements}\label{witnessing-section}

The terms and types of the target language $\Lang{\TargetT}$ are assumed to contain the terms and types of the source language $\Lang{\SourceT}$. $\Lang{\TargetT}$ is also assumed to have a distinguished set of ground types $\WSorts = \{W, W', W_1, \ldots\}$ referred to as the \emph{ground witnessing types}. Built upon the ground witnessing types, we have the function types and also star types, i.e., types whose constituents are non-empty finite subsets of a given type. It will be apparent why we need star types in the target language (do notice that, for the sake of simplicity, we do not allow star types in the source language).

\begin{definition}[Witnessing types] Given a collection of ground witnessing types $\WSorts$, we define the {\em witnessing types} as the finite star types over $\WSorts$, denoted by $\WSorts^\omega_*$:
\begin{itemize}
    \item ground witnessing types are in $\WSorts^\omega_*$,
    \item function types: if $\rho, \tau \in \WSorts^\omega_*$ then $\rho \to \tau \in \WSorts^\omega_*$,
    \item non-empty finite sets: if $\rho \in \WSorts^\omega_*$ then $\rho^* \in \WSorts^\omega_*$.
\end{itemize}    
The elements of a witnessing type will be called {\em witnessing elements}. A term of a witnessing type is called a {\em witnessing term}.
\end{definition}

The type language based on $\WSorts$ must contain the usual combinators $\Pi$ and $\Sigma$ (therefore enjoying a lambda calculus). It must also include star constants necessary to deal with the star types. They are $\mathfrak{s}_{\sigma}$ of type $\sigma\to \sigma^*$, $\cup_{\sigma}$ of type $\sigma^*\to (\sigma^*\to \sigma^*)$ and $\bigcup_{\sigma,\tau}$ of type $\sigma^*\to (\sigma\to \tau^*)\to \tau^*$. We write $\{q\}$ instead of $\mathfrak{s}q$, $t\cup q$ instead of $\cup tq$ and $\bigcup_{w\in t} qw$ instead of $\bigcup tq$. The target language $\Lang{\TargetT}$ also includes binary relation symbols $\in_\sigma$ infixing between terms of type $\sigma$ on the left and terms of type $\sigma^*$ on the right. For ease of reading, we usually omit the type subscripts of the membership relations. There are axioms for the star constants corresponding to an extended reduction process (on top of the combinatorial reductions):
\begin{itemize}
    \item $\bigcup_{y \in \{x\}} f y = f x$
    \item $\bigcup_{z \in x \cup y} f z = (\bigcup_{z \in x} f z) \cup (\bigcup_{z \in y} f z)$.
\end{itemize}
We work informally with equality in the target theory $\TargetT$, as ordinary mathematicians do. It is, nevertheless, worth noting that the extended reduction calculus still enjoys the properties of strong normalization and confluence. We refer the reader to the paper \cite{FerreiraFerreira(17)} for this and connected issues.

Summing up, the language $\Lang{\TargetT}$ of the target theory $\TargetT$ has a distinguished stock of terms within a lambda calculus framework (the witnessing terms). These witnessing terms live in $\WSorts^\omega_*$ and enjoy basic set-theoretic properties in the target theory. For instance, $\forall x (x \in \{a\} \leftrightarrow x=a)$. They play a crucial role in the functional interpretations and in proving their soundness. The logic of the target theory can be either intuitionistic or classical (and disjunction may be present). In proving the soundness of functional interpretations (for instance, the soundness of weakening), we must ensure that there is a closed term of each witnessing type. Therefore, we postulate the existence of a constant $e_W$ in each ground witnessing type $W$. From these constants, we can recursively define a closed term for each type in $\WSorts^\omega_*$ as follows: $e_{\rho \to \tau}$ is $\lambda x^\rho. e_\tau$ and $e_{\rho^*}$ is $\{ e_\rho \}$.

\subsection{The uniform interpretation and its soundness}
\label{sec-uniform-interpretation}

The uniform interpretation of a source theory $\SourceT$ into a target theory $\TargetT$ is parametrized by the following data. To each $r$-ary relation symbol $R(x_1, \ldots, x_r)$ of the source language $\Lang{\SourceT}$, we associate:

\begin{itemize} 
    \item Two tuples $\pvec \tau_R^+$ and $\pvec \tau_R^-$ of witnessing types in $\WSorts^\omega_*$ (these tuples may be empty).
    \item A relation $\uInter{R(x_1, \ldots, x_r)}{\pvec c}{\pvec d}$, expressed in the target language $\Lang{\TargetT}$, between elements $x_1, \ldots, x_r$ of types of $\Lang{\SourceT}$, a tuple of elements $\pvec c$ of types $\pvec \tau_R^+$ and a tuple of elements $\pvec d$ of types $\pvec \tau_R^-$. This is called the {\em information relation} associated with $R(x_1, \ldots, x_r)$.
    \item The information relation associated with a logical equality symbol $\equiv_\sigma$ of $\Lang{\SourceT}$ is fixed as follows: $\pvec \tau_{\equiv_\sigma}^+$ and $\pvec \tau_{\equiv_\sigma}^-$ are the empty tuples and $\uInter{x\equiv_\sigma y}{}{}$ is $x=y$. 
\end{itemize}

\begin{definition}[Base interpretation of $\Lang{\SourceT}$ into $\Lang{\TargetT}$] A \emph{base interpretation of $\Lang{\SourceT}$ into $\Lang{\TargetT}$} is a family $$(\pvec \tau_R^+, \pvec \tau_R^-, \uInter{R(x_1, \ldots, x_r)}{\pvec c}{\pvec d})_{R}$$ as above, which associates to each relation symbol $R$ of $\Lang{\SourceT}$ an information relation in $\Lang{\TargetT}$.
\end{definition}

Given a base interpretation, we can extend it to all formulas of $\Lang{\SourceT}$. This is the content of the next definition.

\begin{definition}[$\U$-interpretation of $\Lang{\SourceT}$ into $\Lang{\TargetT}$] \label{def-uniform-interpretation} Let be given a base interpretation of $\Lang{\SourceT}$ into $\Lang{\TargetT}$. For each formula $A$ of $\Lang{\SourceT}$, we define its $\U${\em -interpretation} $\uInter{A}{\pvec a}{\pvec b}$ into $\Lang{\TargetT}$. The definition is by induction on the logical structure of $A$. \\[1mm]
For atomic formulas $R(t_1, \ldots, t_n)$, its $\U$-interpretation is defined as the given information relation $\uInter{R(t_1, \ldots, t_n)}{\pvec c}{\pvec d}$. For $\bot$ we define
\eqleft{
\begin{array}{lcl}
	\uInter{\bot}{}{} 
		& :\equiv & \bot. 
\end{array}
}
Assuming that $A$ and $B$ have $\U$-interpretations $\uInter{A}{\pvec a}{\pvec b}$ and $\uInter{B}{\pvec c}{\pvec d}$, respectively, we define:
 
\eqleft{
\begin{array}{lcl}
	\uInter{A \wedge B}{\pvec a, \pvec c}{\pvec b, \pvec d} 
		& :\equiv & \uInter{A}{\pvec a}{\pvec b} \wedge \uInter{B}{\pvec c}{\pvec d} \\[1mm]
	\uInter{A \to B}{\pvec f, \pvec g}{\pvec a, \pvec d} 
		& :\equiv & \forall \pvec b \in \pvec g \pvec a \pvec d \uInter{A}{\pvec a}{\pvec b} \to 
				\uInter{B}{\pvec f \pvec a}{\pvec d} \\[1mm]
	\uInter{\forall x^\sigma A(x)}{\pvec a}{\pvec b} & :\equiv & \forall x^\sigma \uInter{A(x)}{\pvec a}{\pvec b} \\[1mm]
	\uInter{\exists x^\sigma A(x)}{\pvec a}{\pvec B} & :\equiv & \exists x^\sigma \forall \pvec b \in \pvec B \, \uInter{A(x)}{\pvec a}{\pvec b}.
\end{array}	
}
\end{definition}

Note that in the base interpretation we assume that $\pvec c$ and $\pvec d$ in $\uInter{R(x_1, \ldots, x_r)}{\pvec c}{\pvec d}$ have witnessing types. That implies that, for arbitrary formulas $A$, the tuples $\pvec a$ and $\pvec b$ in $\uInter{A}{\pvec a}{\pvec b}$ also have witnessing types.

\begin{notation}
    If in the $\U$-interpretation of  $A$, the formula $\uInter{A}{\pvec a}{\pvec b}$ is such that $\pvec a$ is the empty tuple, we say that \emph{$A$ has no positive $\U$-witnesses}, while if $\pvec b$ is empty, we say that \emph{$A$ has no negative $\U$-witnesses}. When $A$ has neither positive nor negative $\U$-witnesses, we say that $A$ is \emph{$\U$-witness-free}. For instance, $\bot$ is $\U$-witness-free.
\end{notation}

\begin{remark}\label{equals}
    If we highlight a variable $x$ in the formula $A$, the formula $\uInter{A}{\pvec a}{\pvec b}$ induces a relation between $x$, $\pvec a$ and $\pvec b$. Take now a term $t$ of the same type as the variable $x$, free for $x$ in $A$ (for simplicity of notation, we do not display the variables of $t$). It is clear that $\uInter{A[t/x]}{\pvec a}{\pvec b} = \uInter{A}{\pvec a}{\pvec b}[t/x]$, assuming $x$ is not free in $\pvec a$ and $\pvec b$. This fact is embedded in the notation, but we wanted to bring it out explicitly here. \end{remark}

\begin{remark}\label{fixing-base} One should think of the $\U$-interpretation as a \emph{family} of interpretations parametrized by a base interpretation of $\Lang{\SourceT}$ into $\Lang{\TargetT}$. Each choice of the base interpretation might give rise to a different $\U$-interpretation (see examples of base interpretations in Section \ref{relativization-section}, and concrete interpretations of arithmetic in Sections \ref{precise-section} and \ref{bounding-section}). \end{remark}

As can be seen from the clause of implication, the $\U$-interpretation relies upon the Diller-Nahm interpretation \cite{DillerNahm(74)}. In fact, finite (and non-empty) collections of counter-witnesses are hypothesized in the antecedent of the interpreting formula. We find that using the Diller-Nahm clause for implication is the most natural way of proceeding, given our general framework. If one insisted on having only one tuple of counter-witnesses, as happens with G\"odel's {\em dialectica} interpretation, then one would need to be able to make definitions by cases according to whether the interpretation of a formula holds or not. Note that, in our framework, these formulas need not be quantifier-free and, therefore, we do not have such terms. In the context of pure classical logic, \cite{GerhardyKohlenbach(05)} has a treatment which introduces, in an ad hoc manner, terms for definition by cases. It is unclear whether such treatment can be adapted to our general setting. So, we opted for the finiteness approach of Diller-Nahm.

The uniform character of the $\U$-interpretations comes from the treatment of the quantifiers. In a sense, the interpretation of the quantifiers is not computationally relevant. This is apparent in the clause of the universal quantifier. In the case of the existential quantifier, finite sets of counter-witnesses for $\exists x^\sigma A(x)$ appear, but they only appear if $A(x)$ already has negative $\U$-witnesses to start with. If there are neither positive nor negative $\U$-witnesses for $A(x)$, there are also none for $\exists x^\sigma A(x)$ and $\forall x^\sigma A(x)$. We may say that in the uniform interpretation of the quantifiers, no new information is created; rather, the original information is carried along and, in the case of the existential quantifier, its negative information is upgraded to finite sets so that a proclaimed existence is able to work against all the tuples coming from these finite sets of counter-witnesses. 

Could the clause for existential quantification be the simpler 
\eqleft{
    \begin{array}{lcl}
    \uInter{\exists x^\sigma A(x)}{\pvec a}{\pvec b} & :\equiv & \exists x^\sigma \uInter{A(x)}{\pvec a}{\pvec b}\,?
	\end{array}
}	
The answer is no, at least not if we keep the clauses of the other connectives as they are. The problem lies with the rule $\exists$-L (existential introduction on the left). We will come back to this issue in the proof of the Uniform Soundness Theorem (see Theorem \ref{thm-uniform-soundness}).

\begin{definition}[Formulas $\U$-interpretable in $\TargetT$] Let be given a base interpretation of $\Lang{\SourceT}$ into $\Lang{\TargetT}$. A formula $A(\pvec x)$ of $\Lang{\SourceT}$ is said to be $\U$-\emph{interpretable in $\TargetT$} if there are closed witnessing terms $\pvec t$ (i.e. closed terms of type $\WSorts^\omega_*$) such that $\TargetT$ proves $\forall \pvec b\, \uInter{A(\pvec x)}{\pvec t}{\pvec b}$. 
\end{definition}

Note that $A(\pvec x)$ may have free variables $\pvec x$. Nevertheless, $\pvec t$ is a tuple of {\em closed} witnessing terms. This definition makes sense because the interpretation of quantifiers is ``uniform''. This is confirmed by the Uniform Soundness Theorem below.

\begin{definition}[$\TargetT$-sound base interpretation of $\SourceT$] \label{def-sound-base} A base interpretation of $\Lang{\SourceT}$ into $\Lang{\TargetT}$ is said to be a $\TargetT$-\emph{sound interpretation of $\SourceT$} if the non-logical axioms of $\SourceT$ are $\U$-interpretable in $\TargetT$.
\end{definition}

When studying functional interpretations (of intuitionistic theories), it is convenient to have in mind the formula $\exists \pvec a \forall \pvec b \, \uInter{A}{\pvec a}{\pvec b}$. The conclusion of a pertinent soundness theorem gives, under appropriate conditions, a tuple of witnessing terms for the existential quantifier.

\begin{theorem}[Uniform Soundness Theorem] \label{thm-uniform-soundness} Let be given a $\TargetT$-sound base interpretation of $\SourceT$. 
If 
\[
    \Gamma(\pvec x) \proves_{\SourceT} A(\pvec x)
\]
then there are tuples of closed witnessing terms $\pvec s, \pvec t$ such that
$$\forall \pvec b \in \pvec s \pvec a \pvec d \, \uInter{\Gamma(\pvec x)}{\pvec a}{\pvec b} \proves_{\TargetT} \uInter{A(\pvec x)}{\pvec t \pvec a}{\pvec d},$$
where $\pvec a, \pvec d$ and $\pvec x^{\pvec \sigma}$ are free variables. \\
In particular, any theorem of $\SourceT$ is $\U$-interpretable in $\TargetT$. 
\end{theorem}

\begin{notation}
    In the above, if $\Gamma$ is $B_1, \ldots, B_n$, $\pvec s$ is $\pvec s_1, \ldots, \pvec s_n$, $\pvec a$ is $\pvec a_1, \ldots, \pvec a_n$, and $\pvec b$ is $\pvec b_1, \ldots, \pvec b_n$, then $\forall \pvec b \in \pvec s \pvec a \pvec d \, \uInter{\Gamma(\pvec x)}{\pvec a}{\pvec b}$ abbreviates $\bigwedge_{i=1}^n \forall \pvec b_i \in \pvec s_i \pvec a \pvec d \, \uInter{B_i}{\pvec a_i}{\pvec b_i}$.
\end{notation}

\begin{proof} As we have said, formal proofs take place in a Gentzen style sequent calculus (cf. system $\mathbf{G1i}$ of \cite{Troelstra(96)}) without the rules of disjunction. The argument for the theorem above proceeds by induction on the number of steps of the formal derivation of the sequent $\Gamma(\pvec x) \proves A(\pvec x)$ in $\SourceT$. \\[1mm]
We first discuss the equality axioms, the axioms for combinators and the non-logical axioms. The latter are interpreted by the hypothesis that the base interpretation is sound (cf.\ Definition \ref{def-sound-base}). The first equality axiom (reflexivity) and the axioms for combinators are self-interpretable, since they are $\U$-witness-free. The interpretation of the other equality axioms, namely those of the form 
\[ \forall x^\sigma, y^\sigma (A \wedge \, x \equiv_\sigma y \, \to \, A'), \] 
with appropriate (quantifier-free) formulas $A$, $A'$ and variables $x$ and $y$, asks for closed terms $\pvec q$ and $\pvec t$ such that the target theory proves the following: for all $\pvec u$ and $\pvec v'$,
$$\forall x^\sigma \forall y^\sigma \left( \forall \pvec v \in \! \pvec q \pvec u \pvec v' (\uInter{A}{\pvec u}{\pvec v} \wedge x = y) \to\uInter{A'}{\pvec t \pvec u}{\pvec v'} \right).$$ 
Let $\pvec q \pvec u \pvec v' = \pvec\{\pvec v'\pvec\}$ and $\pvec t \pvec u = \pvec u$. These terms do the job because the above follows from the axioms of equality in the target theory (since $x, y$ are not free in $\pvec t \pvec u$ and $\pvec v, \pvec v'$, cf. Remark \ref{equals}).
The structural rules and the propositional rules (rules for the logical connectives) pose no problem because they follow directly from the proof of the soundness of the Diller-Nahm interpretation \cite{DillerNahm(74)}. However, for the benefit of the reader, we study the three structural rules of contraction, weakening and cut. \\[1mm]
We study the structural rule of contraction because its analysis necessitates the clause for implication that takes advantage of finite collections of counter-witnesses. The rule of contraction in the sequent calculus for intuitionistic logic is
\[
\begin{prooftree}
    \Gamma, A, A \proves B
    \justifies
    \Gamma, A \proves B
\end{prooftree}
\]
Suppose that we have closed terms $\pvec r, \pvec t, \pvec t'$ and $\pvec q$ witnessing the premise of the rule. That is, assume that
$$\forall \pvec b \in \pvec r \pvec a \pvec c \pvec c' \pvec v \, \uInter{\Gamma}{\pvec a}{\pvec b}, \forall \pvec d \in \pvec t \pvec a \pvec c \pvec c' \pvec v \, \uInter{A}{\pvec c}{\pvec d},  \forall \pvec d \in \pvec t' \pvec a \pvec c \pvec c' \pvec v \, \uInter{A}{\pvec c'}{\pvec d} \proves_{\TargetT} \uInter{B}{\pvec q \pvec a \pvec c \pvec c'}{\pvec v}.$$ 
We need to exhibit closed terms $\pvec r', \pvec s$ and $\pvec p$ such that
$$\forall \pvec b \in \pvec r' \pvec a \pvec c \pvec v \, \uInter{\Gamma}{\pvec a}{\pvec b}, \forall \pvec d \in \pvec s \pvec a \pvec c \pvec v \, \uInter{A}{\pvec c}{\pvec d} \,\, \proves_{\TargetT} \, \uInter{B}{\pvec p \pvec a \pvec c}{\pvec v}.$$
Let $\pvec r', \pvec s$ and $\pvec p$ be terms so that we have the equalities $\pvec r'\pvec a \pvec c \pvec v = \pvec r \pvec a \pvec c \pvec c \pvec v$, $\pvec p \pvec a \pvec c = \pvec q \pvec a \pvec c \pvec c$ and $\pvec s \pvec a \pvec c \pvec v = \pvec t \pvec a \pvec c \pvec c \pvec v \cup \pvec t' \pvec a \pvec c \pvec c \pvec v$. This choice clearly works. \\[1mm] 
We also study the structural rule of weakening. This is the rule
\[
\begin{prooftree}
    \Gamma \proves C
    \justifies
    \Gamma, A \proves C
\end{prooftree}
\]
Suppose that we have closed terms $\pvec r$ and $\pvec q$ witnessing the premise of the rule. That is, we assume
$$\forall \pvec b \in \pvec r \pvec a \pvec v \, \uInter{\Gamma}{\pvec a}{\pvec b} \,\proves_{\TargetT}\, \uInter{C}{\pvec q \pvec a}{\pvec v}.$$
We must obtain closed terms $\pvec t$, $\pvec s$ and $\pvec p$ such that
$$\forall \pvec b \in \pvec t \pvec a \pvec c \pvec v \, \uInter{\Gamma}{\pvec a}{\pvec b}, \forall \pvec d \in \pvec s \pvec a \pvec c \pvec v \, \uInter{A}{\pvec c}{\pvec d}\,\proves_{\TargetT}\, \uInter{C}{\pvec p \pvec a \pvec c}{\pvec v}.$$
Let $\pvec t$, $\pvec s$ and $\pvec p$ be such that $\pvec t \pvec a \pvec c \pvec v = \pvec r \pvec a \pvec v$, $\pvec s \pvec a \pvec c \pvec v = \{\pvec e\}$ and $\pvec p \pvec a \pvec c = \pvec q \pvec a$. Here $\pvec e$ is a tuple of closed terms of appropriate types and $\{ \pvec e \}$ is the corresponding tuple of singletons (see the end of Subsection \ref{witnessing-section}). This choice works. \\[1mm]
The third rule that we study is the cut rule.
\[
\begin{prooftree}
    \Gamma \proves B \quad \quad \Delta, B \proves C
    \justifies
    \Gamma, \Delta \proves C
\end{prooftree}
\]
So, suppose that there are closed terms $\pvec t$ and $\pvec q$ which witness the first premise of the cut rule. So, 
$$\forall {\pvec b} \in \pvec t \pvec a \pvec d \, \uInter{\Gamma}{\pvec a}{\pvec b} \, \proves_{\TargetT} \, \uInter{B}{\pvec q \pvec a}{\pvec d}.$$ 
Suppose, further, that there are closed terms $\pvec t', \pvec s$ and $\pvec r$ which witness the second premise of the cut rule, i.e., 
$$\forall {\pvec f} \in {\pvec t'}{\pvec e}{\pvec c}{\pvec v} \, \uInter{\Delta}{\pvec e}{\pvec f}, \forall {\pvec d} \in {\pvec s}{\pvec e}{\pvec c}{\pvec v} \uInter{B}{\pvec c}{\pvec d}\,\, \proves_{\TargetT} \, \uInter{C}{\pvec r \pvec e \pvec c}{\pvec v}.$$
We must exhibit closed terms ${\pvec l}, {\pvec m}$ and ${\pvec p}$ such that 
$$\forall {\pvec b} \in {\pvec l}{\pvec a}{\pvec e}{\pvec v} \, \uInter{\Gamma}{\pvec a}{\pvec b}, \forall {\pvec f} \in {\pvec m}{\pvec a}{\pvec e}{\pvec v} \, \uInter{\Delta}{\pvec e}{\pvec f} \,\, \proves_{\TargetT} \, \uInter{C}{\pvec p\pvec a \pvec e}{\pvec v}.$$
It is easy to check that $\pvec m \pvec a \pvec e \pvec v = \pvec t' \pvec e (\pvec q \pvec a) \pvec v$, $\pvec p \pvec a \pvec e = \pvec r \pvec e (\pvec q \pvec a)$ and $$\pvec l\pvec a \pvec e \pvec v = \bigcup_{\pvec d \in \pvec s \pvec e (\pvec q \pvec a)\pvec v} \pvec t \pvec a \pvec d$$ do the job. \\[1mm]
We finish by studying each of the four quantifier rules because here there is some novelty. \\[1mm]
Consider the rule $\exists$-L: 
\[
\begin{prooftree}
    \Gamma, A(z^\sigma) \proves B
    \justifies
    \Gamma, \exists z^\sigma A(z^\sigma) \proves B
\end{prooftree}
\]
where $\sigma$ is a finite type and the variable $z^\sigma$ is neither free in $\Gamma$ nor in $B$. Suppose that we have closed terms $\pvec u, \pvec s$ and $\pvec t$ witnessing the premise of the rule. That means
$$\forall \pvec e \in \pvec u \pvec c \pvec a \pvec d \, \uInter{\Gamma}{\pvec c}{\pvec e}, \forall \pvec b \in \pvec s \pvec c \pvec a \pvec d \uInter{A(z)}{\pvec a}{\pvec b} \proves_{\TargetT} \uInter{B}{\pvec t \pvec c \pvec a}{\pvec d}.$$ 
In order to realise the conclusion of the rule, we must define terms $\pvec v, \pvec r$ and $\pvec q$, such that
$$\forall \pvec e \in \pvec v \pvec c \pvec a \pvec d \, \uInter{\Gamma}{\pvec c}{\pvec e}, \forall \pvec B \in \pvec r \pvec c \pvec a \pvec d \, \exists z^\sigma \forall \pvec b \in \pvec B \uInter{A(z)}{\pvec a}{\pvec b} \proves_{\TargetT} \uInter{B}{\pvec q \pvec c \pvec a}{\pvec d}.$$ 
We can just take $\pvec q = \pvec t$, $\pvec v = \pvec u$ and $\pvec r \pvec c \pvec a \pvec d = \pvec \{\pvec s \pvec c \pvec a \pvec d\pvec \}$. \\[1mm]
Before continuing with the study of the other three quantifier rules, we want to note that the simpler clause for the existential quantifier discussed before the statement of the theorem, namely putting $\uInter{\exists x^\sigma A(x)}{\pvec a}{\pvec b}$ as $\exists x^\sigma \uInter{A(x)}{\pvec a}{\pvec b}$, would fail at this point. With the simpler clause, one would need terms $\pvec v, \pvec r$ and $\pvec q$, such that 
$$\forall \pvec e \in \pvec v \pvec c \pvec a \pvec d \, \uInter{\Gamma}{\pvec c}{\pvec e}, \forall \pvec b \in \pvec r \pvec c \pvec a \pvec d \, \exists z^\sigma \uInter{A(z)}{\pvec a}{\pvec b} \proves_{\TargetT} \uInter{B}{\pvec q \pvec c \pvec a}{\pvec d}.$$ 
The problem lies with the fact that the antecedent of the premise of the rule requires the stronger $\exists z^\sigma \forall \pvec b \in \pvec s \pvec c \pvec a \pvec d \uInter{A(z)}{\pvec a}{\pvec b}$ whereas, in the conclusion of the rule, the antecedent has instead the weaker condition $\forall \pvec b \in \pvec r \pvec c \pvec a \pvec d \,\exists z^\sigma \uInter{A(z)}{\pvec a}{\pvec b}$, and in general we do not have 
$$\forall \pvec b \in \pvec r \pvec c \pvec a \pvec d \,\exists z^\sigma \uInter{A(z)}{\pvec a}{\pvec b} \proves_{\TargetT} \exists z^\sigma \forall \pvec b \in \pvec r \pvec c \pvec a \pvec d \uInter{A(z)}{\pvec a}{\pvec b}.$$ 
Let us analyze the rule $\forall$-R:
\[
\begin{prooftree}
    \Gamma \proves A(z^\sigma)
    \justifies
    \Gamma  \proves \forall z^\sigma A(z^\sigma)
\end{prooftree}
\]
where the variable $z^\sigma$ is not free in $\Gamma$. Suppose that we have closed terms $\pvec q$ and $\pvec t$ witnessing the premise of the rule. Hence, 
$$\forall \pvec b \in \pvec q \pvec a \pvec d \, \uInter{\Gamma}{\pvec a}{\pvec b} \proves_{\TargetT} \uInter{A(z^\sigma)}{\pvec t \pvec a}{\pvec d}.$$
We get 
$$\forall \pvec b \in \pvec q \pvec a \pvec d \uInter{\Gamma}{\pvec a}{\pvec b} \proves_{\TargetT} \forall z^\sigma \uInter{A(z^\sigma)}{\pvec t \pvec a}{\pvec d}$$
and, obviously, this is the {\em desideratum} for the conclusion. \\[1mm]
We now look at the rule $\exists$-R:
\[
\begin{prooftree}
    \Gamma \proves A(r^\sigma)
    \justifies
    \Gamma \proves \exists y^\sigma \!A(y^\sigma)
\end{prooftree}
\]
Assume that we have terms $\pvec s$ and $\pvec t$ witnessing the premise of the rule, i.e.,
$$\forall \pvec b \in \pvec s \pvec a \pvec f \, \uInter{\Gamma}{\pvec a}{\pvec b}\,\, \proves_{\TargetT} \, \uInter{A(r)}{\pvec t \pvec a}{\pvec f}.$$
Hence
$$
    \forall \pvec f \in \pvec F \forall \pvec b \in \pvec s \pvec a \pvec f \, \uInter{\Gamma}{\pvec a}{\pvec b}\,\, \proves_{\TargetT}\, \forall \pvec f \in \pvec F \uInter{A(r)}{\pvec t \pvec a}{\pvec f}
$$
and, {\em a fortiori}, 
$$\forall \pvec b \in \bigcup_{\pvec f \in \pvec F} \pvec s \pvec a \pvec f \, \uInter{\Gamma}{\pvec a}{\pvec b}\,\, \proves_{\TargetT} \, \exists y^\sigma \forall \pvec f \in \pvec F \uInter{A(y)}{\pvec t \pvec a}{\pvec f},$$
for all $\pvec a$ and $\pvec F$. This is what is needed. \\[1mm]
Finally, we discuss the rule $\forall$-L:
\[
\begin{prooftree}
    \Gamma, A(r^\sigma) \proves B
    \justifies
    \Gamma, \forall y^\sigma A(y^\sigma) \proves B
\end{prooftree}
\]
Assume that we have terms which witness the premise, i.e., that there are terms $\pvec q, \pvec r$ and $\pvec t$ such that 
$$\forall \pvec b \in \pvec q \pvec a\pvec e\pvec v \uInter{\Gamma}{\pvec a}{\pvec b}, \forall \pvec f \in \pvec r \pvec a\pvec e\pvec v \uInter{A(r)}{\pvec e}{\pvec f} \,\,\proves_{\TargetT}\, \,\uInter{B}{\pvec t \pvec a\pvec e}{\pvec v}.$$
Hence:
$$\forall \pvec b \in \pvec q \pvec a\pvec e\pvec v \uInter{\Gamma}{\pvec a}{\pvec b}, \forall \pvec f \in \pvec r \pvec a\pvec e\pvec v \forall y^\sigma \uInter{A(y^\sigma)}{\pvec e}{\pvec f} \,\,\proves_{\TargetT}\, \,\uInter{B}{\pvec t \pvec a\pvec e}{\pvec v}.$$
This is the {\em desideratum} for the conclusion. \end{proof}

\section{Relativization and type-informative interpretations}
\label{relativization-section}

Our treatment of an ``informative'' interpretation of the quantifiers is based on the background uniform interpretation of the quantifiers (previous section) together with distinguished information predicates $\witness{x^\sigma}{\sigma}$ associated with each finite type $\sigma$. It is the presence of these (unary) information predicates that accounts for a ``type-informative'' interpretation of the quantifiers. 

\begin{definition}[Type-informative predicates and the language $\LangI{\SourceT}$]\label{def-type-informative-predicates} Let $\LangI{\SourceT}$ consist of the language $\Lang{\SourceT}$ extended with unary predicates $\witness{x^\sigma}{\sigma}$, one for each finite type $\sigma$ of $\Lang{\SourceT}$. We call these distinguished unary predicates $\witness{x^\sigma}{\sigma}$, {\em type-informative predicates}.
\end{definition}

We will consider different base interpretations of $\LangI{\SourceT}$.

\begin{example}[Precise interpretation] 
\label{precise-def} Assume that $\sigma$ is a type of $\Lang{\SourceT}$ which is also a witnessing type of $\Lang{\TargetT}$. Define the {\em precise interpretation} of $\witness{x^\sigma}{\sigma}$ as: 
\begin{itemize}
    \item $\pvec \tau_\sigma^+$ is $\sigma$,
    \item $\pvec \tau_\sigma^-$ is the empty tuple, and
    \item $\uwitness{x}{\sigma}{a}{}$ is $x=a$.
\end{itemize}
\end{example}

\begin{example}[Finite-set interpretation] 
\label{finite-set-def} Assume that $\sigma$ is a type of $\Lang{\SourceT}$ such that $\sigma^*$ is a witnessing type of $\Lang{\TargetT}$. Define the {\em finite-set interpretation} of $\witness{x^\sigma}{\sigma}$ as: 
\begin{itemize}
    \item $\pvec \tau_\sigma^+$ is $\sigma^*$,
    \item $\pvec \tau_\sigma^-$ is the empty tuple, and
    \item $\uwitness{x}{\sigma}{a}{}$ is $x \in_\sigma a$.
\end{itemize}
\end{example}

\begin{example}[Uniform interpretation]\label{uniform-def} The {\em uniform interpretation} of $\witness{x^\sigma}{\sigma}$ is defined as: 
\begin{itemize}
    \item $\pvec \tau_\sigma^+$ and $\pvec \tau_\sigma^-$ are empty tuples, and
    \item $\uwitness{x}{\sigma}{}{}$ is always true (for $x$ of type $\sigma$).
\end{itemize}
\end{example}

\begin{notation} \label{type-cases}
    If in the base interpretation $\uwitness{x}{\sigma}{\pvec a}{\pvec b}$ of the type-informative predicate $\witness{x^\sigma}{\sigma}$ the tuple $\pvec b$ is empty, we say that the type $\sigma$ {\em carries no negative witnesses}. Both the precise interpretations and the finite-set interpretations above illustrate this situation. If $\pvec a$ is the empty tuple, we say that $\sigma$ {\em carries no positive witnesses}. When both $\pvec a$ and $\pvec b$ are empty, we say that the type $\sigma$ is {\em witness-free}. Uniform interpretations, as above, are witness-free.
\end{notation}

\begin{definition}[Type-informative base interpretation of $\LangI{\SourceT}$ into $\TargetT$]\label{def-type-informative-base} A base interpretation of $\LangI{\SourceT}$ into $\Lang{\TargetT}$ is called a \emph{type-informative base interpretation of $\LangI{\SourceT}$ into $\TargetT$} if each sentence of the following two families is $\U$-interpretable in $\TargetT$:
\begin{itemize}
    \item[(i)] The sentences $\witness{c^\sigma}{\sigma}$, where $c^\sigma$ is a constant of $\Lang{\SourceT}$.
    \item[(ii)] The sentences $\forall x^\sigma \forall f^{\sigma \to \tau} (\witness{x}{\sigma} \wedge \witness{f}{\sigma \to \tau} \to \witness{fx}{\tau})$, where $\sigma, \tau$ are finite types of $\Lang{\SourceT}$.
\end{itemize}
\end{definition}

\begin{remark}\label{terms-source}
    It is clear that the sentences of the above two conditions imply $\witness{t}{\sigma}$, for every closed term $t$ of type $\sigma$. Since our framework assumes that, for each ground type $X$, there is a constant $e_X$ of type $X$ then, from these constants and lambda-abstraction (combinators), it is well known that we can recursively define a closed term for each type of the source theory. As a consequence, the type-informative predicates are non-empty. This observation is important because, in the next subsection, we define a relativization of formulas to the type-informative predicates and, for the relativization to operate properly, these predicates must be non-empty.
\end{remark}

\begin{remark} \label{criterium} Note that, by unwinding the definition of $\U$-interpretability, the sentence 
$$\forall x^\sigma \forall f^{\sigma \to \tau} (\witness{x}{\sigma} \wedge \witness{f}{\sigma \to \tau} \to \witness{fx}{\tau})$$ 
of $\LangI{\SourceT}$ is $\U$-interpretable in $\TargetT$ if, and only if, there are closed terms $\pvec q_{\sigma,\tau}$, $\pvec r_{\sigma,\tau}$  and $\pvec t_{\sigma,\tau}$ such that the following holds in $\TargetT$: for all $\pvec a, \pvec c,\pvec v$ of appropriate types and for all $x^\sigma$ and $f^{\sigma \to \tau}$ 
\[
    \forall \pvec b \in \pvec q_{\sigma,\tau} \pvec a \pvec c \pvec v\, 
    \uwitness{x}{\sigma}{\pvec a}{\pvec b}
    \wedge
    \forall \pvec d \in \pvec r_{\sigma,\tau} \pvec a \pvec c \pvec v\, 
    \uwitness{f}{\sigma \to \tau}{\pvec c}{\pvec d} 
    \to
    \uwitness{fx}{\tau}{\pvec t_{\sigma,\tau} \pvec a \pvec c}{\pvec v}.
\] 
\end{remark}

\begin{example}[Precise interpretation is type-informative] \label{all-precise-def} The precise interpretation (see Example \ref{precise-def}) for all the finite types of $\Lang{\SourceT}$ is a type-informative base interpretation of $\LangI{\SourceT}$ into $\TargetT$. To see this, we need to check that the sentences $(i)$ and $(ii)$ of Definition \ref{def-type-informative-base} are $\U$-interpretable in $\TargetT$. The sentences $(i)$ are $\U$-interpretable because one always has $\uwitness{c}{\sigma}{c}{}$, for constants $c$ of type $\sigma$ of $\Lang{\SourceT}$. By using the criterion of the previous remark, it is clear that the sentences $(ii)$ are also $\U$-interpretable in $\TargetT$: just put $\pvec q_{\sigma,\tau}$ and $\pvec r_{\sigma,\tau}$ the empty tuple and $\pvec t_{\sigma,\tau}$ given by $\pvec t_{\sigma,\tau} \pvec a \pvec c = \pvec c \pvec a$.
\end{example}

\begin{example} The uniform interpretation (see Example \ref{uniform-def}) for all the finite types of $\Lang{\SourceT}$ is obviously a type-informative base interpretation of $\LangI{\SourceT}$ into $\TargetT$ (there is nothing to show). 
\end{example}

\subsection{The main soundness theorem}
\label{main-soundness-section}

The concept of relativization, as in the definition below, is well known. It plays a fundamental role in what follows.

\begin{definition}[Relativization to type-informative predicates] \label{def-relativization} To each formula $A$ of $\Lang{\SourceT}$, we associate a formula $\rel{A}$ of $\LangI{\SourceT}$ as follows:
\eqleft{
\begin{array}{lcl}
	\rel{(A \wedge B)}
		& :\equiv & \rel{A} \wedge \rel{B} \\[1mm]
	\rel{(A \to B)} 
		& :\equiv & \rel{A} \to \rel{B}\\[1mm]
	\rel{(\forall x^\sigma A(x))} & :\equiv & \forall x^\sigma (\witness{x}{\sigma} \to \rel{A(x)}) \\[1mm]
	\rel{(\exists x^\sigma A(x))}& :\equiv & \exists x^\sigma (\witness{x}{\sigma} \wedge \rel{A(x)})
\end{array}	
}
where in the base case, when $A$ is an atomic predicate or $\bot$, $\rel{A}$ is $A$ itself. 
\end{definition}

The formula $\rel{A}$ is obtained by relativizing the quantifications of $\Lang{\SourceT}$ to corresponding type-informative predicates.

\begin{definition}[$\I$-interpretation of $\Lang{\SourceT}$ into $\Lang{\TargetT}$] \label{I-interpretation-def} Let be given a base interpretation of $\LangI{\SourceT}$ into $\Lang{\TargetT}$. For each formula $A$ of $\Lang{\SourceT}$, we define its $\I${\em -interpretation} $\iInter{A}{\pvec a}{\pvec b}$ into $\Lang{\TargetT}$ as follows:
\[ \iInter{A}{\pvec a}{\pvec b} :\equiv \uInter{\rel{A}}{\pvec a}{\pvec b} \] 
where $\uInter{\cdot}{}{}$ is the $\U$-interpretation of $\LangI{\SourceT}$ into $\Lang{\TargetT}$.
We call $\iInter{A}{\pvec a}{\pvec b}$ the $\I${\em -interpretation} (or {\em type-informative interpretation}) of $A$. 
\end{definition}

\newcommand{\bounded}[4]{#2 \lhd^{#1}_{#4} #3}

\begin{notation}\label{positive-negative-I-witnesses}
    If in the $\I$-interpretation of $A$, the formula $\iInter{A}{\pvec a}{\pvec b}$ is such that $\pvec a$ is the empty tuple, we say that \emph{$A$ has no positive $\I$-witnesses}, while if $\pvec b$ is empty, we say that \emph{$A$ has no negative $\I$-witnesses}. When $A$ has neither positive nor negative $\I$-witnesses, we say that $A$ is \emph{$\I$-witness-free}. 
\end{notation}

When $A$ is quantifier-free, $\iInter{A}{\pvec a}{\pvec b}$ is $\uInter{A}{\pvec a}{\pvec b}$. In particular, for the type-informative predicates, we have that $\iInter{\witness{x}{\sigma}}{\pvec a}{\pvec b}$ is $\uInter{\witness{x}{\sigma}}{\pvec a}{\pvec b}$.

\begin{proposition}[Inductive presentation of $\iInter{A}{\pvec a}{\pvec b}$] \label{I-interpretation-prop}     Assume that $A$ and $B$ have $\I$-interpretations $\iInter{A}{\pvec a}{\pvec b}$ and $\iInter{B}{\pvec c}{\pvec d}$, then: 
\eqleft{
\begin{array}{lcl}
	\iInter{A \wedge B}{\pvec a, \pvec c}{\pvec b, \pvec d} 
		& \mathrm{\,is\,} & \iInter{A}{\pvec a}{\pvec b} \wedge \iInter{B}{\pvec c}{\pvec d} \\[1mm]
	\iInter{A \to B}{\pvec f, \pvec g}{\pvec a, \pvec d} 
		& \mathrm{\,is\,} & \forall \pvec b \in \pvec g \pvec a \pvec d \iInter{A}{\pvec a}{\pvec b} \,\to 
				\iInter{B}{\pvec f \pvec a}{\pvec d} \\[1mm]
	\iInter{\forall x^\sigma A(x)}{\pvec f, \pvec g}{\pvec u, \pvec b} & \mathrm{\,is\,} & \forall x^\sigma (\forall \pvec v \in \pvec g \pvec u \pvec b \, \iInter{\witness{x}{\sigma}}{\pvec u}{\pvec v} \, \to \iInter{A(x)}{\pvec f \pvec u}{\pvec b}) \\[1mm]
	\iInter{\exists x^\sigma A(x)}{\pvec u, \pvec a}{\pvec V, \pvec B} & \mathrm{\,is\,} & \exists x^\sigma (\forall \pvec v \in \pvec V \, \iInter{\witness{x}{\sigma}}{\pvec u}{\pvec v} \wedge \forall \pvec b \in \pvec B \, \iInter{A(x)}{\pvec a}{\pvec b}).
\end{array}	
}
\end{proposition}

\begin{proof} A simple verification shows that clauses of $\I$-interpretation are as above. \end{proof}

\begin{remark}\label{fixing-base-2} As with the $\U$-interpretation (cf.\ Remark \ref{fixing-base}), the $\I$-interpretation should also be seen as a \emph{family of interpretations} parametrized by a base interpretation of $\LangI{\SourceT}$ into $\Lang{\TargetT}$. Given such a base interpretation, we can speak of the $\I$-interpretation of formulas of $\Lang{\SourceT}$ and the $\U$-interpretation of formulas of $\LangI{\SourceT}$. 
\end{remark}

\begin{definition}[Formulas $\I$-interpretable in $\TargetT$] Let be given a base interpretation of $\LangI{\SourceT}$ into $\Lang{\TargetT}$. A formula $A(\pvec x^{\pvec \sigma})$ of $\Lang{\SourceT}$ is said to be $\I$-\emph{interpretable} in $\TargetT$ if the sentence
$$\forall \pvec x^{\pvec \sigma} (\witness{\pvec x}{\pvec \sigma} \to \rel{(A(\pvec x))})$$
of $\LangI{\SourceT}$ is $\U$-interpretable in $\TargetT$.
\end{definition}

If we unwind the definitions, the $\I$-interpretability of $A(\pvec x)$ corresponds to the existence of sequences of closed terms $\pvec q$ and $\pvec t$ of $\WSorts^\omega_*$ such that $\TargetT$ proves the following sentence: for all $\pvec b$ and $\pvec c$ and for all $\pvec x^{\pvec \sigma}$, 
$$\mbox{if \,}\forall \pvec d \in \pvec q \pvec c \pvec b \, \iInter{\witness{\pvec x}{\pvec \sigma}}{\pvec c}{\pvec d} \mbox{,\, then } \iInter{A(\pvec x)}{\pvec t \pvec c}{\pvec b}.$$
Note that we are assuming that the free variables of $A$ are among $\pvec x$.

\begin{lemma} \label{lemma-universal}
    Let $A(\pvec x^{\pvec \sigma})$ be a formula (with its free variables as shown). Assume that the target theory $\TargetT$ proves
$$\forall \pvec b \,\iInter{A(\pvec x)}{\pvec t}{\pvec b},$$ 
for some tuple of closed witnessing terms $\pvec t$ of the target language ($\pvec t$ interprets $A(\pvec x)$ uniformly in $\pvec x$). Under these assumptions, the statement $A(\pvec x)$ is $\I$-interpretable.
\end{lemma}

\begin{proof}
    The $\I$-interpretation of $A(\pvec x)$ asks for tuples of closed terms $\pvec q$ and $\pvec s$ such that the target theory $\TargetT$ proves
$$\forall \pvec u, \pvec b \, \forall \pvec x^{\pvec \sigma}(\forall \pvec v \in \pvec q \pvec u \pvec b \iwitness{\pvec x}{\pvec \sigma}{\pvec u}{\pvec v} \to \iInter{A(\pvec x)}{\pvec s \pvec u}{\pvec b}).$$
We can take $\pvec s$ to be $\pvec \lambda \pvec u. \pvec t$ and $\pvec q$ any tuple of closed terms of the appropriate types.
\end{proof}

\begin{remark}
    By the above lemma, when we are trying to show that a formula $A(\pvec x^{\pvec \sigma})$ of the source language (with its free variables as shown) is $\I$-interpretable, we can ignore parameters and try to find a tuple of closed witnessing terms $\pvec t$ such that the target theory proves $\forall \pvec b \,\iInter{A(\pvec x)}{\pvec t}{\pvec b}$ (equivalently, that it proves $\forall \pvec x^{\pvec \sigma} \, \forall \pvec b \,\iInter{A(\pvec x)}{\pvec t}{\pvec b}$). Note, however, that with respect to rules parameters cannot be ignored (see Remark \ref{parameters_in_rules}).
\end{remark}

\begin{remark} \label{remark-universal}
    We now describe a variation of the above lemma. Let $\pvec \sigma$ be a tuple of types which carry no negative witnesses, and let $A(\pvec x^{\pvec \sigma})$ be a formula which has no positive $\I$-witnesses. Assume that the target $\TargetT$ proves 
$$\forall \pvec c, \pvec b \,\forall \pvec x^{\pvec \sigma} \,(\iwitness{\pvec x}{\pvec \sigma}{\pvec c}{} \to \iInter{A(\pvec x)}{}{\pvec b}).$$ 
In this situation, the statement $A(\pvec x)$ (or, equivalently, $\forall \pvec x^{\pvec \sigma} A(\pvec x)$) is $\I$-interpretable. In fact, the $\I$-interpretation of this formula is just what is assumed. An illustration of this is, of course, the consideration of universal principles in G\"odel's {\em dialectica} interpretation of arithmetic.
\end{remark}

\begin{definition}[$\TargetT$-sound type-informative base interpretation of $\SourceT$] \label{def-I-sound-base} A base interpretation of $\LangI{\SourceT}$ into $\Lang{\TargetT}$ is called a $\TargetT$-sound type-informative base interpretation of $\SourceT$ if
\begin{itemize}
    \item it is a type-informative base interpretation of $\LangI{\SourceT}$ into $\TargetT$ (cf.\ Definition \ref{def-type-informative-base}), and
    \item each non-logical axiom $A$ of $\SourceT$ is $\I$-interpretable in $\TargetT$.
\end{itemize}
\end{definition}

For convenience, the second item above includes the axioms for equality and the axioms for combinators. However, by the previous lemma, the $\I$-interpretability of the axioms for combinators is automatic. The $\I$-interpretability of the axioms for equality is also automatic: their $\I$-interpretability follows from the observation that the $\I$-interpretation of an axiom of equality is (the relativization of) an axiom of equality.

In the theorem below, we assume that the free variables of $\Gamma$ and $A$ are among $\pvec x$.

\begin{theorem}[Main Soundness Theorem] \label{main-soundness} Let be given a $\TargetT$-sound type-informative base interpretation of $\SourceT$. If 
\[
    \Gamma(\pvec x^{\pvec \sigma}) 
    \proves_{\SourceT} 
    A(\pvec x^{\pvec \sigma})
\]
then there are tuples of closed witnessing terms $\pvec q, \pvec s$ and $\pvec t$ such that  
$$\forall \pvec e \in \pvec q \pvec a \pvec c \pvec d \, \iInter{\witness{\pvec x}{\pvec \sigma}}{\pvec c}{\pvec e}, \forall \pvec b \in \pvec s \pvec a \pvec c \pvec d \iInter{\Gamma(\pvec x)}{\pvec a}{\pvec b} \proves_{\TargetT} \iInter{A(\pvec x)}{\pvec t \pvec a \pvec c}{\pvec d}.$$
\end{theorem}

\begin{proof} We consider the auxiliary theory $\SourceTI$, formulated in the language $\LangI{\SourceT}$, whose axioms are the sentences $(i)$ and $(ii)$ of Definition \ref{def-type-informative-base} and $(iii)$ the relativized sentences $\rel{A}$, for each non-logical axiom $A$ of $\SourceT$. Under these circumstances, it is a well-known fact that if $\Gamma(\pvec x^{\pvec \sigma}) \proves_{\SourceT} A(\pvec x^{\pvec \sigma})$, then 
$$\witness{\pvec x}{\pvec \sigma}, \rel{\Gamma(\pvec x)} \proves_{\SourceTI} \rel{A(\pvec x)}.$$ 
The axioms $(i)$, $(ii)$ and $(iii)$ of $\SourceTI$ are $\U$-interpretable by our assumption that we are given a $\TargetT$-sound type-informative base interpretation of $\SourceT$. Therefore, by the Uniform Soundness Theorem (see Theorem \ref{thm-uniform-soundness}), there are tuples of closed witnessing terms $\pvec q, \pvec s$ and $\pvec t$ such that  
$$\forall \pvec e \in \pvec q \pvec a \pvec c \pvec d \, \uInter{\witness{\pvec x}{\pvec \sigma}}{\pvec c}{\pvec e}, \forall \pvec b \in \pvec s \pvec a \pvec c \pvec d \uInter{\rel{\Gamma(\pvec x)}}{\pvec a}{\pvec b} \proves_{\TargetT} \uInter{\rel{A(\pvec x)}}{\pvec t \pvec a \pvec c}{\pvec d}.$$ This gives what we want. \end{proof}

\begin{figure}[h]
\begin{center}
\begin{tikzpicture}[
    >= {Triangle[length=5pt, width=4.5pt]},
    every path/.style={line width=0.9pt}
]
    \node (S) at (0,0) {\LARGE $\mathbf{S}$};
    \node (SI) at (4.5,0) {\LARGE $\mathbf{S}_{\mathcal{I}}$};
    \node (T) at (9.2,0) {\LARGE $\mathbf{T}$};

    \draw[->] (S) -- (SI)
        node[midway, above=7pt] {relativization}
        node[midway, below=6pt] {$(\,\cdot\,)^{\mathcal{I}}$};

    \draw[->] (SI) -- (T)
        node[midway, above=5pt, align=center] {uniform\\$\mathcal{U}$-interpretation}
        node[midway, below=6pt] {$\langle\,\cdot\,\rangle$};

    \draw[->] ([xshift=3pt, yshift=-6pt]S.south) to[bend right=24]
        node[midway, below=8pt] {$\mathcal{I}$-interpretation $|\cdot|$}
        ([xshift=-5pt, yshift=-4pt]T.south);
\end{tikzpicture}
\end{center}
\caption{The $\I$-interpretation}
\label{fig:I-interpretation}
\end{figure}

\medbreak
Figure \ref{fig:I-interpretation} illustrates the role of the auxiliary theory $\SourceTI$ in the Main Soundness Theorem.
 
\subsection{Some refinements}
\label{refinements}

In this subsection, we discuss some interesting particular cases of the general setting above, which tend to occur in practice when dealing with type-informative interpretations.

\medbreak
\noindent
(1)\,\, In many interesting situations, being of a certain type carries no negative witnesses (see \cite{DinisOliva(21)}). In such cases, the $\I$-interpretation of the quantifiers is much simpler:
\eqleft{
\begin{array}{lcl}
	\iInter{\forall x^\sigma A(x)}{\pvec f}{\pvec u, \pvec b} & \mathrm{\,is\,} & \forall x^\sigma (\iInter{\witness{x}{\sigma}}{\pvec u}{} \, \to \iInter{A(x)}{\pvec f \pvec u}{\pvec b}) \\[1mm]
	\iInter{\exists x^\sigma A(x)}{\pvec u, \pvec a}{\pvec B} & \mathrm{\,is\,} & \exists x^\sigma (\iInter{\witness{x}{\sigma}}{\pvec u}{} \wedge \forall \pvec b \in \pvec B \, \iInter{A(x)}{\pvec a}{\pvec b}).
\end{array}	
}
Therefore, in the case when the types of the variables $\pvec x$ of the Main Soundness Theorem \ref{main-soundness} carry no negative witnesses, we get the simpler (but important) conclusion that there are closed terms $\pvec t$ such that the target theory proves the following (assume $\Gamma(\pvec x)$ is the empty tuple, for the sake of simplicity): for all $\pvec c$ and $\pvec d$ and for all $\pvec x^{\pvec \sigma}$,
$$\mbox{if \,}\iInter{\witness{\pvec x}{\pvec \sigma}}{\pvec c}{} \mbox{,\, then } \iInter{A(\pvec x)}{\pvec t \pvec c}{\pvec d}.$$

\medbreak
\noindent
(2)\,\, Let $\pvec \sigma$, $\pvec \tau$ and $\pvec \rho$ be tuples of types that carry no negative witnesses. Suppose that $A(\pvec x^{\pvec \sigma}, \pvec w^{\pvec \rho})$ and $B(\pvec z^{\pvec \tau}, \pvec w^{\pvec \rho})$ are formulas which have no positive $\I$-witnesses. Finally, assume that the target theory $\TargetT$ proves the following: for $\pvec a$ and $\pvec w^{\pvec \rho}$ such that $\iwitness{\pvec w}{\pvec \rho}{\pvec a}{}$, $$\forall \pvec c, \pvec b \,\forall \pvec x^{\pvec \sigma} \,(\iwitness{\pvec x}{\pvec \sigma}{\pvec c}{} \to \iInter{A(\pvec x, \pvec w)}{}{\pvec b})\, \to\, \forall \pvec e, \pvec d \,\forall \pvec z^{\pvec \tau} \,(\iwitness{\pvec z}{\pvec \tau}{\pvec e}{} \to \iInter{B(\pvec z, \pvec w)}{}{\pvec d}).$$
Consider the rule 
\[
\begin{prooftree}
    \Gamma(\pvec w^{\pvec \rho}) \proves \forall \pvec x^{\pvec \sigma} A(\pvec x, \pvec w)
    \justifies
    \Gamma(\pvec w^{\pvec \rho}) \proves \forall \pvec z^{\pvec \tau} B(\pvec z, \pvec w)
\end{prooftree}
\]
where $\Gamma(\pvec w^{\pvec \rho})$ is a tuple of formulas that have no negative $\I$-witnesses. We claim that this rule is admissible for the $\I$-interpretation, i.e., that if the premise of the rule is $\I$-interpretable, then so is the conclusion. The $\I$-interpretation of the premise of the rule amounts to having the following in the target theory: 
$$\forall \pvec a, \pvec u, \pvec c, \pvec b\, \forall \pvec w^{\pvec \rho} \, (\iwitness{\pvec w}{\pvec \rho}{\pvec a}{} \wedge \iInter{\Gamma(\pvec w)}{\pvec u}{} \, \to \, \forall \pvec x^{\pvec \sigma} (\iwitness{\pvec x}{\pvec \sigma}{\pvec c}{} \to \iInter{A(\pvec x, \pvec w)}{}{\pvec b})).$$
By the assumption, we can obviously conclude $$\forall \pvec a, \pvec u, \pvec e, \pvec d\, \forall \pvec w^{\pvec \rho} \, (\iwitness{\pvec w}{\pvec \rho}{\pvec a}{} \wedge \iInter{\Gamma(\pvec w)}{\pvec u}{} \, \to \, \forall \pvec z^{\pvec \tau} (\iwitness{\pvec z}{\pvec \tau}{\pvec e}{} \to \iInter{B(\pvec z, \pvec w)}{}{\pvec d})).$$
This is the interpretation of the conclusion of the rule. The quantifier-free rule of extensionality due to G\"odel \cite{Goedel(21)} and Spector \cite{Spector(62)} is an example of an admissible rule as described above (as long as we have extensionality in the target theory). 

As mentioned in the introduction, in a theory of arithmetic and analysis, with ground types for the natural numbers as well as for the real numbers, one may consider a binary relation symbol $=_\RR$ whose base interpretation is thus: for $x$ and $y$ of the type of the reals, $\uInter{x=_\RR y}{}{k}$ is  $| x-y| \leq \frac{1}{k+1}$ ($k$ a natural number, of course). Let us suppose that $\uwitness{x}{\RR}{n}{}$ is $|x| \leq n$ (where $n$ is of type $\NN$). The implication $$\forall x^\RR \forall y^\RR (x \equiv_\RR y \to x =_\RR y)$$ is $\I$-interpretable. This is an application of Remark \ref{remark-universal}. However, the converse implication $$\forall x^\RR \forall y^\RR (x =_\RR y \to x \equiv_\RR y)$$ need not be $\I$-interpretable. A weaker replacement can, nevertheless, be useful. The following rule is admissible
\[
\begin{prooftree}
    \Gamma(\pvec w,x,y) \proves x =_\RR y
    \justifies
    \Gamma(\pvec w,x,y) \proves x \equiv_\RR y
\end{prooftree}
\]
where the types of $\pvec w$ carry no negative witnesses and $\Gamma(\pvec w,x,y)$ is a tuple of formulas whose $\I$-interpretations have no negative witnesses. This example is sound as long as the target theory is standard in the sense that two reals are identical if the absolute value of their difference is less than $\frac{1}{k+1}$, for every natural number $k$. 

An equivalent and perhaps more familiar way of stating the above rule is
\[
\begin{prooftree}
    \Gamma(\pvec w) \proves t[\pvec w] =_\RR q[\pvec w]
    \justifies
    \Gamma(\pvec w) \proves t[\pvec w] \equiv_\RR q[\pvec w]
\end{prooftree}
\]
where $t[\pvec w]$ and $q[\pvec w]$ are terms of type $\RR$ with free variables among $\pvec w$. The latter version is a consequence of the former by applying it to the premise $$\Gamma(\pvec w) \wedge x \equiv_\RR t[\pvec w] \wedge y \equiv_\RR q[\pvec w] \proves x =_\RR y.$$

\subsection{A technical note on existential quantification}

Proposition \ref{I-interpretation-prop} shows that the $\I$-interpretation can be given a direct definition. Regarding the interpretation of the existential quantifier, there are situations where a simpler interpretation works, namely:
\eqleft{
\begin{array}{lcl}
\iInter{\exists x^\sigma A(x)}{\pvec c, \pvec a}{\pvec d, \pvec b} & :\equiv & \exists x^\sigma(\iwitness{x}{\sigma}{\pvec c}{\pvec d}
        \wedge \iInter{A(x)}{\pvec a}{\pvec b}).
    \end{array}
}
As already mentioned, this simpler interpretation does not work in general because of the rule $\exists$-L. That notwithstanding, for certain choices of the base interpretation of $\LangI{\SourceT}$ and the target theory $\TargetT$, the clause for the existential quantifier can take the above simpler form. 
Next, we give a sufficient condition for the use of this simpler form.

\begin{proposition}\label{simpler} Suppose that $\sigma$ is a finite type whose base interpretation of $\witness{x}{\sigma}$ enjoys, for an arbitrary formula $A(z, \pvec x)$, the following property (in the target theory $\TargetT)$: For all $\pvec a$ and $\pvec c$, and finite sets $\pvec R, \pvec S$, if
$$\forall \pvec e \in \pvec R \, \forall \pvec b \in \pvec S \,\exists z^\sigma( \iwitness{z}{\sigma}{\pvec c}{\pvec e} \wedge \iInter{A(z, \pvec x)}{\pvec a}{\pvec b}),$$
then
$$\exists z^\sigma(\forall \pvec e \in \pvec R \,\iwitness{z}{\sigma}{\pvec c}{\pvec e} \wedge \forall \pvec b \in \pvec S \,\iInter{A(z, \pvec x)}{\pvec a}{\pvec b}).
$$ 
Then the Main Soundness Theorem \ref{main-soundness} holds with the simpler clause for the existential quantifier of type $\sigma$. 
\end{proposition}
\begin{proof} As commented, the simpler existential clause only finds an obstruction in the proof of the soundness theorem in the rule $\exists$-L. So, let us analyze in detail the situation of the $\exists$-L rule:
\[
\begin{prooftree}
    \Gamma(\pvec x^{\pvec \tau}), A(z^\sigma, \pvec x^{\pvec \tau}) \proves B(\pvec x^{\pvec \tau})
    \justifies
    \Gamma(\pvec x^{\pvec \tau}), \exists z^\sigma A(z^\sigma, \pvec x^{\pvec \tau}) \proves B(\pvec x^{\pvec \tau})
\end{prooftree}
\]
where $\sigma$ and $\pvec \tau$ are finite types and the variable $z$ is neither free in $\Gamma(\pvec x^{\pvec \tau})$ nor in $B(\pvec x^{\pvec \tau})$. \\[1mm]
For simplicity, assume that $\Gamma(\pvec x^{\pvec \tau})$ is empty. By induction hypothesis we have closed terms $\pvec s, \pvec u, \pvec v$ and $\pvec t$ such that, for all $\pvec h$, $\pvec c$, $\pvec a$, $\pvec d$\, and for all $\pvec x^{\pvec \tau}$ and $z^\sigma$, if 
$$\forall \pvec f \in \pvec u \pvec h \pvec c \pvec a \pvec d\, \iwitness{\pvec x}{\pvec \tau}{\pvec h}{\pvec f},\,\forall \pvec e \in \pvec v \pvec h \pvec c \pvec a \pvec d \, \iwitness{z}{\sigma}{\pvec c}{\pvec e} \mbox{\, and \,}\forall \pvec b \in \pvec s \pvec h \pvec c \pvec a \pvec d \iInter{A(z, \pvec x)}{\pvec a}{\pvec b}$$
then $\iInter{B(\pvec x)}{\pvec t \pvec h \pvec c \pvec a}{\pvec d}$. 
Therefore, for all $\pvec h$, $\pvec c$, $\pvec a$, $\pvec d$\, and for all $\pvec x^{\pvec \tau}$, if
$$\forall \pvec f \in \pvec u \pvec h \pvec c \pvec a \pvec d\, \iwitness{\pvec x}{\pvec \tau}{\pvec h}{\pvec f} \wedge
\exists z^\sigma (\forall \pvec e \in \pvec v \pvec h \pvec c \pvec a \pvec d \, \iwitness{z}{\sigma}{\pvec c}{\pvec e} \wedge
\forall \pvec b \in \pvec s \pvec h \pvec c \pvec a \pvec d \iInter{A(z, \pvec x)}{\pvec a}{\pvec b})$$
we have $\iInter{B(\pvec x)}{\pvec t \pvec h \pvec c \pvec a}{\pvec d}$. 
From our assumption we can then conclude that %
$$\forall \pvec f \in \pvec u \pvec h \pvec c \pvec a \pvec d\, \iwitness{\pvec x}{\pvec \tau}{\pvec h}{\pvec f},
\forall \pvec e \in \pvec v \pvec h \pvec c \pvec a \pvec d \, \forall \pvec b \in \pvec s \pvec h \pvec c \pvec a \pvec d \,\exists z^\sigma( \iwitness{z}{\sigma}{\pvec c}{\pvec e} \wedge \iInter{A(z, \pvec x)}{\pvec a}{\pvec b})$$
implies $\iInter{B(\pvec x)}{\pvec t \pvec h \pvec c \pvec a}{\pvec d}$. This is the realization for the conclusion of the rule $\exists$-L when the simpler clause for the existential quantifier is used for the finite type $\sigma$. 
\end{proof}

\smallbreak
\begin{example}\label{precise-quantifiers} There is a paradigmatic case for the application of the proposition above. It is when we have the precise interpretation as in Example \ref{precise-def} because we can obviously infer 
$$\exists z^\sigma( z = c \wedge \forall \pvec b \in \pvec S \,\iInter{A(z, \pvec x)}{\pvec a}{\pvec b})$$ 
from
$$\forall \pvec b \in \pvec S \,\exists z^\sigma( z=c \wedge \iInter{A(z, \pvec x)}{\pvec a}{\pvec b}).$$
Therefore, in this case, we may use the simpler interpretation of the existential quantifier so that $\iInter{\exists x^{\sigma} A(x)}{c, \pvec a}{\pvec b}$ is $\exists x \, (\iwitness{x}{\sigma}{c}{} \wedge \, \iInter{A(x)}{\pvec a}{\pvec b})$, i.e., $\exists x \, (x=c \wedge \, \iInter{A(x)}{\pvec a}{\pvec b})$. This simplifies to $\iInter{A(c)}{\pvec a}{\pvec b}$. For the universal quantifier, we have that $\iInter{\forall x^{\sigma} A(x)}{\pvec f}{c, \pvec b}$ is $\forall x\, (\iwitness{x}{\sigma}{c}{} \to \iInter{A(x)}{\pvec f c}{\pvec b})$. This clearly simplifies to $\iInter{A(c)}{\pvec f c}{\pvec b}$. In summary, in the precise interpretation, we have:
\eqleft{
\begin{array}{lcl}
	\iInter{\exists x^{\sigma} A(x)}{c, \pvec a}{\pvec b} & \equiv & \iInter{A(c)}{\pvec a}{\pvec b} \\[1mm]
	\iInter{\forall x^{\sigma} A(x)}{\pvec f}{c, \pvec b} & \equiv & \iInter{A(c)}{\pvec f c}{\pvec b}.
\end{array}	
}
These are the familiar clauses for the quantifiers used by G\"odel in his {\em dialectica} interpretation.
\end{example}

\section{Disjunction and the type of the Booleans}\label{disjunction-section}

When the Boolean type $\BB$ is available in the source theory $\SourceT$, we can define disjunction in the usual manner: 
\begin{equation} \label{def-disjunction}
    A\vee B :\equiv \exists i^\BB ((i\equiv_\BB 0 \to A) \wedge (i\equiv_\BB 1 \to B))
\end{equation}
using $0$ for True, and $1$ for False. Formally, we consider a source theory $\SourceT_\BB$ which has the Boolean type as a ground type, having two constants $0$ and $1$ and the axioms
\begin{itemize}
    \item[(${\rm D}_\BB$)] $\neg (0 \equiv_\BB 1)$
    \item[(${\rm E}_\BB$)] $D(0) \wedge D(1) \to \forall i^\BB D(i)$, \,where $D(i)$ is any formula of the source language.
\end{itemize}
The theory $\SourceT_\BB$ proves disjunction introduction and disjunction elimination. Disjunction introduction, in the form $A \to A\vee B$ and $B \to A\vee B$, follows readily from the axiom $\neg (0 \equiv_\BB 1)$. Let us now argue for disjunction elimination in the form: if $A \to C$ and $B \to C$ then $A \vee B \to C$. Assume $A \to C$ and $B \to C$. Let $D(i)$ be the formula  $$(i\equiv_\BB 0 \to A) \wedge (i\equiv_\BB 1 \to B) \to C.$$ Clearly, one has $D(0)$ and $D(1)$. By $({\rm E}_\BB)$, we get $A \vee B \to C$.

It is worth remarking the following immediate consequence of the definition of disjunction:

\begin{proposition}[Excluded middle for Booleans] \label{boolean_either_or} The theory $\SourceT_\BB$ proves the sentence $\forall i^\BB (i \equiv_\BB 0 \vee i \equiv_\BB 1)$. \end{proposition}

An $\I$-interpretation of the theory $\SourceT_\BB$ must be able to interpret both (${\rm D}_\BB$) and (${\rm E}_\BB$). By Remark \ref{remark-universal}, there is no issue concerning (${\rm D}_\BB$) because it does not require witnesses. On the other hand, the $\I$-interpretation of $({\rm E}_\BB)$ depends upon the base interpretation of the predicates $\witness{i}{\BB}$. We will investigate two possible base interpretations. One interprets $\witness{i}{\BB}$ precisely, as in Example \ref{precise-def}. Hence, $\BB$ is a witnessing type, $\pvec \tau_\BB^+$ is $\BB$, $\pvec \tau_\BB^-$ is the empty tuple and $\uwitness{i}{\BB}{j}{}$ is defined to be $i = j$. This is called the \emph{precise interpretation of disjunction}. The other option is to interpret the Booleans uniformly, as in Example \ref{uniform-def}: we let $\pvec \tau_\BB^+$ and $\pvec \tau_\BB^-$ be the empty tuples and define the type-informative predicate $\uwitness{i}{\BB}{}{}$ as always true (for Boolean elements $i$). This is called the \emph{uniform interpretation of disjunction}. Do notice that both interpretations validate $\I_\BB(0)$ and $\I_\BB(1)$. In the next two subsections, we discuss each interpretation in turn.

Before we go to the next subsection, let us discuss briefly the inclusion in the source language of if-then-else (or conditional) constants. This inclusion may sometimes be useful. These are the constants $\mathrm{Cond}_\sigma$ of type\, $\BB \to \sigma \to \sigma \to \sigma$ satisfying the axioms 
\begin{equation} \label{def-cond}
    \mathrm{Cond}_\sigma(0,x,y) \equiv_\sigma x 
\quad \quad \mbox{and} \quad \quad 
\mathrm{Cond}_\sigma(1,x,y) \equiv_\sigma y.
\end{equation}
The axioms for the conditional constants pose no problem for an $\I$-interpretation of the source theory because they do not require witnesses. However, the interpretability of the sentences $\I_{\BB \to \sigma \to \sigma \to \sigma}(\mathrm{Cond}_\sigma)$ has to be checked in each given concrete situation.

\subsection{Precise interpretation and if-then-else}\label{precise-if-else}

In this short subsection, we work with $\I$-interpretations based on the precise interpretation of the Booleans. It is quite easy to see that, in this case, the soundness of the $\I$-interpretation of $({\rm E}_\BB)$ requires terms $\pvec s$, $\pvec q$ and $\pvec t$ such that the following holds in the target theory $\TargetT$: for all Booleans $j$ and for all appropriate witnessing tuples $\pvec a_0$, $\pvec a_1$ and $\pvec b$, 
\[
\forall \pvec b' \in \pvec s j \pvec a_0 \pvec a_1 \pvec b \, 
    \iInter{D(0)}{\pvec a_0}{\pvec b'} 
\wedge 
\forall \pvec b' \in \pvec q j \pvec a_0 \pvec a_1 \pvec b \,
    \iInter{D(1)}{\pvec a_1}{\pvec b'} 
\to \iInter{D(j)}{\pvec t j \pvec a_0 \pvec a_1}{\pvec b}.
\]
We can put $\pvec s j \pvec a_0 \pvec a_1 \pvec b = \pvec q j \pvec a_0 \pvec a_1 \pvec b = \pvec \{ \pvec b \pvec \}$ and 
\[ 
\pvec t j \pvec a_0 \pvec a_1 = 
\left\{ 
\begin{array}{cl} 
    \pvec a_0 &\mbox{\;\;if \,$j = 0$} \\[1mm] 
    \pvec a_1 & \mbox{\;\;if \,$j = 1$.} 
\end{array} 
\right.
\]
Of course, these terms $\pvec t$ can be defined with the help of the conditional constants $\mathrm{Cond}$. In sum, for the precise interpretation of disjunction, we need to have the Boolean type $\BB$ in the target language (as a witnessing type) together with the conditional constants.

It is worthwhile to discuss the precise interpretation of disjunction in some detail. Following the definition (\ref{def-disjunction}) of disjunction $A \vee B$ in terms of quantification over Booleans, we have that $\iInter{A \vee B}{j, \pvec a, \pvec c}{\pvec F, \pvec G}$\, is 
$$\exists i^\BB \forall \pvec b \in \pvec F\, \forall \pvec d \in \pvec G \, (\iwitness{i}{\BB}{j}{} \wedge (i = 0 \to  \, \iInter{A}{\pvec a}{\pvec b}) \wedge (i = 1 \to  \, \iInter{B}{\pvec c}{\pvec d})).$$ 
By the definition of the precise base interpretation, we get 
$$(j=0 \to \forall \pvec b \in \pvec F \iInter{A}{\pvec a}{\pvec b}) \wedge (j=1 \to \forall \pvec d \in \pvec G \iInter{B}{\pvec c}{\pvec d}).$$
Note, however, that we are under the conditions of the application of Proposition \ref{simpler} (see, also, Example \ref{precise-quantifiers}). We can work with the simpler clause for the existential quantifier and obtain that
\eqleft{
\begin{array}{lcl}
	\iInter{A \vee B}{j, \pvec a, \pvec c}{\pvec b, \pvec d} 
		& \mbox{\,\,is\,\,} & (j = 0 \to \iInter{A}{\pvec a}{\pvec b}) \wedge (j = 1 \to \iInter{B}{\pvec c}{\pvec d}).
\end{array}
}
This is the familiar interpretation of disjunction in G\"odel's {\em dialectica} interpretation or in its Diller-Nahm variant.

\subsection{Uniform interpretation and monotonicity}
\label{uniform-booleans-sec}

We now consider $\I$-interpretations based on the uniform interpretation of the Booleans. The $\I$-interpretation of $({\rm E}_\BB)$ requires the existence of terms $\pvec s$, $\pvec q$ and $\pvec t$ such that the following holds in the target theory $\TargetT$: for all appropriate witnessing tuples $\pvec a_0$, $\pvec a_1$ and $\pvec b$, 
\[
\forall \pvec b' \in \pvec s \pvec a_0 \pvec a_1 \pvec b \, 
    \iInter{D(0)}{\pvec a_0}{\pvec b'} 
\wedge 
\forall \pvec b' \in \pvec q \pvec a_0 \pvec a_1 \pvec b \,
    \iInter{D(1)}{\pvec a_1}{\pvec b'} 
\to 
\forall i^\BB
\iInter{D(i)}{\pvec t \pvec a_0 \pvec a_1}{\pvec b}.
\]
Take $\pvec s \pvec a_0 \pvec a_1 \pvec b = \pvec q \pvec a_0 \pvec a_1 \pvec b = \pvec \{ \pvec b \pvec \}$. The problem lies with producing the terms $\pvec t$. In the above implication, the consequent is equivalent to the conjunction of $\iInter{D(0)}{\pvec t \pvec a_0 \pvec a_1}{\pvec b}$ and $\iInter{D(1)}{\pvec t \pvec a_0 \pvec a_1}{\pvec b}$. Therefore, it is sufficient to have terms $\pvec t$ such that both implications $\iInter{D(0)}{\pvec a_0}{\pvec b} \to \iInter{D(0)}{\pvec t \pvec a_0 \pvec a_1}{\pvec b}$ and $\iInter{D(1)}{\pvec a_1}{\pvec b} \to \iInter{D(1)}{\pvec t \pvec a_0 \pvec a_1}{\pvec b}$ hold.

For the uniform interpretation of disjunction, we do not need the Boolean type as a witnessing type in the target language (nor conditional constants). Instead, we need the following:

\begin{definition}[Joining of bounds] An atomic relation symbol $R(\pvec x)$ of $\LangI{\SourceT}$ is said to have the \emph{joining of bounds property in} $\TargetT$ if in the target theory $\TargetT$ we have a tuple of maps ${\pvec \vee}_R \colon \pvec \tau^+_{R} \times \pvec \tau^+_{R} \to \pvec \tau^+_{R}$ such that $\TargetT$ proves $\pvec a {\pvec \vee}_R \pvec b = \pvec b {\pvec \vee}_R \pvec a$ and the implication
    \eqleft{\mbox{if \,}\uInter{R(\pvec x)}{\pvec a}{\pvec c} {\,\,\,then\,\,\, } \uInter{R(\pvec x)}{\pvec a {\pvec \vee}_R \pvec b}{\pvec c}.}
A target theory $\TargetT$ is said to have \emph{joining of bounds for} $\LangI{\SourceT}$, if all atomic relation symbols of $\LangI{\SourceT}$ have the joining of bounds property in $\TargetT$.
\end{definition} 

With this notion, we can introduce a notion of joining of bounds for all formulas of $\LangI{\SourceT}$.  

\begin{definition}\label{joining-bounds} Assume $\TargetT$ has the \emph{joining of bounds for} $\LangI{\SourceT}$. We can extend the joining of bounds to all formulas of $\LangI{\SourceT}$ as follows. For an atomic formula $R(t_1, \ldots,t_n)$, where $R$ is an $n$-ary relation symbol and $t_1$, \ldots, $t_n$ are terms of appropriate types, we define ${\pvec \vee}_{R(t_1,\ldots, t_n)}$ to be ${\pvec \vee}_R$. Let be given formulas $A$ and $B$ with $\U$-interpretations $\uInter{A}{\pvec a}{\pvec b}$ and $\uInter{B}{\pvec c}{\pvec d}$. We discuss the four cases arising from the build-up of formulas, using the inductive presentation of the $\U$-interpretation given by Definition \ref{def-uniform-interpretation}.
\begin{enumerate}

\item[\mbox{\rm{(i)}}] {\em Conjunction.} We have that $\uInter{A \wedge B}{\pvec a, \pvec c}{\pvec b, \pvec d} $ is $\uInter{A}{\pvec a}{\pvec b} \wedge \uInter{B}{\pvec c}{\pvec d}$ and $\pvec \tau^+_{A\wedge B} = \pvec \tau^+_A, \pvec \tau^+_B$. We define:
$${\pvec \vee}_{A\wedge B}: \pvec \tau^+_{A \wedge B} \times \pvec \tau^+_{A \wedge B} \to \pvec \tau^+_{A \wedge B}$$
$$\pvec a, \pvec c; \pvec a', \pvec c' \,\,\,\leadsto\,\,\, \pvec a {\pvec \vee}_A \pvec a' , \pvec c {\pvec \vee}_B \pvec c'.$$
\item[\mbox{\rm{(ii)}}] {\em Implication.} We have that $\uInter{A \to B}{\pvec f, \pvec g}{\pvec a, \pvec d}$\, is $\forall \pvec b \in \pvec g \pvec a \pvec d \,\uInter{A}{\pvec a}{\pvec b} \,\,\to \uInter{B}{\pvec f \pvec a}{\pvec d}$. The tuple $\pvec \tau^+_{A\to B}$ is the concatenation of the tuple of the types of $\pvec f$ with the tuple of the types of $\pvec g$. We define:
$${\pvec \vee}_{A\to B}: \pvec \tau^+_{A \to B} \times \pvec \tau^+_{A \to B} \to \pvec \tau^+_{A \to B}$$
$$\pvec f, \pvec g; \pvec f', \pvec g' \,\,\,\leadsto\,\,\, \lambda \pvec a. (\pvec f\pvec a {\pvec \vee}_B \pvec f' \pvec a) , \lambda \pvec a,\pvec d. (\pvec g \pvec a \pvec d \cup \pvec g' \pvec a \pvec d).$$
\item[\mbox{\rm{(iii)}}] {\em Universal quantification}. We have that $\uInter{\forall x^\sigma A(x)}{\pvec a}{\pvec b}$ is $\forall x^\sigma \uInter{A(x)}{\pvec a}{\pvec b}$ and $\pvec \tau^+_{\forall x A(x)} = \pvec \tau^+_A$. We define: 
$${\pvec \vee}_{\forall x A(x)}: \pvec \tau^+_{\forall x A(x)} \times \pvec \tau^+_{\forall x A(x)} \to \pvec \tau^+_{\forall x A(x)}$$
$$\pvec a; \pvec a' \,\,\,\leadsto\,\,\, \pvec a {\pvec \vee}_A \pvec a'.$$
\item[\mbox{\rm{(iv)}}] {\em Existential quantification}. We have that $\uInter{\exists x^\sigma A(x)}{\pvec a}{\pvec B}$ is $\exists x^\sigma \forall \pvec b \in \pvec B \, \uInter{A(x)}{\pvec a}{\pvec b}$ and $\pvec \tau^+_{\exists x A(x)} = \pvec \tau^+_A$. We define:
$${\pvec \vee}_{\exists x A(x)}: \pvec \tau^+_{\exists x A(x)} \times \pvec \tau^+_{\exists x A(x)} \to \pvec \tau^+_{\exists x A(x)}$$
$$\pvec a; \pvec a' \,\,\,\leadsto\,\,\, \pvec a {\pvec \vee}_A \pvec a'.$$
\end{enumerate}
\end{definition}

The above definition generalizes and refines a notion in \cite{Ferreira(20)} applied to so-called end-star types. As in that case, we have the crucial monotonicity property:

\begin{lemma}[Monotonicity]\label{monotonicity-lemma} For any formula $A$ of $\LangI{\SourceT}$ and $\pvec a, \pvec b$ tuples having the types of the positive $\U$-witnesses of $A$, we have that the target theory $\TargetT$ proves $\pvec a {\pvec \vee}_A \pvec b = \pvec b {\pvec \vee}_A \pvec a$ and, for every tuple $\pvec c$ having the types of the negative $\U$-witnesses of $A$, we have the implication
\eqleft{\mbox{if \,}\uInter{A}{\pvec a}{\pvec c} {\,\,\,then\,\, } \uInter{A}{\pvec a {\pvec \vee}_A \pvec b}{\pvec c}.}
\end{lemma}

\begin{proof} The proof is by a simple induction on the complexity of $A$. The commutativity of ${\pvec \vee}_A$ is immediate provided that we accept a {\em modicum} of extensionality in the target theory. For the second claim, the atomic case is given and the cases of conjunction and quantifications are immediate. We discuss the case of implication. Let us use the notation (ii) above. Suppose that $\uInter{A \to B}{\pvec f, \pvec g}{\pvec a, \pvec d}$, i.e., $\forall \pvec b \in \pvec g \pvec a \pvec d \,\uInter{A}{\pvec a}{\pvec b} \,\to \uInter{B}{\pvec f \pvec a}{\pvec d}$. Let $\pvec f'$ and $\pvec g'$ be tuples with the types of $\pvec f$ and $\pvec g$, respectively. We must show $$\uInter{A \to B}{\pvec f, \pvec g \,{\pvec \vee}_{A \to B}  \pvec f', \pvec g'}{\pvec a, \pvec d}.$$ This is the same as showing that $$\forall \pvec b \in \pvec g \pvec a \pvec d \cup \pvec g' \pvec a \pvec d\,\uInter{A}{\pvec a}{\pvec b} \,\, \to \uInter{B}{\pvec f \pvec a \vee_B \pvec f' \pvec a}{\pvec d}.$$ Suppose that the antecedent above holds. {\em A fortiori}, $\forall \pvec b \in \pvec g \pvec a \pvec d \,\uInter{A}{\pvec a}{\pvec b}$. By the supposition, we infer $\uInter{B}{\pvec f \pvec a}{\pvec d}$. By induction hypothesis, $\uInter{B}{\pvec f \pvec a {\pvec \vee}_{B} \pvec f' \pvec a}{\pvec d}$. We are done. \end{proof}

\begin{remark}\label{monotonicity-remark}
    We are mainly interested in the $\I$-interpretations. The joining of bounds carries naturally to these interpretations, since $\iInter{A}{\pvec a}{\pvec c}$ is defined as $\uInter{\rel{A}}{\pvec a}{\pvec c}$. Therefore, given a formula $A$ of $\Lang{\SourceT}$ and given $\pvec a$ and $\pvec b$ tuples having the types of the positive $\I$-witnesses of $A$ then, for every tuple $\pvec c$ having the types of the negative $\I$-witnesses of $A$, we have the implication 
\eqleft{\mbox{if} \,\,\iInter{A}{\pvec a}{\pvec c} {\,\,\,\mbox{then}\,\, } \iInter{A}{\pvec a {\pvec \vee}_A \pvec b}{\pvec c}.}
We are abusing language here since ${\pvec \vee}_A$ should really be ${\pvec \vee}_{\rel A}$. We tolerate this abuse of language.
\end{remark}

\medbreak
With the joining of bounds and the monotonicity property in place, we can $\I$-interpret $({\rm E}_\BB)$ by putting $\pvec t \pvec a_0 \pvec a_1$ as $\pvec a_0 {\pvec \vee_D} \pvec a_1$. 

It is worth working out explicitly the uniform interpretation of $A\vee B$. It turns out that 
\begin{equation}\label{disjunction-clause}
\iInter{A \vee B}{\pvec a, \pvec c}{\pvec F, \pvec G} \mbox{\hspace{5mm}is\hspace{5mm}} \forall \pvec b \in \pvec F \, \iInter{A}{\pvec a}{\pvec b} \vee \forall \pvec d \in \pvec G \, \iInter{B}{\pvec c}{\pvec d}.
\end{equation}
The earliest version of an interpretation of disjunction like the one above occurred in \cite{FerreiraOliva(05)}, but with majorizability instead of finiteness. The above uniform interpretation of disjunction with finiteness is already found in, for instance, \cite{Ferreira(20)}. As discussed in these papers, it gives rise to a form of semi-intuitionistic logic.

Can the uniform interpretation of disjunction be simplified? We investigate the following clause:
\begin{equation}\label{simplified-disjunction-clause}
\iInter{A \vee B}{\pvec a, \pvec c}{\pvec b, \pvec d} \mbox{\hspace{5mm}is\hspace{5mm}} \iInter{A}{\pvec a}{\pvec b} \vee \iInter{B}{\pvec c}{\pvec d}.
\end{equation}
This simplified clause preserves interpretability with respect to the rules of disjunction introduction. However, an issue emerges with respect to the rule of disjunction elimination. Let us discuss it.

Suppose that we have interpretations of $A \to C$ and $B \to C$. This means that we have terms $\pvec t$, $\pvec q$, $\pvec s$ and $\pvec r$ such that the target theory proves that, for all $\pvec a$, $\pvec d$ and $\pvec v$,
\eqleft{\forall \pvec b \in \pvec q \pvec a \pvec v \iInter{A}{\pvec a}{\pvec b} \,\to \iInter{C}{\pvec t \pvec a}{\pvec v}\hspace{10mm}\mathrm{and}\hspace{10mm} \forall \pvec e \in \pvec r \pvec d \pvec v \iInter{B}{\pvec d}{\pvec e} \,\to \iInter{C}{\pvec s \pvec d}{\pvec v}.}
We must be able to interpret $A\vee B \to C$. With the new clause for disjunction, this means that we need to exhibit terms $\pvec p$, $\pvec f$ and $\pvec g$ such that the target theory proves that, for all $\pvec a$, $\pvec d$ and $\pvec v$, $$\forall \pvec b \in \pvec f \pvec a \pvec d \pvec v \forall \pvec e \in \pvec g \pvec a \pvec d \pvec v (\iInter{A}{\pvec a}{\pvec b} \vee \iInter{B}{\pvec d}{\pvec e}) \, \to \iInter{C}{\pvec p \pvec a \pvec d}{\pvec v}.$$ Naturally, we take terms such that $\pvec p \pvec a \pvec d = \pvec t \pvec a \vee_C \pvec s \pvec d$, $\pvec f \pvec a \pvec d \pvec v = \pvec q \pvec a \pvec v$ and $\pvec g \pvec a \pvec d \pvec v = \pvec r \pvec d \pvec v$. Clearly, we achieve our goal if 
\begin{equation*}
    \forall \pvec b \in \pvec q \pvec a \pvec v \forall \pvec e \in \pvec r \pvec d \pvec v (\iInter{A}{\pvec a}{\pvec b} \vee \iInter{B}{\pvec d}{\pvec e}) \,\to\, \forall \pvec b \in \pvec q \pvec a \pvec v \iInter{A}{\pvec a}{\pvec b} \vee \forall \pvec e \in \pvec r \pvec d \pvec v \iInter{B}{\pvec d}{\pvec e}. 
\end{equation*}
This implication is classically true. In fact, it is even intuitionistically true (with a limited use of arithmetical reasoning), because the above quantifications range over finite sets. Formally, this truth is argued in a target theory by induction on the number of elements of the finite sets. This is what paper \cite{BergBriseidSafarik(12)} does.

In conclusion, we can work with the simpler clause for disjunction (\ref{simplified-disjunction-clause}) as long as the target theory is able to prove implications of the form
\begin{equation*}
    \forall \pvec b \in \pvec R \forall \pvec e \in \pvec S (\iInter{A}{\pvec a}{\pvec b} \vee \iInter{B}{\pvec d}{\pvec e}) \,\to\, \forall \pvec b \in \pvec R \iInter{A}{\pvec a}{\pvec b} \vee \forall \pvec e \in \pvec S \iInter{B}{\pvec d}{\pvec e} 
\end{equation*}
where $\pvec a$, $\pvec d$, $\pvec R$ and $\pvec S$ are of appropriate types. Compare this with the discussion on the simplification of the clause for the existential quantifier in Subsection \ref{sec-uniform-interpretation} and in the proof of the Soundness Theorem \ref{thm-uniform-soundness}.

\section{The canonical interpretation of the function types}
\label{canonical-section}

For the remainder of the paper, we assume that the source theory $\SourceT$ contains the Booleans $\BB$, with axioms (${\rm D}_\BB$) and (${\rm E}_\BB$), and hence disjunction, as above. We also assume that the Booleans are either interpreted precisely or uniformly, as discussed in the previous section. In the latter case, we presuppose that we have a joining of bounds.

Given a base interpretation of $\witness{x}{X}$, for ground types $X$, there may be different choices for interpreting $\witness{h}{\rho \to \sigma}$, for function types $\rho \to \sigma$. In this section, we discuss a canonical way of extending base interpretations of ground types to all finite types. 

\begin{definition}[Canonical interpretation of $\I_{\rho \to \sigma}$] \label{canonical-def} Let a base interpretation of $\witness{x}{X}$ be given for each ground type $X$ of $\Lang{\SourceT}$. The {\em canonical base interpretation of} $\LangI{\SourceT}$ for the type-informative predicates is defined recursively as follows: For ground types the interpretation is as in the given base interpretation; for function types $\rho \to \sigma$ the definition of the interpretation of $\witness{h}{\rho \to \sigma}$ is 
\begin{equation} \label{def-canonical-inter}
    \uwitness{h}{\rho \to \sigma}{\pvec f, \pvec g}{\pvec a, \pvec d} \; :\equiv \; \forall x^\rho (\forall \pvec b \in \pvec g \pvec a \pvec d \uwitness{x}{\rho}{\pvec a}{\pvec b} \to \uwitness{hx}{\sigma}{\pvec f \pvec a}{\pvec d})
\end{equation}
where $\pvec f, \pvec g$ and $\pvec a, \pvec d$ have suitable types.
\end{definition}

This is the same as saying that $\uwitness{h}{\rho \to \sigma}{\pvec f, \pvec g}{\pvec a, \pvec d}$ is $\uInter{\forall x^\rho(\I_\rho(x) \to \I_\sigma(hx))}{\pvec f, \pvec g}{\pvec a, \pvec d}.$ Therefore, we can make the basic observation that the sentence 
\begin{equation} \label{eq-ii-equiv}
 \forall h^{\rho \to \sigma}(\I_{\rho \to \sigma}(h) \leftrightarrow \forall x^\rho(\I_\rho(x) \to \I_\sigma(hx))) 
\end{equation}
is $\U$-interpretable.  

\begin{proposition} \label{prop-canonical-type-informative} Let be given a base interpretation for all the ground types. Then, its canonical extension satisfies (ii) of Definition \ref{def-type-informative-base} and also satisfies (i) of that definition for the combinators.
\end{proposition}

\begin{proof} The sentences $(ii)$ of Definition \ref{def-type-informative-base} are a logical consequence of (\ref{eq-ii-equiv}), and hence are $\U$-interpretable, by the Uniform Soundness Theorem. Let us check that condition $(i)$ of Definition \ref{def-type-informative-base} holds for the combinators. Consider the combinator $\Sigma_{\rho, \tau, \sigma}$. Its type is $(\sigma \to \tau \to \rho)\to (\sigma \to \tau) \to \sigma \to \rho$. If we use thrice the basic observation above, we conclude that $\witness{\Sigma_{\rho, \tau, \sigma}}{}$ is $\U$-interpretable if, and only if, $$\forall x^{\sigma \to \tau \to \rho} \forall y^{\sigma \to \tau} \forall z^\sigma (\I_{\sigma\to\tau\to\rho}(x) \wedge \I_{\sigma \to \tau}(y) \wedge \I_\sigma (z) \to \I_\rho(xz(yz)))$$ is. The latter sentence is indeed $\U$-interpretable: just use (again) thrice the basic observation above. \end{proof}

We now consider base interpretations for the ground types that interpret the ground type $\BB$ of the Booleans uniformly. We investigate how a notion of joining of bounds for the ground types can be extended to all finite types, interpreted canonically. 

\begin{definition}[Joining of bounds for $\witness{h}{\rho \to \sigma}$] \label{join-bound-function-def} Suppose that we have operations 
$${\pvec \vee}_{\I_X} \colon \pvec \tau^+_{X} \times \pvec \tau^+_{X} \to \pvec \tau^+_{X}$$ 
for all ground types $X$. We extend this operation to the canonical type-informative predicates of function type as follows:
$$\pvec f, \pvec g {\pvec \vee}_{\I_{\rho \to \sigma}} \pvec f',\pvec g' \mbox{\,\,\,is defined as\,\,\,} \lambda \pvec a. (\pvec f \pvec a {\pvec \vee}_{\I_\sigma} \pvec f' \pvec a), \lambda \pvec a,\pvec d. (\pvec g \pvec a \pvec d \cup \pvec g' \pvec a \pvec d).$$    
\end{definition}

The following proposition is fundamental for dealing with canonical interpretations that interpret disjunction uniformly:

\begin{proposition} \label{join-bound-function-prop} If the type-informative predicates of ground types have the joining of bounds property, then the canonical type-informative predicates of function types (as given by the previous definition) also have the joining of bounds property.    
\end{proposition}
\begin{proof} By induction on the type. This is a direct consequence of Lemma \ref{monotonicity-lemma}, by observing that ${\pvec \vee}_{\I_{\rho \to \sigma}{(h)}}$ is in fact ${\pvec \vee}_{\forall x^\rho(\I_\rho(x) \to \I_\sigma(hx))}$. 
\end{proof}

The following result complements Proposition \ref{prop-canonical-type-informative}:

\begin{proposition} \label{cond-canonical-type-informative} Let be given a base interpretation for all the ground types and consider its canonical extension for the function types. If the conditional functionals are present in the source language, then the sentences $\witness{\mathrm{Cond}_\sigma}{\BB \to \sigma \to \sigma \to \sigma}$ are $\U$-interpretable.
\end{proposition}

\begin{proof}
    We claim that the following statement is $\U$-interpretable: $$\forall i^\BB \forall x^\sigma, y^\sigma (\witness{i}{\BB} \wedge \witness{x}{\sigma} \wedge \witness{y}{\sigma} \to \witness{\mathrm{Cond}_\sigma (i,x,y)}{\sigma}).$$ Note that we can deduce $\I(\mathrm{Cond}_\sigma)$ from the above statement and three applications of (\ref{eq-ii-equiv}). Now, by the claim, the $\U$-interpretability of (\ref{eq-ii-equiv}) and the fact that $\U$-interpretability is preserved by logical deduction (cf. the proof of the uniform soundness theorem), we conclude that $\I(\mathrm{Cond}_\sigma)$ is $\U$-interpretable.

It remains to prove the claim. Given our assumption on the interpretation of the Booleans, we need to consider two cases. 

Firstly, suppose that the Booleans are interpreted precisely. In this case, the claim amounts to showing that there are terms $\pvec s$, $\pvec q$ and $\pvec t$ such that the following holds in the target theory $\TargetT$: for all Booleans $j$ and for all appropriate witnessing tuples $\pvec a_0$, $\pvec a_1$ and $\pvec b$, and $x$, $y$, 
\[
\forall \pvec b' \in \pvec s j \pvec a_0 \pvec a_1 \pvec b \, 
    \uwitness{x}{\sigma}{\pvec a_0}{\pvec b'} 
\wedge 
\forall \pvec b' \in \pvec q j \pvec a_0 \pvec a_1 \pvec b \,
    \uwitness{y}{\sigma}{\pvec a_1}{\pvec b'} 
\to \uwitness{\mathrm{Cond}_\sigma(j,x,y)}{\sigma}{\pvec t j \pvec a_0 \pvec a_1}{\pvec b}.
\]
Of course, we can put $\pvec s j \pvec a_0 \pvec a_1 \pvec b = \pvec q j \pvec a_0 \pvec a_1 \pvec b = \pvec \{ \pvec b \pvec \}$ and 
\[ 
\pvec t j \pvec a_0 \pvec a_1 = 
\left\{ 
\begin{array}{cl} 
    \pvec a_0 &\mbox{\;\;if \,$j = 0$} \\[1mm] 
    \pvec a_1 & \mbox{\;\;if \,$j = 1$.} 
\end{array} 
\right.
\]

Secondly, we study the case where the Booleans are interpreted uniformly, and we have a joining of bounds. In this case, the claim amounts to showing that there are terms $\pvec s$, $\pvec q$ and $\pvec t$ such that the following holds in the target theory $\TargetT$: for all Booleans $j$ and for all appropriate witnessing tuples $\pvec a_0$, $\pvec a_1$ and $\pvec b$, and $x$, $y$, 
\[
\forall \pvec b' \in \pvec s \pvec a_0 \pvec a_1 \pvec b \, 
    \uwitness{x}{\sigma}{\pvec a_0}{\pvec b'} 
\wedge 
\forall \pvec b' \in \pvec q \pvec a_0 \pvec a_1 \pvec b \,
    \uwitness{y}{\sigma}{\pvec a_1}{\pvec b'} 
\to \forall i^\BB \uwitness{\mathrm{Cond}_\sigma(i,x,y)}{\sigma}{\pvec t \pvec a_0 \pvec a_1}{\pvec b}.
\]
We put $\pvec s \pvec a_0 \pvec a_1 \pvec b = \pvec q \pvec a_0 \pvec a_1 \pvec b = \pvec \{ \pvec b \pvec \}$ and $\pvec t \pvec a_0 \pvec a_1 = \pvec a_0 \vee_{\I_\sigma} \pvec a_1$.
\end{proof}

\begin{remark} By the proposition above and Proposition \ref{prop-canonical-type-informative}, in order for the canonical interpretation to be type-informative, it is only necessary to check, in each concrete situation, that condition $(i)$ of Definition \ref{def-type-informative-base} holds for the non-logical constants of the language of $\SourceT$ (constants other than the combinators and the conditional constants).
\end{remark}

Within the realm of arithmetic, canonical interpretations give rise to new models of G\"odel's theory $\GodelT$. This will be briefly discussed in Subsection \ref{canonical-precise-section} and in Subsection \ref{canonical-bounding-section}. We will also discuss, in the context of arithmetic, definitions of type-informative predicates for function types other than the canonical one.

\section{Interlude on arithmetic}\label{interlude}

A central example for the application of the framework of the previous sections concerns the theory of Heyting arithmetic in all finite types $\HAomega$ (and its extensions). The theory $\HAomega$ is usually defined with a unique ground type $\NN$, the type of the natural numbers (we call \emph{arithmetical types} the finite types built upon the ground type $\NN$).  However, given our discussion of disjunction in Section \ref{disjunction-section}, $\HAomega$ is presented in this paper with two ground types: $\NN$ and the type $\BB$ of the Booleans. This is a minor departure from the usual presentations of $\HAomega$. The literature discusses several variants of $\HAomega$, according to the manner in which they treat (intensional or, in our terminology, logical) equality. The reader can consult \cite{AvigadFeferman(98)} for a brief discussion on these matters. There is a purely technical reason why the issue of equality is raised: G\"odel's {\em dialectica} interpretation needs characteristic functions for the atomic formulas of the language in order to be able to interpret the rule of contraction for conjunction (therefore, primitive relations of equality need characteristic functions). However, as discussed in Section \ref{uniform-interpretation-section}, we depart from G\"odel's interpretation and use the Diller-Nahm variant instead. In this respect, the issue of logical equality vanishes (the issue of extensional equality is an altogether different matter). Using the framework of this paper, we present the theory $\HAomega$ with a logical equality sign at each finite type (this is like the neutral theory $\NHAomega$ of \cite{Troelstra(73)}). We also include in $\HAomega$ the conditional constants $\mathrm{Cond}_\sigma$ and their axioms.

The non-logical vocabulary of $\HAomega$ consists of the arithmetical constant $\Zero$ of type $\NN$, the constant $\Suc$ of type $\NN \to \NN$, as well as recursor constants $\Rec_\sigma$, of type $$\NN \to \sigma \to (\NN \to \sigma \to \sigma) \to \sigma$$ associated with each finite type $\sigma$ of the language of $\HAomega$. Do notice that the language of arithmetic $\Lang{\HAomega}$ has no relation symbols but the logical equality symbols $\equiv_\sigma$. 

The arithmetical constant zero $\Zero$ and the successor constant $\Suc$ are regulated by the two usual axioms that characterize an infinite Dedekind set, saying that $\Suc$ is injective and that $\Zero$ is not in the range of $\Suc$. The axioms for the recursors are as usual: 
\eqleft{
\begin{array}{lcl}
    \Rec_\sigma \Zero w f 
        & \equiv_\sigma & w \\[1mm]
    \Rec_\sigma (\Suc n) w f 
        & \equiv_\sigma & f n (\Rec_\sigma n w f)
\end{array}
}
giving rise to the definitions of the primitive recursive functionals of finite type in the sense of G\"odel \cite{Goedel(58)} (see also Hilbert \cite{Hilbert(26)}). Technically speaking, we should have constants for simultaneous recursion, as explained (for instance) in \cite{Kohlenbach(08)}. For the sake of simplicity, we use the above and maintain this register throughout the paper. The scheme of induction for all the formulas of the language completes the axiomatization of $\HAomega$. In the interpretations below, we will work with the rule of induction, instead of the axiom scheme. It is well known that, in the setting of Heyting arithmetic, we can even work with the parameter-free rule of induction. We will do so in the sequel.

We should also say a few words regarding the target theory $\TargetT$ where the interpretations of the theory $\HAomega$ take place. We focus on describing its terms (among which are the witnessing terms, with an associated term reduction calculus). For the concrete examples of the next sections, the language of the target theory is the language of $\HAomega$, endowed with the star types and the vocabulary of the star combinatory calculus, including recursors pertaining to the star types (do notice that the reduction calculus extends to the recursors, as shown in \cite{GFerreiraFirmino(24)}). We also include in the term calculus a witnessing constant $\mathrm{m}$ of type $\NN^* \to \NN$ whose intended meaning is that $\mathrm{m}F$, or $\mathrm{m}(F)$, is the maximum of the non-empty finite set $F$, where $F \subseteq \NN$ (the reduction calculus can also accommodate this constant). Regarding the target theories proper (i.e., their axioms), as commented in Section \ref{uniform-interpretation-section}, we do not need to be too precise about them.

That notwithstanding, a natural choice for effecting the verification of the interpretations of the source theory $\HAomega$ is to use the full set-theoretic structure $\Sws$, namely the typestructure defined as follows: $\mathcal{S}_\NN = \NN$, $\mathcal{S}_\BB = \{0,1\}$, 
\eqleft{\mathcal{S}_{\rho \to \sigma} = \{f: f\mbox{ is a function from }\mathcal{S}_\rho\mbox{ to }\mathcal{S}_\sigma \} \mbox{ and}} 
\eqleft{\mathcal{S}_{\sigma^*} = \{F \subseteq \mathcal{S}_\sigma: F\mbox{ is non-empty and finite} \}.} 
This amounts to taking the target theory $\TargetT$ as the set of truths in $\Sws$ (formulated in the language of the target theory). The restriction of $\Sws$ to the types built without the star-construct is denoted by $\Sw$. As is well known, the arithmetical part of $\Sw$ is a model of G\"odel's theory $\GodelT$.

Suppose that we have a sound type-informative base interpretation of $\HAomega$ (sound with respect to $\Sws$). For each arithmetical type $\sigma$ of the language of $\HAomega$, let
\eqleft{
\begin{array}{lcl}
    \mathcal{T_\sigma} 
        & := & \{ x \in \mathcal{S}_\sigma : \Sws \models \exists \pvec a \forall \pvec b \, \uInter{\witness{x}{\sigma}}{\pvec a}{\pvec b} \}. 
\end{array}
}
\begin{lemma}
    Given arithmetical types $\rho$ and $\sigma$ of the language of $\HAomega$, if $f \in \mathcal{T}_{\rho \to \sigma}$ and $x \in \mathcal{T}_\rho$, then $f(x) \in \mathcal{T}_\sigma$.
\end{lemma}

\begin{proof}
    This is a consequence of clause $(ii)$ of Definition \ref{def-type-informative-base}. See also Remark \ref{criterium}.
\end{proof}

Therefore, the typestructure $\mathcal{T}^\omega$ given by the family $(\mathcal{T}_\sigma)_\sigma$ is an applicative structure for the typed language of G\"odel's $\GodelT$. Moreover:

\begin{lemma}
    The combinators, zero, the successor function and the recursors of $\Sw$ are in $\mathcal{T}^\omega$.
\end{lemma}

\begin{proof}
    This is a consequence of clause $(i)$ of Definition \ref{def-type-informative-base}. 
\end{proof}

Notice that the above lemmas imply that $\mathcal{T}_\NN$ is $\NN$. In sum, we have the following important fact:

\begin{theorem}\label{models_of_T} 
$\mathcal{T}^\omega$ is a model of G\"odel's theory $\GodelT$.
\end{theorem}

When we apply the above theorem to concrete sound type-informative base interpretations of $\HAomega$, we obtain models of G\"odel's theory $\GodelT$. In the sequel, we will give examples of that. Note that the models $\mathcal{T}^\omega$ constructed  as above may not be extensional models. We believe, however, that it is also possible to consider (at least in certain situations) both extensional and hereditary versions of these models, but we will not discuss this issue in this paper.

\subsection{Star calculus and majorizability}\label{star-majorizability}

In 1973, William Howard introduced in \cite{Howard(73)} a notion of majorizability between functionals of the same arithmetical type. This majorizability relation was later modified by Marc Bezem \cite{Bezem(85)}, yielding Bezem's so-called \emph{strong majorizability relation}:
\eqleft{
\begin{array}{lcl}
		x \Bmaj_{\NN} y & :\equiv & x \leq y \\[1mm]
		f \Bmaj_{\sigma \to \tau} g & :\equiv & \forall x^\sigma, y^\sigma (x \Bmaj_\sigma y \to fx \Bmaj_\tau gy \wedge gx \Bmaj_\tau gy).
	\end{array}	
}
In contrast to Howard's notion, the strong majorizability notion is transitive and enjoys the following property: if $x\Bmaj_\sigma y$, then $y \Bmaj_\sigma y$. As Bezem reports in \cite{Bezem(85)}, the above definition of strong majorizability can be effected over any (arithmetical) typestructure. 

We need to extend the notion of strong majorizability to arithmetical star types, i.e., to the finite star types built upon the ground type $\NN$. This is done with the following extra clause:
\eqleft{
\begin{array}{lcl}
        x \Bmaj_{\sigma^*} y & :\equiv & \forall u \in x \,\exists v \in y \, (u \Bmaj_\sigma v) \wedge \forall v \in y \,(v \Bmaj_\sigma v).
	\end{array}	
}

It is easy to see that transitivity and the other property mentioned above still hold in the wider setting of the arithmetical star types.

\begin{definition}[Monotone elements] \label{star-monotone}
    Given an element $a$ of arithmetical star type $\sigma$, we say that $a$ is {\em monotone} if $a \Bmaj_\sigma a$.
\end{definition}

Note that an element $a$ of star type $\sigma^*$ is monotone if, and only if, each element of $a$ is monotone.

\begin{proposition}\label{star-monotone-prop}
    The combinators $\Pi_{\rho, \sigma}$ and $\Sigma_{\rho, \sigma, \gamma}$, and the three star operators $\mathfrak{s}_{\sigma}$, $\cup_{\sigma}$ and $\bigcup_{\sigma,\tau}$ are monotone.
\end{proposition}

\begin{proof} It is well known (and easy to check) that the combinators are monotone. The star operator $\mathfrak{s}_{\sigma}$ is of type $\sigma\to \sigma^*$ and, set-theoretically, it is $\mathfrak{s}_{\sigma}x = \{x\}$. We must show that if $x \Bmaj_\sigma a$, then $\{x\} \Bmaj_{\sigma^*} \{a\}$. This is evident because $x \Bmaj_\sigma a$ entails $a \Bmaj_\sigma a$. \\[1mm]
The union operator $\cup_{\sigma}$ is of type $\sigma^*\to (\sigma^*\to \sigma^*)$ and, set-theoretically, it is $\cup_{\sigma}xy = x \cup_\sigma y$. We must show that if $x \Bmaj_{\sigma^*} a$ and $y \Bmaj_{\sigma^*} b$, then $x \cup y \Bmaj_{\sigma^*} a \cup b$ and $a \cup b \Bmaj_{\sigma^*} a \cup b$. This is clear by definition. \\[1mm]
It remains to see that the indexed unions $\bigcup_{\sigma,\tau}$ are monotone. Such an indexed union has type $\sigma^*\to (\sigma\to \tau^*)\to \tau^*$ and, set-theoretically, it is $\bigcup_{\sigma,\tau} xf = \bigcup_{w\in x} fw$. \\[1mm]
Let us suppose that $x \Bmaj_{\sigma^*} a$ and $f \Bmaj_{\sigma \to \tau^*} g$. \\[1mm]
Suppose that $u \in \bigcup_{w\in x} fw$. We must show that there is $v \in \bigcup_{z\in a} gz$ such that $u\Bmaj_\tau v$. To see this, take $w_0 \in x$ with $u \in fw_0$. Since $x \Bmaj_{\sigma^*} a$, there is $z_0 \in a$ with $w_0 \Bmaj_\sigma z_0$. Therefore, $fw_0 \Bmaj_{\tau^*} gz_0$. Given that $u \in fw_0$, there is $v \in gz_0$ with $u \Bmaj_\tau v$. Clearly, $v \in \bigcup_{z\in a} gz$ and we are done. \\[1mm]
Finally, we must see that all the elements of $\bigcup_{z\in a} gz$ are monotone. Take $v$ and $z$ with $v \in gz$ and $z \in a$. Since $a$ is monotone, so are its elements. Therefore $z$ is monotone. We get that $gz$ is monotone, because $g$ is also monotone. Again, the elements of $gz$ are monotone. Hence, $v$ is monotone.
\end{proof}

\begin{proposition}
    The constant $\mathrm{m}$ of type $\NN^* \to \NN$ is monotone.
\end{proposition}

\begin{proof}
    We assume that the constant $\mathrm{m}$ enjoys (in the target theory) the basic properties of the maximum function (on finite, non-empty, sets). Accordingly, we have $\forall F^{\NN^*}, G^{\NN^*} (F \Bmaj_{\NN^*} G \to \mathrm{m}(F) \leq \mathrm{m}(G))$. This is clear because the condition $F \Bmaj_{\NN^*} G$ is equivalent to $\forall n\in F \exists m \in G \, (n\leq m)$.
\end{proof}

The maximum functional applied to a pair of elements of the same arithmetical type is well known from the literature: $\max_\NN (n,m)$ is the maximum of the numbers $n$ and $m$; if $f$ and $g$ are arithmetical functionals of function type $\sigma \to \tau$, then 

\begin{equation}\label{max-functional}
    \max_{\sigma \to \tau}(f,g) :\equiv \lambda x^\sigma.\max_\tau (fx,gx).
\end{equation} 

The definition of the maximum functional can be extended to arithmetical star types. If $x$ and $y$ are of type $\sigma^*$, then

\begin{equation}\label{max-functional_star}
    \max_{\sigma^*}(x,y) :\equiv x \cup_\sigma y.
\end{equation} 

\noindent
We assume that our target theory has enough extensionality so that it can show the equalities $\max_{\sigma \to \tau}(f,g) = \max_{\sigma \to \tau}(g,f)$ and $x \cup_\sigma y = y \cup_\sigma x$.

\begin{lemma} \label{lemma-max} Let $\sigma$ be an arithmetical type. Suppose that $x$, $y$, $a$ and $b$ are elements of type $\sigma$. If $x \Bmaj_\sigma a$ and $y \Bmaj_\sigma b$ then $x \Bmaj_\sigma \max_{\sigma} (a,b)$,  $y \Bmaj_\sigma \max_{\sigma} (a,b)$ and $\max_\sigma (x,y) \Bmaj_\sigma \max_\sigma(a,b)$. \end{lemma}

The above lemma has an easy proof by induction on the complexity of the type $\sigma$ (see Subsection 3.6 of \cite{Kohlenbach(08)}; the arguments for the extra star case are straightforward). As a simple consequence, $\max_\sigma$ is monotone and if $a$ and $b$ are monotone of type $\sigma$, so is $\max_\sigma (a,b)$ and, moreover, $a \Bmaj_\sigma \max_\sigma(a,b)$ and $b \Bmaj_\sigma \max_\sigma(a,b)$.

\section{Precise interpretations of arithmetic}
\label{precise-section}

In this section, our $\I$-interpretations are defined precisely at the ground type $\NN$, as in Example \ref{precise-def}. So, $\uwitness{x}{\NN}{n}{}$ is $x = n$. We also interpret the ground Boolean type precisely (a moment of thought shows that it cannot be interpreted in a uniform manner because the ground type $\NN$ does not have a notion of joining of bounds when interpreted precisely). The precise interpretation at the ground type $\NN$ does not determine the type-informative interpretations of the function types. One choice is to extend the precise interpretation (at the ground type) to the function types {\em precisely}, as discussed in Example \ref{all-precise-def}. 
\begin{definition}[$\Ipp$-interpretation] The $\Ipp$-interpretation is the base interpretation of arithmetic that defines the type-informative predicates precisely at every finite type. 
\end{definition}

This is nothing but the Diller-Nahm variant of the {\em dialectica} interpretation (see the comments of Example \ref{precise-quantifiers} and the second part of Subsection \ref{precise-if-else}). It is a well-known interpretation and we will not discuss it further. 

\subsection{Canonical upon the precise interpretation \texorpdfstring{$\Icp$}{Icp}}
\label{canonical-precise-section}

An alternative choice is to have the canonical interpretation at the function types, as discussed in Section \ref{canonical-section}. This choice is a novelty and we discuss it now.
\begin{definition}[$\Icp$-interpretation] The $\Icp$-interpretation is the base interpretation of arithmetic that defines the type-informative predicates as follows: it is precise at the ground types $\NN$ and $\BB$ and it is extended canonically at the function types.
\end{definition}

\begin{notation}
    The $\U$-interpretation fixed by the definition above is called the $\Ucp$-interpretation (see Remark \ref{fixing-base-2}).
\end{notation}

We want to show that the base interpretation of the $\Icp$-interpretation is sound  and type-informative. First, we must show that it is type-informative, i.e., that we have clauses $(i)$ and $(ii)$ of Definition \ref{def-type-informative-base}. By both Proposition \ref{prop-canonical-type-informative} and Proposition \ref{cond-canonical-type-informative}, clause $(ii)$ is automatic  and, for clause $(i)$, we need only check the arithmetical constants. So, we must check that the sentences $\witness{\Zero}{\NN}$, $\witness{\Suc}{\NN \to \NN}$ and $\witness{\Rec_\sigma}{\NN \to \sigma \to (\NN \to \sigma \to \sigma) \to \sigma}$ are $\Ucp$-interpretable in the target theory. 

It is illustrative to exhibit the $\Ucp$-interpretations of the type-informative predicates at the low types $\NN$, $\NN \to \NN$ and $(\NN \to \NN) \to \NN$ (these are known as type 0, type 1 and type 2, respectively). We have that
\eqleft{
\begin{array}{lcl}
    \uwitness{x}{\NN}{n}{} 
        & \mathrm{is} & x=n \\[1mm]
   \uwitness{f}{\NN \to \NN}{g}{k} 
        & \mathrm{is} & fk = gk \\[1mm]
   \uwitness{\Phi}{(\NN \to \NN) \to \NN}{\Psi, \mu}{h} 
        & \mathrm{is} & \forall f (\forall k \in \mu h (fk = hk) \to \Phi f = \Psi h)
\end{array}
}
where $n :\NN$, $g,h : \NN \to \NN$, $\mu: (\NN\to \NN) \to \NN^*$ and $\Psi: (\NN\to \NN) \to \NN$. It is clear that one has $\uwitness{\Zero}{\NN}{0}{}$ and $\forall k \uwitness{\Suc}{\NN \to \NN}{S}{k}$, where $S$ is the successor function. The following proposition deals with the recursors.

\begin{proposition} \label{rec_precise} The sentence $\witness{\Rec_\sigma}{\NN \to \sigma \to (\NN \to \sigma \to \sigma) \to \sigma}$ is $\Ucp$-interpretable, for each type $\sigma$. 
\end{proposition}
\begin{proof} By the definition of the canonical interpretation of function types (Def. \ref{canonical-def}), the recursor $\Rec_\sigma$ enjoys the property of the proposition if, and only if, the sentence 
$$\forall n^\NN \forall x^\sigma \forall f^{\NN \to \sigma \to \sigma} (\witness{n}{\NN} \wedge \witness{x}{\sigma} \wedge \witness{f}{\NN \to \sigma \to \sigma} \to \witness{\Rec_\sigma n x f}{\sigma})$$ is $\Ucp$-interpretable. Our aim is, therefore, to show that the above sentence is $\Ucp$-interpretable. By definition,
\[
\uwitness{f}{\NN \to \sigma \to \sigma}{\pvec \phi, \pvec B}{k, \pvec a, \pvec b} \mbox{\,\, is \,\,}\forall y^\sigma (\forall \pvec b' \in \pvec B k \pvec a \pvec b \uwitness{y}{\sigma}{\pvec a}{\pvec b'} \to \uwitness{f k y}{\sigma}{\pvec \phi k \pvec a}{\pvec b}). 
\]
With this notation, it is not difficult to see that we must obtain closed terms $\pvec \Psi$, $\pvec B^*$, $K$, $\pvec A$ and $\pvec B^\dagger$ such that the target theory proves the following: for all appropriate $n$, $\pvec a$, $\pvec \phi$, $\pvec B$, $\pvec b$, and for all $x^\sigma$ and $f^{\NN \to \sigma \to \sigma}$, if 
\eqleft{\forall \pvec b' \in \pvec B^* n \pvec a \pvec \phi \pvec B \pvec b \uwitness{x}{\sigma}{\pvec a}{\pvec b'} \mbox{\,\,\,\,\,\,and}} 
\eqleft{\forall k \in K n \pvec a \pvec \phi \pvec B \pvec b \; \forall \pvec a' \in \pvec A n \pvec a \pvec \phi \pvec B \pvec b \; \forall \pvec b'' \in \pvec B^\dagger n \pvec a \pvec \phi \pvec B \pvec b \, \uwitness{f}{\NN \to \sigma \to \sigma}{\pvec \phi, \pvec B}{k, \pvec a', \pvec b''}}
then $\uwitness{\Rec_\sigma n x f}{\sigma}{\pvec \Psi n \pvec a \pvec \phi \pvec B}{\pvec b}$. \\
Given $n$, $\pvec a$, $\pvec \phi$, $\pvec B$, $\pvec b$, we define, using recursors, two sequences $\pvec a_i$ and $\pvec B_i$ inductively as follows:
\[
\begin{array}{rclcrcl}
\pvec a_0 & = & \pvec a & \hspace{8mm} & \pvec B_n & = & \{\pvec b\} \\[1mm]
\pvec a_{i+1} & = & \pvec \phi i\pvec a_i & \hspace{8mm} & \pvec B_i & = & \bigcup_{\pvec b'\in \pvec B_{i+1}} \pvec B i \pvec a_i \pvec b', 
\end{array}
\]
for $i$ with $0\leq i < n$. Now, we take $\pvec \Psi$, $\pvec B^*$, $K$, $\pvec A$ and $\pvec B^\dagger$ such that
\[
\begin{array}{rcl}
     \pvec \Psi n \pvec a \pvec \phi \pvec B & = & \pvec a_n \\[1mm]
     \pvec B^* n \pvec a \pvec \phi \pvec B \pvec b & = & \pvec B_0 \\[1mm]
     K n \pvec a \pvec \phi \pvec B \pvec b & = & \{ 0, 1, \ldots, n \} \\[1mm]
     \pvec A n \pvec a \pvec \phi \pvec B \pvec b & = & \{ \pvec a_0, \pvec a_1, \ldots, \pvec a_n \} \\[1mm]
     \pvec B^\dagger n \pvec a \pvec \phi \pvec B \pvec b & = & \bigcup_{j\leq n} \pvec B_j.
\end{array}
\]
Let us see that these terms do the job. Given $n, \pvec a, \pvec \phi, \pvec B$ and $\pvec b$, fix $x$ and $f$ and assume
\begin{equation*}
    \forall \pvec b' \in \pvec B_0 \uwitness{x}{\sigma}{\pvec a}{\pvec b'}
\end{equation*}
and, for each $k < n$ and for all $\pvec b'' \in \bigcup_{j\leq n} \pvec B_j$,
\begin{equation} \label{eq-rec-hash}
     \forall y^\sigma (\forall \pvec b' \in \pvec B k \pvec a_k \pvec b'' \uwitness{y}{\sigma}{\pvec a_k}{\pvec b'} \to \uwitness{f k y}{\sigma}{\pvec a_{k+1}}{\pvec b''}). 
\end{equation}
We claim that for all $k$ such that $0 \leq k \leq n$, we have $\forall \pvec b' \in \pvec B_k \uwitness{\Rec_\sigma k x f}{\sigma}{\pvec a_k}{\pvec b'}$. In particular, when $k=n$, we get $\uwitness{\Rec_\sigma n x f}{\sigma}{\pvec a_n}{\pvec b}$. This is what we want. We argue by induction on $0 \leq k \leq n$. The base case is given. Let $0\leq k < n$. Take $\pvec b'' \in \pvec B_{k+1}$. We must show $\uwitness{\Rec_\sigma (k+1) x f}{\sigma}{\pvec a_{k+1}}{\pvec b''}$, that is, $\uwitness{fk(\Rec_\sigma k x f)}{\sigma}{\pvec a_{k+1}}{\pvec b''}$. By considering in (\ref{eq-rec-hash}) above $y$ to be $\Rec_\sigma k x f$, this follows if we have $\forall \pvec b' \in \pvec B k \pvec a_k \pvec b'' \uwitness{\Rec_\sigma k x f}{\sigma}{\pvec a_k}{\pvec b'}$. Given that $\pvec B k \pvec a_k \pvec b'' \subseteq \pvec B_k$, we are done by induction hypothesis. \end{proof}

The next theorem guarantees that the Main Soundness Theorem (cf.\ Theorem \ref{main-soundness}) applies to $\HAomega$ for the $\Icp$-interpretation.

\begin{theorem} 
    The base interpretation of the $\Icp$-interpretation is sound and type-informative.
\end{theorem}

\begin{proof}
    We have to see that all the non-logical axioms of $\HAomega$ are $\Icp$-interpretable. Only the induction axioms need a comment. As mentioned in Section \ref{interlude}, we work with the parameter-free rule of induction. The preservation of interpretability is done with the usual argument with recursors (e.g., \cite[Section 8.2]{Kohlenbach(08)}).
\end{proof}

\begin{remark}\label{parameters_in_rules}
    The absence of parameters in the rule of induction simplifies matters in the case of the canonical interpretations. In such interpretations, the type of parameters may carry negative witnesses and, therefore, when interpreting the conclusion of the rule, the negative witnesses of the parameters accumulate backwards. This accumulation can be dealt with, but matters become more difficult. Without parameters, the situation (as commented in the proof above) is as usual.
\end{remark}

In contrast with G\"odel's {\em dialectica} interpretation, we have the following:

\begin{proposition} \label{prop-cont-2}
    The pointwise continuity of type 2 functionals 
    \[
    \forall \Phi^{(\NN \to \NN) \to \NN} \forall f^{\NN \to \NN} \exists N \forall g^{\NN \to \NN} (\forall n \leq N (f n \equiv_{\NN} g n) \to \Phi f \equiv_{\NN} \Phi g)
    \]
    is $\Icp$-interpretable.
\end{proposition}

\begin{proof} Recall that 
\[
    \uwitness{\Phi}{(\NN \to \NN) \to \NN}{\Psi, \mu}{h} \mbox{\;\;is\;\;} \forall f' (\forall n \in \mu h (f' n = h n) \to \Phi f' = \Psi h)    
\]
and 
\[
    \uwitness{f}{\NN \to \NN}{\alpha}{n} \mbox{\;\;is\;\;} f n = \alpha n \mbox{\;\;\;and} \quad \uwitness{g}{\NN \to \NN}{\beta}{n} \mbox{\;\;is\;\;} g n = \beta n.
\]

We must find terms $t$, $s$, $r$ and $q$ such that, for all $\Psi$, $\mu$, $\alpha$ and $\beta$ and for all $\Phi$ and $f$, if 
$$\forall h \in t\Psi \mu\alpha \beta \,\forall f' (\forall n \in \mu h(f'n = hn) \to \Phi f' = \Psi h)$$ 
and 
$$\forall n \in s\Psi\mu\alpha\beta (f n = \alpha n)$$ then $$\forall g\,(\forall n \in r \Psi \mu\alpha \beta \, (gn=\beta n) \wedge \forall n \leq q\Psi \mu \alpha (fn=gn) \to \Phi f = \Phi g).$$
Note that, in the previous computations, we are considering the formula $$\forall n \leq N (f n \equiv_{\NN} g n) \to \Phi f \equiv_{\NN} \Phi g$$ as a quantifier-free formula (as we clearly can).

Take $t\Psi\mu\alpha\beta = \{\alpha \}$, $s\Psi \mu \alpha \beta = \mu \alpha$, $r\Psi \mu \alpha \beta = \mu \alpha$ and $q\Psi \mu \alpha = \mathrm{m}(\mu\alpha)$, where we recall that $\mathrm{m}$ is the maximum functional of Section \ref{interlude}. Accordingly, we must show that, for all $g$, if
\begin{itemize}
    \item[$(i)$] $\forall f' (\forall n \in \mu \alpha \, (f' n = \alpha n) \to \Phi f' = \Psi \alpha)$,
    \item[$(ii)$] $\forall n \in \mu \alpha \,(f n = \alpha n)$,
    \item[$(iii)$] $\forall n \in \mu \alpha \,(g n = \beta n)$, and
    \item[$(iv)$] $\forall n \leq \mathrm{m}(\mu \alpha) (f n = g n)$,
\end{itemize}
then $\Phi f = \Phi g$. \\[1mm]
By $(i)$ with $f' = f$ and using $(ii)$, we have $\Phi f = \Psi \alpha$. By $(iv)$, we have that $\forall n \in \mu\alpha \,(fn=gn)$. Together with $(ii)$, we get $\forall n \in \mu\alpha \,(gn = \alpha n)$. Therefore, by $(i)$ with $f' = g$, we get $\Phi g =\Psi \alpha$.  Hence, $\Phi f = \Psi \alpha = \Phi g$. Note that $\beta$ and assumption $(iii)$ are not needed. \end{proof}

Since the continuity of type-2 functionals implies their extensionality, it follows that:

\begin{corollary} \label{prop-ext-2}
    The extensionality principle for type 2 functionals 
    \[
    \forall \Phi^{(\NN \to \NN) \to \NN} \forall f, g^{\NN \to \NN} (\forall n^\NN (f n \equiv_\NN g n) \to \Phi f \equiv_\NN \Phi g)
    \]
    is $\Icp$-interpretable.
\end{corollary}

As discussed in Section \ref{interlude}, we can consider the following typestructure:

\begin{definition}[The typestructure $\mathcal{C^=}$] The typestructure $\mathcal{C}^=$ is defined as $(\mathcal{C^=_\sigma})_\sigma$, with
\eqleft{
\begin{array}{lcl}
    \mathcal{C^=_\sigma} 
        & := & \{ x \in \mathcal{S}_\sigma : \Sws \models \exists \pvec a \forall \pvec b \, \uInter{\witness{x}{\sigma}}{\pvec a}{\pvec b} \} 
\end{array}
}
for each arithmetical type $\sigma$, where $\langle \cdot \rangle$ is the $\Ucp$-interpretation. 
\end{definition}

The following is a consequence of Theorem \ref{models_of_T}:

\begin{proposition} $\mathcal{C}^=$ is a model of G\"odel's theory $\GodelT$.
\end{proposition}

What are the elements of $\mathcal{C}^=_{\NN \to \NN}$? If we look at the explicit computation of the $\Ucp$-interpretation of the predicate $\I_{\NN \to \NN}(f)$, it is immediate that $\forall k \, \uwitness{f}{\NN \to \NN}{f}{k}$. So, $\mathcal{C}^=_{\NN \to \NN}$ is $\mathcal{S}_{\NN \to \NN}$, i.e., it consists of all the set-theoretic functions from $\NN$ to $\NN$. 

What are the elements of $\mathcal{C}^=_{(\NN \to \NN) \to \NN}$? This is a more interesting question. According to the $\Ucp$-interpretation of the predicate $\I_{(\NN \to \NN) \to \NN}(\Phi)$, a functional $\Phi$ of $\mathcal{S}_{(\NN \to \NN) \to \NN}$ is in $\mathcal{C}^=_{(\NN \to \NN) \to \NN}$ if, and only if, $$\exists \mu, \Psi \forall h \forall f (\forall k \in \mu h (fk = hk) \to \Phi f = \Psi h).$$ This is equivalent to saying that $$\exists \mu \forall h \forall f (\forall k \in \mu h (fk = hk) \to \Phi f = \Phi h).$$ (To see the implication from top to bottom, just consider the case in which $f$ is taken to be $h$.) Hence, a functional $\Phi : (\NN \to \NN) \to \NN$ is in $\mathcal{C}^=_{(\NN \to \NN) \to \NN}$ if, and only if, it is continuous with a modulus of pointwise continuity (given by $\mu$). Since in $\Sw$ a continuous functional $\Phi : (\NN \to \NN) \to \NN$ always has a modulus of pointwise continuity, $\mathcal{C}^=_{(\NN \to \NN) \to \NN}$ is simply the set of continuous functionals from the Baire space to the natural numbers.

We do not know much more about the typestructure $\mathcal{C}^=$, nor how it relates to known models of G\"odel's theory $\GodelT$.

\section{Bounding interpretations of arithmetic}
\label{bounding-section}

In the previous section, we considered two different interpretations of arithmetic based on the precise interpretation at the ground type $\NN$ (see also Example \ref{precise-def}). In the case of arithmetic, we can also consider an alternative interpretation of the ground type $\NN$ using the ordering of the natural numbers:

\begin{example}[Bounding interpretation] \label{bounding-def} The {\em bounding interpretation} of $\witness{n}{\NN}$ is defined as: 
\begin{itemize}
    \item $\pvec \tau_\NN^+$ is $\NN$,
    \item $\pvec \tau_\NN^-$ is the empty tuple, and
    \item $\uwitness{n}{\NN}{m}{}$ is $n \leq m$.
\end{itemize}
\end{example}

The Boolean type is interpreted uniformly (cf.\ Section \ref{disjunction-section} and Example \ref{uniform-def}). In fact, the precise interpretation of the Booleans does not work alongside the bounding interpretation at the ground type $\NN$. For instance, $\HAomega$ proves the sentence $\forall x^\NN (x \equiv_\NN \Zero \vee x \not\equiv_\NN \Zero)$, but the bounding interpretation of this sentence, with the precise interpretation at the Boolean type, is $$\exists f^{\NN \to \BB} \forall n^\NN\, \forall x \leq n ((fn = 0 \to x = 0) \wedge (fn = 1 \to x \neq 0)).$$ A soundness theorem would give a closed term $t$ of type $\NN \to \BB$ such that the target theory proves $\forall n^\NN \forall x \leq n \,((tn = 0 \to x = 0) \wedge (tn = 1 \to x \neq 0)).$ Put $n=1$ and call $b$ the concrete Boolean value $t1$. We get $$\forall x \leq 1 ((b=0 \to x=0) \wedge (b=1 \to x\neq 0)).$$ This is impossible. Technically, the problem lies with the recursor $\Rec_\BB$, needed to interpret the simple induction that shows $\forall x^\NN (x \equiv_\NN \Zero \vee x \not\equiv_\NN \Zero)$. 

Summing up: in this section, $\I$-interpretations are bounding at the ground type $\NN$ and uniform at the Boolean type $\BB$. We, therefore, need to have a notion of joining of bounds. We let $n \vee_{\I_\NN} m$ be $\max(n,m)$. There is nothing to do regarding joining of bounds with respect to $\I_\BB$ or the atomic equality symbols $\equiv_\sigma$. On the other hand, the joining of bounds at the type-informative predicates of function types must be defined according to the concrete situation at hand. In the next subsections, we discuss three options. The third is the canonical interpretation. For the latter option, the notion of joining of bounds for the type-informative predicates of function types is the one described in Section \ref{canonical-section}.

It is easy to calculate the bounding interpretation of first-order quantifications:
\eqleft{
\begin{array}{lcl}
     \iInter{\forall x^\NN A(x)}{\pvec f}{n, \pvec b} 
        & \mbox{\,\,\,is\,\,\,} & \forall x \leq n \, \iInter{A(x)}{\pvec f n}{\pvec b} \\[1mm]
    \iInter{\exists x^\NN A(x)}{n, \pvec a}{\pvec B} 
        & \mbox{\,\,\,is\,\,\,} & \exists x \leq n \, \forall \pvec b \in \pvec B \, \iInter{A(x)}{\pvec a}{\pvec b}.
\end{array}
}
We are using bounded quantifiers in $\Lang{\TargetT}$ as abbreviations in the usual manner.

\subsection{Majorizability upon the bounding interpretation \texorpdfstring{$\Imb$}{Imb}}\label{maj_upon_bounding-section}

The first option we consider is to extend the bounding interpretation from ground types to function types using Bezem's strong majorizability notion (see Subsection \ref{star-majorizability}). We have also to account for the Boolean type $\BB$. We first define a transformation that associates to each finite type built upon the ground types $\NN$ and $\BB$ either an arithmetical type or no type at all (an empty tuple of types). 

\begin{definition}[$\sigma \mapsto \sigma^\dagger$] \label{def-dagger} For each finite type $\sigma$ based on the ground types $\NN$ and $\BB$, let us define either an arithmetical type $\sigma^\dagger$ or the empty tuple. We let $\NN^\dagger$ be $\NN$, let $\BB^\dagger$ be the empty tuple, and we let $(\sigma \to \tau)^\dagger$ be $\sigma^\dagger \to \tau^\dagger$. In case $\sigma^\dagger$ is the empty tuple, we understand that $(\sigma \to \tau)^\dagger$ is $\tau^\dagger$. In case $\tau^\dagger$ is empty, we understand that $(\sigma \to \tau)^\dagger$ is empty. 
\end{definition}

\begin{notation}\label{bounded-type}
    When $\sigma^\dagger$ is empty, we say that the type $\sigma$ is \emph{uniform} or {\em bounded} and we say that its elements are \emph{bounded}. Note that $(\sigma^{\dagger})^\dagger$ is $\sigma^\dagger$ provided that we understand the dagger of the empty tuple as the empty tuple.
\end{notation}

We give three examples. If $\rho$ is $(\NN \to \NN) \to \BB$, then $\rho^\dagger$ is empty. If $\rho$ is the type $\NN \to (\BB \to \NN)$, then $\rho^\dagger$ is $\NN \to \NN$. If $\rho$ is $(\NN \to \BB) \to \NN$, then $\rho^\dagger$ is $\NN$. 

We extend Bezem's majorizability notion to all unbounded finite types $\rho$ of $\HAomega$ (i.e., to types $\rho$ such that $\rho^\dagger$ is not empty). The extended relation $x \Bmaj_\rho a$ is a relation between $x$ of type $\rho$ and $a$ of arithmetical type $\rho^\dagger$. For $x,a : \NN$, $x \Bmaj_\NN a$ is $x \leq a$. If $\sigma$ and $\tau$ are unbounded, $f: \sigma \to \tau$ and $g: \sigma^\dagger \to \tau^\dagger$, define $f \Bmaj_{\sigma \to \tau} g$ as
$$\forall x^\sigma, a^{\sigma^\dagger} (x \Bmaj_\sigma a \to fx \Bmaj_\tau ga) \wedge \forall z^{\sigma^\dagger}, a^{\sigma^\dagger} (z \Bmaj_{\sigma^\dagger} a \to gz \Bmaj_{\tau^\dagger} ga).$$
If $\sigma$ is bounded, $\tau$ is unbounded, $f: \sigma \to \tau$ and $c: \tau^\dagger$, we define $f \Bmaj_{\sigma \to \tau} c$ as $$\forall x^\sigma (fx \Bmaj_\tau c) \wedge c \Bmaj_{\tau^\dagger} c.$$

Note that if $\rho$ is an arithmetical type, then $\rho^\dagger$ is $\rho$ and the above definition coincides with Bezem's definition. Let us see how this definition fares with the three examples given above. Suppose that $f: \NN \to (\BB \to \NN)$ and $g: \NN \to \NN$. Then, the definition of $f \Bmaj_{\NN \to (\BB \to \NN)} g$ gives 
$$\forall a^\NN \forall x\leq a (\forall i^\BB (f(x,i) \leq g(a)) \wedge g(x) \leq g(a)).$$ 
If $\phi: (\NN \to \BB) \to \NN$ and $n:\NN$, then $\phi \Bmaj_{(\NN \to \BB) \to \NN} n$ means $\forall h^{\NN \to \BB} (\phi(h) \leq n)$. Finally, if $\sigma$ is the type $(\NN \to \NN) \to \BB$ then $\sigma$ is bounded.

The next two lemmas adapt well-known properties of Bezem's strong majorizability notion to the above wider setting.

\begin{lemma}\label{self-maj} Let $\sigma$ be an unbounded type and let be given $x$ of type $\sigma$ and $a$ of type $\sigma^\dagger$. If $x \Bmaj_\sigma a$, then $a \Bmaj_{\sigma^\dagger} a$. \end{lemma}

\begin{proof} This is clear given the second conjunct of the definition of strong majorizability for the arrow type.
\end{proof}

\begin{lemma}\label{transitivity} Let $\sigma$ be an unbounded type, and let be given $x$ of type $\sigma$ and $a$ and $b$ of type $\sigma^\dagger$. If $x\Bmaj_\sigma a$ and $a \Bmaj_{\sigma^\dagger} b$, then $x \Bmaj_\sigma b$. \end{lemma}

\begin{proof} By induction on the complexity of the (unbounded) type $\sigma$. The base case of the ground type $\NN$ is obvious. \\[1mm]
Let be given unbounded types $\rho$ and $\tau$, and take $f$ of type $\rho \to \tau$ and $g$ and $h$ of type $\rho^\dagger \to \tau^\dagger$. Suppose that $f \Bmaj_{\rho \to \tau} g$ and $g \Bmaj_{\rho^\dagger \to \tau^\dagger} h$. We must show that $f \Bmaj_{\rho \to \tau} h$. Take $x$ of type $\rho$, $a$ of type $\rho^\dagger$ and assume that $x \Bmaj_\rho a$. We get $fx \Bmaj_\tau ga$. By the previous lemma, $a \Bmaj_{\rho^\dagger} a$ and, therefore, we also have $ga \Bmaj_{\tau^\dagger} ha$. By induction hypothesis, we obtain $fx \Bmaj_\tau ha$. We have argued the first conjunct of what we want. The second conjunct follows immediately from the previous lemma. \\[1mm]
It remains to check the case when $\rho$ is bounded, but this is also easy.
\end{proof}

The following lemma helps in simplifying some calculations.

\begin{lemma}\label{on-maj} Let $\sigma$, $\tau$ and $\rho$ be unbounded types, let $f$ be of type $\sigma \to (\tau \to \rho)$ and let $g$ be of type $\sigma^\dagger \to (\tau^\dagger \to \rho^\dagger)$. Suppose that both
\eqleft{
\begin{array}{l}
    \forall x^\sigma, a^{\sigma^\dagger} \forall y^\tau, b^{\tau^\dagger}(x \Bmaj_\sigma a \wedge y \Bmaj_\tau b \to fxy \Bmaj_\rho gab) \mbox{\,\,and}\\[1mm]
    \forall z^{\sigma^\dagger}, a^{\sigma^\dagger} \forall w^{\tau^\dagger}, b^{\tau^\dagger} (z \Bmaj_{\sigma^\dagger} a \wedge w \Bmaj_{\tau^\dagger} b \to gzw \Bmaj_{\rho^\dagger} gab).
\end{array}
}
Then, $f \Bmaj_{\sigma \to (\tau \to \rho)} g$. \end{lemma}

\begin{proof} Let us suppose the hypothesis of the lemma. We must show that $f \Bmaj_{\sigma \to (\tau \to \rho)} g$. Therefore, we must show both of the following conjuncts: $$\forall x^\sigma, a^{\sigma^\dagger} (x \Bmaj_\sigma a \to fx \Bmaj_{\tau \to \rho} ga) \mbox{\,\,\,\,and\,\,\,\,}\forall z^{\sigma^\dagger}, a^{\sigma^\dagger} (z \Bmaj_{\sigma^\dagger} a \to gz \Bmaj_{\tau^\dagger \to \rho^\dagger} ga).$$ 
For the first of the above conjuncts, take $x \Bmaj_\sigma a$ in order to see that $$\forall y^\tau, b^{\tau^\dagger}(y \Bmaj_\tau b \to fxy \Bmaj_\rho gab) \mbox{\,\,\,\,and\,\,\,\,}\forall w^{\tau^\dagger}, b^{\tau^\dagger} ( w \Bmaj_{\tau^\dagger} b \to gaw \Bmaj_{\rho^\dagger} gab).$$ The first one of these follows from the first hypothesis of the lemma. The second one follows from the second hypothesis, noting that $a \Bmaj_{\sigma^\dagger} a$. \\[1mm]
For the second conjunct, take $z \Bmaj_{\sigma^\dagger} a$. We must show that $$\forall w^{\tau^\dagger}, b^{\tau^\dagger} (w \Bmaj_{\tau^\dagger} b \to gzw \Bmaj_{\rho^\dagger} gab)\mbox{\,\,\,\,and \,\,\,\,}\forall w^{\tau^\dagger}, b^{\tau^\dagger} (w \Bmaj_{\tau^\dagger} b \to gaw \Bmaj_{\rho^\dagger} gab).$$ Both of the above follow from the second hypothesis of the lemma, since $a \Bmaj_{\sigma^\dagger} a$.
\end{proof}

It is now time to define the majorizability interpretation upon the bounding interpretation, dubbed the $\Imb$-interpretation of arithmetic.

\begin{definition}[$\Imb$-interpretation]\label{Imb-interpretation} The $\Imb$-interpretation is the base interpretation of arithmetic that defines the type-informative predicates as follows: 
\begin{itemize}
    \item If $\sigma$ is an unbounded type, then $\pvec \tau_\sigma^+$ is $\sigma^\dagger$, $\pvec \tau_\sigma^-$ is the empty tuple and $\uwitness{x}{\sigma}{a}{}$ is $x \Bmaj_\sigma a$.
    \item If $\sigma$ is a bounded type, then $\witness{x}{\sigma}$ is interpreted uniformly, as in Example \ref{uniform-def}.
\end{itemize}
The witnessing types of the $\Imb$-interpretation are the arithmetical star types and their elements are, by definition, its \emph{monotone elements}.
\end{definition}

\begin{remark} \label{remark-monotone-elements}
    Monotone elements were defined in Subsection \ref{star-majorizability}. Notice that, by Lemma \ref{self-maj}, if $x \Bmaj_\sigma a$, then $a$ is monotone. Quantifications over elements of witnessing types are meant to range over monotone elements. Syntactically, this can be accomplished by relativizing all of these quantifications to elements $x$ such that $x \Bmaj_\sigma x$ (cf. \cite{FerreiraOliva(05), FerreiraOliva(07)}).
\end{remark}

\begin{notation}
    The $\U$-interpretation fixed by the definition above is called the $\Umb$-interpretation.
\end{notation}

Let us show that the $\Imb$-interpretation is a type-informative base interpretation of arithmetic. We must check conditions $(i)$ and $(ii)$ of Definition \ref{def-type-informative-base}. It is well known that the arithmetical combinators are monotone and, therefore, condition $(i)$ holds for them. In general (also for types involving $\BB$), one has $\Pi_{\rho,\sigma} \Bmaj \Pi_{\rho^\dagger,\sigma^\dagger}$ and $\Sigma_{\rho,\sigma,\gamma} \Bmaj \Sigma_{\rho^\dagger,\sigma^\dagger,\gamma^\dagger}$ (with a correct understanding for the cases when types are bounded). Let us now consider condition $(ii)$. We use Remark \ref{criterium}: we let $q_{\sigma,\tau}$ and $r_{\sigma,\tau}$ be the empty tuple and let $t_{\sigma,\tau}$ be given as follows. If both $\sigma$ and $\tau$ are unbounded, then $t_{\sigma,\tau}$ has type $\sigma^\dagger \to (\sigma^\dagger \to \tau^\dagger) \to \tau^\dagger$ and $tac = ca$. If $\sigma$ is bounded and $\tau$ is unbounded, then $t_{\sigma,\tau}$ has type $\tau^\dagger \to \tau^\dagger$ and $tc = c$. If $\tau$ is bounded, then $t_{\sigma,\tau}$ is empty. A discussion by cases easily shows that the condition of Remark \ref{criterium} is met. Note that $t_{\sigma,\tau}$ is, technically, $\lambda a,c.ca$ if $\sigma$ is unbounded, and it is $\lambda c.c$ if $\sigma$ is bounded. In both cases, $t_{\sigma,\tau}$ is monotone since it is defined (by lambda abstraction) using the combinators and application. With respect to condition $(i)$, we still need to check the constants other than the combinators. The substantial cases to discuss are the conditional constants and the recursors. For both of these cases, we need to use the maximum functional between a pair of elements of the same arithmetical type (see Subsection \ref{star-majorizability}). 

\begin{proposition} The sentence $\I_{\BB \to \sigma \to \sigma \to \sigma}(\mathrm{Cond_\sigma})$ is $\Umb$-interpretable, for each type $\sigma$.
\end{proposition}

\begin{proof} If $\sigma$ is unbounded, $(\BB \to \sigma \to \sigma \to \sigma)^\dagger$ is $\sigma^\dagger \to \sigma^\dagger \to \sigma^\dagger$. Clearly, $$\forall i^\BB \forall x^\sigma, y^\sigma \forall a^{\sigma^\dagger}, b^{\sigma^\dagger} (x \Bmaj_\sigma a \wedge y \Bmaj_\sigma b \to \mathrm{Cond}_\sigma (i,x,y) \Bmaj_{\sigma} \mbox{$\max_{\sigma^\dagger}$} (a,b)).$$ With the aid of Lemma \ref{self-maj}, Lemma \ref{transitivity}, Lemma \ref{on-maj}, Lemma \ref{lemma-max} and the above, it can easily be concluded that $$\mathrm{Cond}_\sigma \Bmaj_{\BB \to \sigma \to \sigma \to \sigma} \lambda a,b. \max_{\sigma^\dagger}(a,b).$$ 
There is nothing to show when $\sigma$ is bounded. \end{proof}

\begin{proposition}\label{Rec-Umb} The sentence $\witness{\Rec_\sigma}{\NN \to \sigma \to (\NN \to \sigma \to \sigma) \to \sigma}$ is $\Umb$-interpretable, for each type $\sigma$. 
\end{proposition}

\begin{proof} The proof of this fact mimics Howard's and Bezem's proofs \cite{Howard(73), Bezem(85)} that show that recursors are majorizable. We sketch the argument. Let $\sigma$ be a type of the language of $\HAomega$. If $\sigma$ is bounded, then $\Rec_\sigma$ is bounded and there is nothing to prove. We can assume that $\sigma$ is unbounded. In this case 
$$(\NN \to \sigma \to (\NN \to \sigma \to \sigma) \to \sigma)^\dagger \mbox{\,\,\,is\,\,\,} \NN \to \sigma^\dagger \to (\NN \to \sigma^\dagger \to \sigma^\dagger) \to \sigma^\dagger.$$ 
It is easy to check by induction on $n$ that 
$$\Rec_\sigma n \Bmaj_{\sigma \to (\NN \to \sigma \to \sigma) \to \sigma} \Rec_{\sigma^\dagger} n.$$ 
Define, by recursion, the functional $R'$ of type $\NN \to \sigma^\dagger \to (\NN \to \sigma^\dagger \to \sigma^\dagger) \to \sigma^\dagger$:
\eqleft{
\begin{array}{lcl}
    R' \Zero 
        & \equiv_{\sigma^\dagger \to (\NN \to \sigma^\dagger \to \sigma^\dagger) \to \sigma^\dagger} & \Rec_{\sigma^\dagger} \Zero \\[1mm]
    R' (\Suc n)
        & \equiv_{\sigma^\dagger \to (\NN \to \sigma^\dagger \to \sigma^\dagger) \to \sigma^\dagger} & \max_{\sigma^\dagger \to (\NN \to \sigma^\dagger \to \sigma^\dagger) \to \sigma^\dagger} (R' n, \Rec_{\sigma^\dagger} (\Suc n)).
\end{array}
}
By the lemmata of this subsection, we can see easily, by induction on $n$, that 
$$\forall m \leq n \, (\Rec_\sigma m \Bmaj_{\sigma \to (\NN \to \sigma \to \sigma) \to \sigma} R' n).$$
With the aid of this, we can conclude that $\Rec_\sigma \Bmaj_{\NN \to \sigma \to (\NN \to \sigma \to \sigma) \to \sigma} R'$, as wanted.
\end{proof}

In bounding interpretations of $\HAomega$, we need a notion of joining of bounds. It is at this juncture that we need to have witnessing elements as monotone functionals. Given an unbounded type $\sigma$ of $\HAomega$ and given $a,b : \sigma^\dagger$ we define $a \vee_{\I_\sigma} b$ to be $\max_{\sigma^\dagger}(a,b)$. This is a joining of bounds by Lemma \ref{lemma-max} (and the discussions around it) and Lemma \ref{transitivity}. 

\begin{lemma} For any formula $A$ of the language of arithmetic and $\pvec a, \pvec b$ tuples of monotone elements (of arithmetical star type) having the tuple type of the positive $\Imb$-witnesses of $A$, we have that $\pvec a \vee_A \pvec b$ is a tuple of monotone elements (of arithmetical star type). Furthermore, $\pvec a \Bmaj \! \pvec a \vee_A \pvec b$.
\end{lemma}

\begin{proof}
    The proof is by induction on the complexity of $A$ (seen as a formula of the language of $\HAomega$ extended with the type informative predicates). The base case concerns the type-informative predicates, where $\uwitness{x}{\sigma}{a}{}$ is $x \Bmaj_\sigma a$. As noticed, if $a$ and $b$ are monotone, so is $\max_{\sigma^\dagger} (a,b)$ and $a \Bmaj_{\sigma^\dagger} \max_{\sigma^\dagger} (a,b)$. Note that $a \vee_{\I_\sigma} b$ is $\max_{\sigma^\dagger} (a,b)$. For complex $A$, one looks at the clauses of Definition \ref{joining-bounds}. Since the joinings are defined using combinators, star operations and applications (see Subsection \ref{star-majorizability}), it is clear that $\pvec a \vee_A \pvec b$ is a tuple of monotone elements. We argue the property $\pvec a \Bmaj \! \pvec a \vee_A \pvec b$ for the case of implication (the other cases pose no issues). We use the notation of (ii) of Definition \ref{joining-bounds}. We must show that $\pvec f \Bmaj \lambda \pvec a. (\pvec f\pvec a {\pvec \vee}_B \pvec f' \pvec a)$ and $\pvec g \Bmaj \lambda \pvec a, \pvec d. (\pvec g\pvec a \pvec d \cup \pvec g' \pvec a \pvec d)$. Suppose that $\pvec x \Bmaj \pvec a$. By the monotonicity of $\pvec f$, we have $\pvec f \pvec x \Bmaj \pvec f \pvec a$. By induction hypothesis, $\pvec f \pvec a \Bmaj \pvec f \pvec a \vee_B \pvec f' \pvec a$. We obtain $\pvec f \pvec x \Bmaj \pvec f \pvec a \vee_B \pvec f' \pvec a$, as wished. For the second majorizability property, suppose further that $\pvec y \Bmaj \pvec d$. By the monotonicity of $\pvec g$, we have $\pvec g \pvec x \pvec y \Bmaj \pvec g \pvec a \pvec d$. By induction hypothesis, $\pvec g \pvec a \pvec d \Bmaj \pvec g \pvec a \pvec d \cup \pvec g' \pvec a \pvec d$. Therefore, $\pvec g \pvec x \pvec y \Bmaj \pvec g \pvec a \pvec d \cup \pvec g' \pvec a \pvec d$. \end{proof}

\begin{theorem}\label{soundness of Imb}
    The $\Imb$-interpretation is a sound type-informative base interpretation of $\HAomega$.
\end{theorem}

\begin{proof} We have seen that the $\Imb$-interpretation is type-informative. It remains to show that the non-logical axioms of $\HAomega$ are $\Imb$-interpretable. The only concern are the induction axioms. Their interpretation, using the recursors, is a bit different from the usual argument. So, we work it out in some detail. \\[1mm]
As we have seen before, induction axioms can be replaced by the (parameter-free) rule of induction. Hence, given a formula $A(x^\NN)$, we must see that if $A(\Zero)$ and $\forall x^\NN (A(x) \to A(\Suc x))$ are $\Imb$-interpretable, then $\forall x^\NN A(x)$ is also $\Imb$-interpretable. So, by assumption, there are monotone terms $\pvec t$, $\pvec s$ and $\pvec r$ such that the target theory proves $\forall \pvec b \iInter{A(\Zero)}{\pvec t}{\pvec b}$ and $$\forall n^\NN \forall \pvec a \forall \pvec b \forall x \leq n ( \forall \pvec b' \in \pvec s n \pvec a \pvec b \iInter{A(x)}{\pvec a}{\pvec b'} \to \iInter{A(\Suc x)}{\pvec r n \pvec a}{\pvec b})$$ and so $$\forall n^\NN \forall \pvec a \forall x \leq n ( \forall \pvec b \iInter{A(x)}{\pvec a}{\pvec b} \to \forall \pvec b \iInter{A(\Suc x)}{\pvec r n \pvec a}{\pvec b}).$$
Define, by recursion, functionals $\pvec R$ such that:
\eqleft{
\begin{array}{lcl}
    \pvec R \Zero \pvec t \pvec r
        & \equiv_{\pvec \tau^+_A} & \pvec t \\[1mm]
    \pvec R (\Suc n) \pvec t \pvec r
        & \equiv_{\pvec \tau^+_A} & \pvec R n \pvec t \pvec r \vee_A \pvec r n (\pvec R n \pvec t \pvec r).
\end{array}
}
Remember that $\pvec \tau^+_A$ is the tuple type of the positive witnesses of the formula $A$, as described in Definition \ref{joining-bounds}. Using the monotonicity property of Remark \ref{monotonicity-remark}, it is straightforward to show by induction on $n$ that $$\forall n \forall x \leq n \forall \pvec b \, \iInter{A(x)}{\pvec R n \pvec t \pvec r}{\pvec b}.$$ This gives what we want. Note that the previous lemma together with an easy induction shows that $\pvec R$ is a tuple of monotone terms. \end{proof}

The $\Imb$-interpretation of quantifications is rather simple. If $\sigma$ is unbounded,
\eqleft{
\begin{array}{lcl}
     \iInter{\forall x^\sigma A(x)}{\pvec f}{c, \pvec b} 
        & \mbox{\,\,\,is\,\,\,} & \forall x \Bmaj_\sigma \! c \, \iInter{A(x)}{\pvec f c}{\pvec b} \\[1mm]
    \iInter{\exists x^\sigma A(x)}{c, \pvec a}{\pvec B} 
        & \mbox{\,\,\,is\,\,\,} & \exists x \Bmaj_\sigma \!c \, \forall \pvec b \in \pvec B \, \iInter{A(x)}{\pvec a}{\pvec b}
\end{array}
}
where the meaning of the bounded quantifiers is clear. Note the similarity with the interpretation of first-order quantifications. If $\sigma$ is bounded, we get the clauses of uniform quantification:
\eqleft{
\begin{array}{lcl}
       	\iInter{\forall x^\sigma  A(x)}{\pvec a}{\pvec b} & \mbox{\,\,\,is\,\,\,} & \forall x^\sigma \iInter{A(x)}{\pvec a}{\pvec b} \\[1mm]
	\iInter{\exists x^\sigma A(x)}{\pvec a}{\pvec B} & \mbox{\,\,\,is\,\,\,} & \exists x^\sigma \forall \pvec b \in \pvec B \, \iInter{A(x)}{\pvec a}{\pvec b}.
\end{array}	
}

The $\Imb$-interpretation of this subsection has some affinities with the bounded functional interpretation of \cite{FerreiraOliva(05)}. The reader familiar with the bounded functional interpretation knows that it emphasizes very much intensional notions of majorizability and bounded quantifications thereof. We will briefly comment on this in Subsection \ref{bfi-section}.

We finish this subsection with the typestructure $\Bw$, defined as $(\mathcal{B}_\sigma)_\sigma$ with
\eqleft{
\begin{array}{lcl}
    \mathcal{B_\sigma} 
        & := & \{x \in \mathcal{S}_\sigma : \exists a \in \mathcal{S}_\sigma \, (x \Bmaj_\sigma a) \} 
\end{array}
}
for each arithmetical type $\sigma$. As discussed in Section \ref{interlude}, by Theorem \ref{models_of_T}, it is a model of G\"odel's theory $\GodelT$. It consists of the set-theoretic functionals which have a strong majorant in $\Sw$.

\subsection{Finiteness interpretation \texorpdfstring{$\If$}{If}}\label{finiteness-interpretation}

Let us consider now a second alternative for a ``bounding interpretation'' of the type-informative predicates making use of finite sets:  

\begin{definition}[$\If$-interpretation] The $\If$-interpretation is the base interpretation of arithmetic that defines the type-informative predicates as follows: given any type $\sigma$, $\pvec \tau_\sigma^+$ is $\sigma^*$, $\pvec \tau_\sigma^-$ is the empty tuple and $\uwitness{x}{\sigma}{F}{}$ is $x \in F$ (see Example \ref{finite-set-def}).
\end{definition}

This is a minor departure from a bounding interpretation: the (positive) witness for being an element of the ground type $\NN$ is not a bound for that element, but is rather a finite set of numbers that includes the element. At the ground type $\BB$, we have essentially the uniform interpretation because $\exists F^{\BB^*} \forall i^\BB (i\in F)$.

\begin{proposition} 
The $\If$-interpretation is a type-informative base interpretation of arithmetic.
\end{proposition}

\begin{proof} Condition $(i)$ of Definition \ref{def-type-informative-base} trivially holds because, for any constant $c^\sigma$ of the language, we have $c \in \{c\}$. For condition $(ii)$, we use Remark \ref{criterium}: we let $\pvec q_{\sigma,\tau}$ and $\pvec r_{\sigma,\tau}$ be the empty tuple; $t_{\sigma,\tau}$ is of type $\sigma^* \to (\sigma \to \tau)^* \to \tau^*$ and we let it be defined such that $$t_{\sigma,\tau} FG = \bigcup_{f\in G} \bigcup_{x \in F} \{fx\}.$$ 
This works since $\forall x^\sigma \, \forall f^{\sigma \to \tau} \forall F^{\sigma^*} \forall G^{(\sigma \to \tau)^*} (x \in F \wedge f \in G \to fx \in t_{\sigma, \tau}FG).$
\end{proof}

The notion of joining of bounds for the $\If$-interpretation is clear. Given $\sigma$ a type and $F$ and $G$ of type $\pvec \tau_\sigma^+$, i.e., of type $\sigma^*$, we let $F \vee_{\I_\sigma} G$ be the union $F \cup_\sigma G$. This is obviously a joining of bounds. 

\begin{theorem}
    The $\If$-interpretation is a sound type-informative base interpretation of $\HAomega$.
\end{theorem}

\begin{proof} Now that we have seen that the $\If$-interpretation is type-informative, it remains to show that the non-logical axioms of $\HAomega$ are $\If$-interpretable. As usual, the only concern are the induction axioms. The verification follows the usual line of thought, as in the proof of Theorem \ref{soundness of Imb}. However, it has new features because we are dealing with finite sets instead of bounds (numbers). \\[1mm]
We analyze the (parameter-free) rule of induction. Given a formula $A(x^\NN)$, we must see that if both $A(\Zero)$ and $\forall x^\NN (A(x) \to A(\Suc x))$ are $\If$-interpretable, then formula $\forall x^\NN A(x)$ is also $\If$-interpretable. By assumption, there are terms $\pvec t$, $\pvec s$ and $\pvec r$ such that the target theory proves $\forall \pvec b \iInter{A(\Zero)}{\pvec t}{\pvec b}$ and 
\begin{equation}\label{auxindifb}
\forall F^{\NN^*} \forall \pvec a \forall \pvec b \forall x \in F ( \forall \pvec b' \in \pvec s F \pvec a \pvec b \iInter{A(x)}{\pvec a}{\pvec b'} \to \iInter{A(\Suc x)}{\pvec r F \pvec a}{\pvec b}).
\end{equation} 
We can define by recursion a functional of type $\NN \to \NN^*$ which, to each $n^\NN$, associates the finite set $F_n = \{0,1,\ldots, n\}$. Particularizing in (\ref{auxindifb}), we get 
$$\forall n^\NN \forall \pvec a \forall \pvec b \forall x \leq n ( \forall \pvec b' \in \pvec s' n \pvec a \pvec b \iInter{A(x)}{\pvec a}{\pvec b'} \to \iInter{A(\Suc x)}{\pvec r' n \pvec a}{\pvec b}),$$ 
where $\pvec s' n \pvec a \pvec b = \pvec s F_n \pvec a \pvec b$ and $\pvec r' n \pvec a = \pvec r F_n \pvec a$. Therefore, $$\forall n^\NN \forall \pvec a \forall x \leq n ( \forall \pvec b \iInter{A(x)}{\pvec a}{\pvec b} \to \forall \pvec b \iInter{A(\Suc x)}{\pvec r' n \pvec a}{\pvec b}).$$
Define, by recursion, functionals $\pvec R$ such that:
\eqleft{
\begin{array}{lcl}
    \pvec R \Zero \pvec t \pvec r'
        & \equiv_{\pvec \tau^+_A} & \pvec t \\[1mm]
    \pvec R (\Suc n) \pvec t \pvec r'
        & \equiv_{\pvec \tau^+_A} &  \pvec R n \pvec t \pvec r' \vee_A \pvec r' n (\pvec R n \pvec t \pvec r').
\end{array}
}
Using Remark \ref{monotonicity-remark}, it is straightforward to show, by induction on $n$, that $$\forall n \forall x \leq n \forall \pvec b \, \iInter{A(x)}{\pvec R n \pvec t \pvec r'}{\pvec b}.$$ However, we rather want to conclude $$\forall F^{\NN^*} \forall \pvec b \, (\forall x \in F \, \iInter{A(x)}{\pvec p F}{\pvec b}),$$ for suitable terms $\pvec p$. Using the maximum constant $\mathrm{m}$ of type $\NN^* \to \NN$, we can just take $\pvec p F$ as $\pvec R (\mathrm{m}F)\pvec t \pvec r'$. \end{proof}

The $\If$-interpretation is not really new in the literature. It appeared before, dressed in different clothes. It is a linguistic variant of the functional interpretation for nonstandard arithmetic introduced in \cite{BergBriseidSafarik(12)} (cf.\ alternative presentation \cite{Oliva(20)}, with the usual function application in the clauses for $A \to B$ and $\forall x^\tau A(x)$). Whereas the interpretation of nonstandard arithmetic uses the standard (unary) predicates $\mathsf{st}^\sigma(x)$, we use instead the type-informative predicates $\I_\sigma(x)$. The internal quantifications of \cite{BergBriseidSafarik(12)} correspond to our uniform quantifications. The external quantifications correspond to our relativized quantifications. The internal formulas of \cite{BergBriseidSafarik(12)} are the formulas of our source language $\Lang{\HAomega}$. The external formulas, on the other hand, correspond to the arbitrary formulas of the language $\LangI{\HAomega}$. Here a difference emerges between the finiteness interpretation of $\HAomega$ and the functional interpretation of nonstandard arithmetic of \cite{BergBriseidSafarik(12),Oliva(20)}. The external induction scheme of \cite{BergBriseidSafarik(12)} corresponds to the application of induction to arbitrary formulas of $\LangI{\HAomega}$, whereas the induction scheme of $\HAomega$ only holds for relativized formulas. Moreover, the axioms of \cite{BergBriseidSafarik(12)} include the internal induction scheme whose corresponding principle is not present in our setting, but could be added (because it is stated in the language without the type-informative predicates; by the way, this is also the case with the extensionality axioms of \cite{BergBriseidSafarik(12)} for internal formulas). Another difference is that \cite{BergBriseidSafarik(12)} includes the star type in its language, whereas we opted not to have the star type in our source language. These differences, albeit not fundamental from a theoretical point of view, are very important for the applications of the paper \cite{BergBriseidSafarik(12)} because they allow its authors to take advantage of a more expressive language and of more general induction principles, thus enabling a robust connection of their system to familiar features of nonstandard arithmetic.

\subsection{Canonical upon the bounding interpretation \texorpdfstring{$\Icb$}{Icb}}\label{canonical-bounding-section}

Finally, we also consider a third alternative: combining the bounding interpretation at ground type $\NN$ with the canonical interpretation of function types:
\begin{definition}[$\Icb$-interpretation] The $\Icb$-interpretation is the base interpretation of arithmetic that defines the type-informative predicates as follows:  it is the bounding interpretation at the ground type $\NN$, the uniform interpretation at the ground type $\BB$, and it is extended canonically at the function types. 
\end{definition}
\begin{notation}
    The $\U$-interpretation fixed by the definition above is called the $\Ucb$-interpretation.
\end{notation}

We need a notion of joining of bounds. Recall that joining of bounds for the ground type $\NN$ is given by letting $n \vee_{\I_\NN} m$ be $\max(n,m)$ and, for function types, joining of bounds is described in Definition \ref{join-bound-function-def}. By Proposition \ref{join-bound-function-prop}, we have the joining of bounds property for all types.

It is easy to see that the $\Ucb$-interpretations of the type-informative predicates at types 0, 1 and 2 are as follows:
\eqleft{
\begin{array}{lcl}
    \uwitness{x}{\NN}{n}{} 
        & \mathrm{is} & x \leq n \\[1mm]
   \uwitness{f}{\NN \to \NN}{g}{k} 
        & \mathrm{is} & \forall n \leq k (fn \leq gk) \\[1mm]
   \uwitness{\Phi}{(\NN \to \NN) \to \NN}{\Psi, \mu}{h} 
        & \mathrm{is} & \forall f (\forall k \in \mu h \forall n \leq k (fn \leq hk) \to \Phi f \leq \Psi h)
\end{array}
}
where $n : \NN$, $g,h : \NN \to \NN$, $\mu: (\NN\to \NN) \to \NN^*$ and $\Psi: (\NN\to \NN) \to \NN$. The reader should compare these clauses with the corresponding clauses of Subsection \ref{canonical-precise-section}.

It is clear that one has $\uwitness{\Zero}{\NN}{0}{}$ and $\forall k \uwitness{\Suc}{\NN \to \NN}{S}{k}$, where $S$ is the successor function. With the following result, we show that the $\Icb$-interpretation is a type-informative base interpretation of $\HAomega$:

\begin{proposition}\label{rec_bounding} The sentence $\witness{\Rec_\sigma}{\NN \to \sigma \to (\NN \to \sigma \to \sigma) \to \sigma}$ is $\Ucb$-interpretable, for each type $\sigma$. \end{proposition}
\begin{proof} The proof is similar to that of Proposition \ref{rec_precise}, with some important differences. Our aim is to show that 
$$\forall n^\NN \forall x^\sigma \forall f^{\NN \to \sigma \to \sigma} (\witness{n}{\NN} \wedge \witness{x}{\sigma} \wedge \witness{f}{\NN \to \sigma \to \sigma} \to \witness{\Rec_\sigma n x f}{\sigma})$$
is $\Ucb$-interpretable. By definition,
\[
\uwitness{f}{\NN \to \sigma \to \sigma}{\pvec \phi, \pvec B}{k^*, \pvec a, \pvec b} \mbox{\,\, is \,\,}\forall y^\sigma (\forall \pvec b' \in \pvec B k^* \pvec a \pvec b \uwitness{y}{\sigma}{\pvec a}{\pvec b'} \to  \forall i \leq k^* \uwitness{f i y}{\sigma}{\pvec \phi k^* \pvec a}{\pvec b}). 
\]
With this notation, it is not difficult to see that we must obtain closed terms $\pvec \Psi$, $\pvec B^*$, $K$, $\pvec A$ and $\pvec B^\dagger$ such that the target theory proves the following: for all appropriate $n^*$, $\pvec a$, $\pvec \phi$, $\pvec B$, $\pvec b$, and for all $x^\sigma$ and $f^{\NN \to \sigma \to \sigma}$, if 
\eqleft{
\forall \pvec b' \in \pvec B^* n^* \pvec a \pvec \phi \pvec B \pvec b \uwitness{x}{\sigma}{\pvec a}{\pvec b'} \mbox{\quad and}} 
\eqleft{
\forall k^* \in K n^* \pvec a \pvec \phi \pvec B \pvec b 
\forall \pvec a' \in \pvec A n^* \pvec a \pvec \phi \pvec B \pvec b 
\forall \pvec b'' \in \pvec B^\dagger n^* \pvec a \pvec \phi \pvec B \pvec b \, 
\uwitness{f}{\NN \to \sigma \to \sigma}{\pvec \phi, \pvec B}{k^*, \pvec a', \pvec b''}}
then $\forall i \leq n^*\uwitness{\Rec_\sigma i x f}{\sigma}{\pvec \Psi n^* \pvec a \pvec \phi \pvec B}{\pvec b}$. \\[1mm]
Given $n^*$, $\pvec a$, $\pvec \phi$, $\pvec B$, $\pvec b$, we define, using recursors, two sequences $\pvec a_j$ and $\pvec B_j$ inductively as follows:
\[
\begin{array}{rclcrcl}
\pvec a_0 & = & \pvec a & \hspace{8mm} & \pvec B_{n^*} & = & \{\pvec b\} \\[1mm]
\pvec a_{j+1} & = & \pvec a_j {\pvec \vee}_{\I_\sigma} \, \pvec \phi j\pvec a_j & \hspace{8mm} & \pvec B_j & = & \pvec B_{j+1} \cup \,\bigcup_{\pvec b'\in \pvec B_{j+1}} \pvec B j \pvec a_j \pvec b', 
\end{array}
\]
for $j$ with $0\leq j < n^*$. Now, we take $\pvec \Psi$, $\pvec B^*$, $K$, $\pvec A$ and $\pvec B^\dagger$ such that
\[
\begin{array}{rcl}
     \pvec \Psi n^* \pvec a \pvec \phi \pvec B & = & \pvec a_{n^*} \\[1mm]
     \pvec B^* n^* \pvec a \pvec \phi \pvec B \pvec b & = & \pvec B_0 \\[1mm]
     K n^* \pvec a \pvec \phi \pvec B \pvec b & = & \{ 0, 1, \ldots, n^* \} \\[1mm]
     \pvec A n^* \pvec a \pvec \phi \pvec B \pvec b & = & \{ \pvec a_0, \pvec a_1, \ldots, \pvec a_{n^*} \} \\[1mm]
     \pvec B^\dagger n^* \pvec a \pvec \phi \pvec B \pvec b & = & \bigcup_{j \leq n^*} \pvec B_j.
\end{array}
\]
Let us see that these terms do the job. Given $n^*, \pvec a, \pvec \phi, \pvec B$ and $\pvec b$, fix $x$ and $f$ and assume
\begin{equation} \label{one-hash}
    \forall \pvec b' \in \pvec B_0 \uwitness{x}{\sigma}{\pvec a}{\pvec b'}
\end{equation}
and, for each $k < n^*$ and for all $\pvec b'' \in \bigcup_{j \leq n^*} \pvec B_j$,
$$\forall y^\sigma (\forall \pvec b' \in \pvec B k \pvec a_k \pvec b'' \uwitness{y}{\sigma}{\pvec a_k}{\pvec b'} \to \forall i \leq k \uwitness{f i y}{\sigma}{\pvec \phi k \pvec a_k}{\pvec b''}). $$
By the joining of bounds property for $\I_\sigma$
\begin{equation} \label{two-hashes}
     \forall y^\sigma (\forall \pvec b' \in \pvec B k \pvec a_k \pvec b'' \uwitness{y}{\sigma}{\pvec a_k}{\pvec b'} \to \forall i \leq k \uwitness{f i y}{\sigma}{\pvec a_{k+1}}{\pvec b''}).
\end{equation}
We claim that, for all $0\leq k \leq n^*$,
\begin{equation} \label{three-hashes}
\forall i \leq k \forall \pvec b' \in \pvec B_k  \uwitness{\Rec_\sigma i x f}{\sigma}{\pvec a_k}{\pvec b'}.
\end{equation}
In particular, when $k=n^*$, we get $\forall i \leq n^* \uwitness{\Rec_\sigma i xf}{\sigma}{\pvec a_{n^*}}{\pvec b}$. This is what we want. We argue by induction on $0 \leq k \leq n^*$. The base case $k=0$ is given. Let be given $0\leq k < n^*$. We assume (\ref{three-hashes}) by induction hypothesis. Take $\pvec b'' \in \pvec B_{k+1}$. We must show $\forall i \leq k+1  \uwitness{\Rec_\sigma i x f}{\sigma}{\pvec a_{k+1}}{\pvec b''}$. There are two cases to consider. \\[1mm]
Firstly, the case $i=0$, that is $\uwitness{x}{\sigma}{\pvec a_{k+1}}{\pvec b''}$. Since $\pvec B_{k+1} \subseteq \pvec B_0$, by (\ref{one-hash}) we have $\uwitness{x}{\sigma}{\pvec a}{\pvec b''}$, i.e., $\uwitness{x}{\sigma}{\pvec a_0}{\pvec b''}$. Since $\I_\sigma$ satisfies the joining of bounds property, we can conclude that $\uwitness{x}{\sigma}{\pvec a_{k+1}}{\pvec b''}$. \\[1mm]
Secondly, the case $\forall j \leq k  \uwitness{fj(\Rec_\sigma j x f)}{\sigma}{\pvec a_{k+1}}{\pvec b''}$ (these are the cases $i=j+1$). Given that $\pvec B k \pvec a_k \pvec b'' \subseteq \pvec B_k$, (\ref{two-hashes}) entails 
$$\forall y^\sigma (\forall \pvec b' \in \pvec B_k \uwitness{y}{\sigma}{\pvec a_k}{\pvec b'} \to \forall i \leq k \uwitness{f i y}{\sigma}{\pvec a_{k+1}}{\pvec b''}).$$
Take $j\leq k$ and let $y$ be $\Rec_\sigma j x f$. By the above, we may conclude 
$$\forall \pvec b' \in \pvec B_k \uwitness{\Rec_\sigma j x f}{\sigma}{\pvec a_k}{\pvec b'} \to \uwitness{f j (\Rec_\sigma j x f)}{\sigma}{\pvec a_{k+1}}{\pvec b''}.$$ By the induction hypothesis (\ref{three-hashes}), we have $\forall \pvec b' \in \pvec B_k \uwitness{\Rec_\sigma j x f}{\sigma}{\pvec a_k}{\pvec b'}$. What we want follows by Modus Ponens.   \end{proof}

\begin{theorem}
    The $\Icb$-interpretation is a sound type-informative base interpretation of $\HAomega$.
\end{theorem}

\begin{proof}
    As usual, the induction axioms can be replaced by the parameter-free induction rule. The argument is similar to the argument of Theorem \ref{soundness of Imb}. The other axioms pose no problem.
\end{proof}

In the interpretations of the last two subsections, being of a certain type carries no negative information. This is the reason why the interpretations of the quantifications are rather simple. Canonical interpretations are not like that. Already being an element of type $\NN \to \NN$ carries, as we saw, both positive and negative information. The end result is that the interpretation of the quantifications becomes very involved as the typing rises. Look at what happens in type $\NN \to \NN$:
\eqleft{
\begin{array}{lcl}
    \iInter{\forall f^{\NN \to \NN} A(f)}{\pvec \Phi, \Psi}{\phi, \pvec b} 
        & \mbox{\,\,\,is\,\,\,} & \forall f (\forall k \in \Psi \phi \pvec b \forall n \leq k (f n \leq \phi k) \to \iInter{A(f)}{\pvec \Phi \phi}{\pvec b}) \\[1mm]
    \iInter{\exists f^{\NN \to \NN} A(f)}{\phi, \pvec a}{K, \pvec B} 
        & \mbox{\,\,\,is\,\,\,} & \exists f (\forall k \in K \forall n \leq k (f n \leq \phi k) \wedge \forall \pvec b \in \pvec B \, \iInter{A(f)}{\pvec a}{\pvec b}).
\end{array}
}
The complexity of the interpretation is, perhaps, a barrier for applications. This should be regarded as a preliminary conclusion. We need to study more and deepen our understanding before reaching a definitive conclusion, especially in view of the fact that canonical interpretations are very natural and enjoy interesting properties. As we saw, the $\Icp$-interpretation realizes extensionality for type 2 functionals (see Corollary \ref{prop-ext-2}). The $\Icb$-interpretation shows theoretical advantages with respect to collection principles, as we shall discuss in Subsection \ref{analytic-applications-section}. Also, the $\Icp$-interpretation gives rise to the model $\mathcal{C}^=$ of G\"odel's theory $\GodelT$. For the $\Icb$-interpretation, we have:
\begin{definition}[The typestructure $\mathcal{C^\leq}$] The typestructure $\mathcal{C}^\leq$ is defined as $(\mathcal{C^\leq_\sigma})_\sigma$, with
\eqleft{
\begin{array}{lcl}
    \mathcal{C^\leq_\sigma} 
        & := & \{ x \in \mathcal{S}_\sigma : \Sws \models \exists \pvec a \forall \pvec b \, \uInter{\witness{x}{\sigma}}{\pvec a}{\pvec b} \} 
\end{array}
}
for each arithmetical type $\sigma$, where $\langle \cdot \rangle$ is the $\Ucb$-interpretation. 
\end{definition}
By Theorem \ref{models_of_T}, we have:

\begin{theorem}
$\mathcal{C}^\leq$ is a model of G\"odel's theory $\GodelT$.
\end{theorem}
 
If $f$ is a set-theoretical function from $\NN$ to $\NN$, it is plain that $\forall k \, \uwitness{f}{\NN \to \NN}{f^M}{k}$, where $f^M (k) := \max_{n\leq k} f(n)$. So, $\mathcal{C}^\leq_{\NN \to \NN}$ coincides with $\mathcal{S}_{\NN \to \NN}$, the collection of all set-theoretic functions from $\NN$ to $\NN$. 

The case of $\mathcal{C}^\leq_{(\NN \to \NN) \to \NN}$ is more interesting. According to the explicit computation of the $\Ucb$-interpretation of the predicate $\I_{(\NN \to \NN) \to \NN}(\Phi)$, a functional $\Phi$ of $\mathcal{S}_{(\NN \to \NN) \to \NN}$ is in $\mathcal{C}^\leq_{(\NN \to \NN) \to \NN}$ if, and only if, 
\begin{equation}\label{Cleq}
    \exists \Psi , \mu \forall h \forall f (\forall k \in \mu h \forall n\leq k (fn \leq hk) \to \Phi f \leq \Psi h).
\end{equation}

Obviously, for such $\Psi$, $$\forall h \forall f (\forall k \forall n\leq k (fn \leq hk) \to \Phi f \leq \Psi h).$$ The above says that $\Psi$ majorizes $\Phi$ in the sense of Howard (and we may call the $\mu$ of the first formula, a modulus of pointwise majorizability). In short: if a functional is in $\mathcal{C}^\leq_{(\NN \to \NN) \to \NN}$, then it is majorizable in the sense of Howard (and, therefore, being of type $(\NN \to \NN) \to \NN$, it is well known that it is also majorizable in the sense of Bezem).

It is perhaps a bit surprising but, conversely, majorizable functionals (in the sense of Howard) of type $(\NN \to \NN) \to \NN$ are in $\mathcal{C}^\leq_{(\NN \to \NN) \to \NN}$. This is an easy consequence of a result of Oliva \cite{Oliva(03a)} about the weak continuity of majorizable functionals. Applied to our situation, Oliva's result states the following: if $\Phi : (\NN \to \NN) \to \NN$ is a majorizable functional in the sense of Howard (or Bezem), then $\Phi$ is weakly continuous, i.e., 
$$\forall h \exists r \forall f (\forall k\leq r (fk = hk) \to \Phi(f) \leq r).$$ 
We need a slightly stronger statement. Oliva's proof readily adapts to the stronger situation:

\begin{lemma} If $\Phi : (\NN \to \NN) \to \NN$ is a majorizable functional (in the sense of Bezem or Howard), then $$\forall h \exists r \forall f (\forall k\leq r (fk \leq hk) \to \Phi(f) \leq r).$$ \end{lemma}

\begin{proof} Let $\Phi : (\NN \to \NN) \to \NN$ be a majorizable functional. Take $\Lambda$ a functional of type $(\NN \to \NN) \to \NN$ such that $\Phi \Bmaj_{(\NN \to \NN) \to \NN} \Lambda$. We argue by contradiction. \\[1mm]
Suppose that there is $h:\NN \to \NN$ such that 
$$\forall r \exists f (\forall k\leq r (fk \leq hk) \wedge \Phi(f) > r).$$
So, given a natural number $r$, take $f_r: \NN \to \NN$ such that $\forall k\leq r \,(f_r(k) \leq h(k))$ and $\Phi(f_r) > r$. \\[1mm]
Define $\tilde{h}: \NN \to \NN$ by letting $\tilde{h}(k) := \max(h(k), \max_{r < k} f_r(k))$. We claim that, for all natural numbers $r$ and $k$, $f_r(k) \leq \tilde{h}(k)$. To see this, fix $r \in \NN$. If $k \leq r$, then $f_r (k) \leq h(k) \leq \tilde{h}(k)$. If $k > r$, we have immediately $f_r (k) \leq \tilde{h}(k)$. \\[1mm]
By the above, we may conclude that, for each natural number $r$, $f_r \Bmaj_{\NN \to \NN} \tilde{h}^M$. Let $\ell : = \Lambda (\tilde{h}^M)$. We get, $\ell < \Phi (f_\ell) \leq \Lambda (\tilde{h}^M) = \ell$. This is a contradiction. \end{proof}

Recall, from Subsection \ref{maj_upon_bounding-section}, that $\mathcal{B}_{(\NN \to \NN) \to \NN}$ consists of the set-theoretic functionals of type $(\NN \to \NN) \to \NN$ which are majorizable in the sense of Bezem.

\begin{proposition} $\mathcal{C}^\leq_{(\NN \to \NN) \to \NN} = \mathcal{B}_{(\NN \to \NN) \to \NN}$. \end{proposition}

\begin{proof} It remains to show that if $\Phi$ is a majorizable functional, then $\Phi$ is in $\mathcal{C}^\leq_{(\NN \to \NN) \to \NN}$. By the above lemma $$\forall h \exists r \forall f (\forall k\leq r (fk \leq hk) \to \Phi f \leq r).$$ Let $\Psi(h)$ be the least $r$ above and $\mu(h) = \{0,1,\ldots,\Psi(h)\}$. Clearly, we have $$\forall h \forall f (\forall k \in \mu h \forall n\leq k (fn \leq hk) \to \Phi f \leq \Psi h).$$ To see this, just consider $n=k$ above. So, we have (\ref{Cleq}). We are done. \end{proof}

We do not know much about the typestructure $\mathcal{C}^\leq$ in higher types, nor how it relates to known models of G\"odel's $\GodelT$.

\section{On the uniformity of restricted quantifiers}
\label{uniform-boundedness-section}

Quantifiers are treated uniformly in the background $\U$-interpretation. In general, this is not the case for the $\I$-interpretation because of the computational meaning of the type-informative predicates. However, in certain situations, the $\I$-interpretation of the quantifiers can be treated uniformly (or semi-uniformly). When this happens, some surprising consequences ensue. We refer here to the $\I$-interpretability of certain principles which are not set-theoretically true (collection and contra-collection principles). We discuss these principles abstractly in Subsection \ref{semi-uniform-section}. Applications are discussed in Subsection \ref{analytic-applications-section}.

\begin{definition}[$B$-restricted quantifiers] For a formula $B(x^\sigma, \pvec y)$, with a distinguished free variable $x^\sigma$, we can introduce in the source theory $\SourceT$ primitive \emph{restricted} (or \emph{bounded}) quantifiers $\forall x \Bbounded_\sigma B(\pvec y) \, A(x)$ and $\exists x \Bbounded_\sigma B(\pvec y) \, A(x)$, with axioms
\begin{align}
\forall x \Bbounded_\sigma B(\pvec y) \, A(x) & 
        \; \leftrightarrow \; \forall x^\sigma (B(x, \pvec y) \to A(x)) \label{B-bounded-forall} \\
    \exists x \Bbounded_\sigma B(\pvec y) \, A(x) & 
        \; \leftrightarrow \; \exists x^\sigma (B(x, \pvec y) \wedge A(x)), \label{B-bounded-exists}
\end{align}
where $A(x)$ can have parameters, including from $\pvec y$.
\end{definition}

Note that these $B$-restricted quantifiers are nothing more than relativized quantifiers with respect to the predicate $B(x^\sigma, \pvec y)$. We are interested in the cases where the relativization to the $B$-predicate is both ``computationally advantageous'' and subsumes the corresponding $\I$-predicate. Intuitively, the latter should happen when $B(x^\sigma, \pvec y)$ is ``stronger'' than $\witness{x^\sigma}{\sigma}$.

\begin{definition}[$B$ subsumes $\I$] Given a base interpretation of $\LangI{\SourceT}$, we say that a formula $B(x^\sigma, \pvec y^{\pvec \tau})$ \emph{subsumes $\witness{x}{\sigma}$} if the sentence
\begin{equation} \label{subsumes-fml}
    \forall \pvec y^{\pvec \tau} (\witness{\pvec y}{\pvec \tau} \to \forall x^\sigma (\rel{B(x, \pvec y)} \to \witness{x}{\sigma}))
\end{equation} 
is $\U$-interpretable in $\TargetT$.
\end{definition}

\begin{example} \label{ex-subsumes-1} In bounding interpretations of arithmetic, where $\uInter{\witness{n}{\NN}}{m}{}$ is $n \leq m$, the formula $B(n^\NN, y^\NN) :\equiv n \leq_\NN y$ subsumes $\witness{n}{\NN}$. In effect, the $\U$-interpretation of
\[
    \forall y^{\NN} (\witness{y}{\NN} \to \forall n^\NN (n \leq y \to \witness{n}{\NN}))
\]
asks for a term $t \colon \NN \to \NN$ such that
\[
    \forall y^{\NN}, k^\NN (y \leq k \to \forall n^\NN (n \leq y \to n \leq tk)).
\]  
Indeed, we can take $tk = k$.

\end{example}

\begin{example} \label{ex-subsumes-2} With the base interpretation of the $\Icb$-interpretation of arithmetic, we claim that the formula $B(f^{\NN \to \NN}) : \equiv \forall n^\NN (f n \leq_{\NN} 1)$ subsumes $\witness{f^{\NN \to \NN}}{\NN \to \NN}$. To see this, note that the $\U$-interpretation of (\ref{subsumes-fml}) asks, in this case, for closed terms $t^{\NN \to \NN}$ and $q^{\NN \to \NN^*}$, such that
\[
\forall k^\NN \forall f^{\NN \to \NN} (\forall n \in qk \forall i\leq n (f i \leq 1) \to \forall i \leq k (f i \leq t k)).
\]
We take $t$ and $q$ such that $tk = 1$ and $qk = \{k\}$.

\end{example}

\begin{remark}\label{obstruction} The above example works because we are using the canonical upon the bounding interpretation $\Icb$. If we were using instead the majorizability upon the bounding interpretation $\Imb$, then the $\U$-interpretation of (\ref{subsumes-fml}) would ask for closed terms $t^{\NN\to \NN}$ and $q^{\NN^*}$, such that
\[
\forall f^{\NN\to\NN} (\forall n \in q \forall i \leq n (fi \leq 1) \to f \Bmaj_{\NN \to \NN} t).
\]
Clearly, we cannot produce such witnesses (no finite set $q$ will do).
\end{remark}

\begin{example}\label{notable-case-1} An unrestricted quantification is, in fact, a special case of a $B$-restricted quantification, when $B(x^\sigma)$ is $\exists z^\sigma (z \equiv_\sigma x)$. Moreover, $B(x^\sigma)$ subsumes $\witness{x}{\sigma}$. To see this, just observe that the formula $\forall x^\sigma (\rel{B(x)} \to \witness{x}{\sigma})$ is $\U$-interpretable simply because, by definition, it is $$\forall x^\sigma (\exists z^\sigma (\witness{z}{\sigma} \wedge z \equiv_\sigma x) \to \witness{x}{\sigma}).$$ This formula is provable and, consequently, $\U$-interpretable. With this example, it is clear that restricted quantifiers are not always interesting. They are of interest only when the restriction to $B(x^\sigma)$ brings some computational advantages. \end{example}

\begin{proposition} \label{prop-rel-extended} Suppose that the source theory $\SourceT$ is extended to include primitive $B$-restricted quantifiers with axioms (\ref{B-bounded-forall}) and (\ref{B-bounded-exists}). Let $\SourceTI$ be the auxiliary theory of the proof of Theorem \ref{main-soundness} together with (\ref{subsumes-fml}) as a new axiom. Then, we can  extend the relativization procedure of Definition \ref{def-relativization} as follows:
\eqleft{
\begin{array}{lcl}
	\rel{(\forall x \Bbounded_\sigma B(\pvec y) \, A(x))} & :\equiv & \forall x^\sigma (\rel{B(x^\sigma, \pvec y)} \to \rel{A(x)}) \\[2mm]
	\rel{(\exists x \Bbounded_\sigma B(\pvec y) \, A(x))}& :\equiv & \exists x^\sigma (\rel{B(x^\sigma, \pvec y)} \wedge \rel{A(x)}).
\end{array}	
}
With this definition, we still have that if \,$\Gamma(\pvec x^{\pvec \sigma}) \proves_{\SourceT} A(\pvec x^{\pvec \sigma})$, then 
$$\witness{\pvec x}{\pvec \sigma}, \rel{\Gamma(\pvec x)} \proves \rel{A(\pvec x)},$$ 
where the proof takes place in the auxiliary theory with the new axiom (\ref{subsumes-fml}).
\end{proposition}

\begin{proof} It is enough to see that, in the presence of (\ref{subsumes-fml}), the relativizations of (\ref{B-bounded-exists}) and (\ref{B-bounded-forall}) are logically trivial. Consider the case of $\forall x \Bbounded_\sigma B(\pvec y) \, A(x)$. By definition, its relativization is
\[ \forall x^\sigma (\rel{B(x^\sigma, \pvec y)} \to \rel{A(x)}). \]
On the other hand, the relativization of the right-hand side of (\ref{B-bounded-forall}) is
\[ \forall x^\sigma (\witness{x^\sigma}{\sigma} \wedge \rel{B(x, \pvec y)} \to \rel{A(x)}). \]
Of course, in the presence of the new axiom (\ref{subsumes-fml}), they are equivalent.
\end{proof}

\begin{remark}
    The language $\LangI{\SourceT}$ of the extended auxiliary theory $\SourceTI$ does not change, since only a new axiom is added to the theory. In particular, the language of the auxiliary theory does not include bounded quantifications. 
\end{remark}

As we will see, whenever we can $\U$-interpret that $B(x^\sigma, \pvec y)$ is ``stronger'' than $\witness{x}{\sigma}$, we can simplify the $\I$-relativization of the $B$-restricted quantifiers. Hence, simpler $\I$-interpretations of $\exists x \Bbounded_\sigma B(\pvec y) \, A(x)$ and $\forall x \Bbounded_\sigma B(\pvec y) \, A(x)$ are obtained. Let us formalize this by extending Proposition \ref{I-interpretation-prop}.

\begin{proposition}[Inductive presentation of $\iInter{A}{\pvec a}{\pvec b}$ -- Extended] \label{I-interpretation-prop-ext} Assuming that the relativization of the $B$-restricted quantifiers is as in Proposition \ref{prop-rel-extended}, we have that
\[
\begin{array}{lcl}
	\iInter{\forall x \Bbounded_{\sigma} B(\pvec y) \, A(x)}{\pvec f, \pvec g}{\pvec u, \pvec b} & \mathrm{\,is\,} & \forall x^\sigma (\forall \pvec v \in \pvec g \pvec u \pvec b \, \iInter{B(x, \pvec y)}{\pvec u}{\pvec v} \, \to \iInter{A(x)}{\pvec f \pvec u}{\pvec b}) \\[2mm]
	\iInter{\exists x \Bbounded_{\sigma} B(\pvec y) A(x)}{\pvec u, \pvec a}{\pvec V, \pvec U} & \mathrm{\,is\,} & \exists x^\sigma (\forall \pvec v \in \pvec V \, \iInter{B(x, \pvec y)}{\pvec u}{\pvec v} \wedge \forall \pvec b \in \pvec U \, \iInter{A(x)}{\pvec a}{\pvec b}).
\end{array}	
\]
\end{proposition}

The point of the above clauses for the restricted quantifiers is the fact that the type-informative predicates do not show up in the quantifications $\forall x^\sigma$ and $\exists x^\sigma$ of the right-hand sides. We are interested in restricted quantifications which are computationally weaker than the corresponding unrestricted quantifications. The following are the interesting cases:

\begin{definition}[Uniform and semi-uniform restricted quantifiers] \label{def-bounded-quant} Let be given a base interpretation of $\LangI{\SourceT}$, and let $B(x^\sigma, \pvec y^{\pvec \tau})$ be a formula with a distinguished variable $x^\sigma$ which subsumes $\witness{x}{\sigma}$. We say that a $B$-restricted quantifier is \emph{semi-uniform} if the formula $B(x^\sigma, \pvec y^{\pvec \tau})$ has no positive $\I$-witnesses. We say that the $B$-restricted quantifier is \emph{uniform} if the formula $B(x^\sigma, \pvec y^{\pvec \tau})$ is $\I$-witness-free. \end{definition}

See Notation \ref{positive-negative-I-witnesses} for the terminology. If the $B$-restricted quantifiers are semi-uniform or uniform, the clauses of Proposition \ref{I-interpretation-prop-ext} simplify. Since this has important consequences, we dedicate Subsection \ref{uniform-section} and Subsection \ref{semi-uniform-section} to this issue.

Let us conclude by showing that the simplification of the $\I$-relativization of the $B$-restricted quantifiers still leads to a sound interpretation of the source theory $\SourceT$ (with the primitive $B$-restricted quantifiers). 

\begin{theorem}[Main Soundness Theorem -- Extended] \label{main-soundness-ext} Let be given a $\TargetT$-sound type-informative base interpretation of $\SourceT$. Suppose that the language of $\SourceT$ is extended with primitive $B$-restricted quantifiers for formulas $B(x^\sigma, \pvec y^{\pvec \tau})$ which subsume $\witness{x^\sigma}{\sigma}$ and extend the source theory $\SourceT$ with the corresponding axioms (\ref{B-bounded-forall}) and (\ref{B-bounded-exists}). If 
\[
    \Gamma(\pvec z^{\pvec \rho}) 
    \proves_{\SourceT} 
    A(\pvec z^{\pvec \rho})
\]
then there are tuples of closed witnessing terms $\pvec q, \pvec s$ and $\pvec t$ such that the target theory $\TargetT$ proves the following: for all $\pvec a$, $\pvec c$ and $\pvec d$ and for all $\pvec z^{\pvec \rho}$, 
$$\mbox{if \,}\forall \pvec e \in \pvec q \pvec a \pvec c \pvec d \, \iInter{\witness{\pvec z}{\pvec \rho}}{\pvec c}{\pvec e} \mbox{\, and \,} \forall \pvec b \in \pvec s \pvec a \pvec c \pvec d \iInter{\Gamma(\pvec z)}{\pvec a}{\pvec b} \mbox{,\, then } \iInter{A(\pvec z)}{\pvec t \pvec a \pvec c}{\pvec d}.$$
\end{theorem}

\begin{proof} We follow the proof of Theorem \ref{main-soundness}. By Proposition \ref{prop-rel-extended}, it only remains to show that the new axiom (\ref{subsumes-fml}) of $\SourceTI$ is $\U$-interpretable. But this follows from our assumption that $B(x^\sigma, \pvec y)$ subsumes $\witness{x^\sigma}{\sigma}$.
\end{proof}

\subsection{Uniform restricted quantifiers}
\label{uniform-section}

If the $B$-restricted quantifiers are uniform, their inductive clauses turn out to be very simple:
\[
\begin{array}{lcl}
	\iInter{\forall x \Bbounded_{\sigma} B(\pvec y) A(x)}{\pvec a}{\pvec b} & \mathrm{\,is\,} & \forall x^\sigma (\iInter{B(x, \pvec y)}{}{} \, \to \iInter{A(x)}{\pvec a}{\pvec b}) \\[2mm]
	\iInter{\exists x \Bbounded_{\sigma} B(\pvec y) A(x)}{\pvec a}{\pvec W} & \mathrm{\,is\,} & \exists x^\sigma (\iInter{B(x, \pvec y)}{}{} \wedge \, \forall \pvec b \in \pvec W \iInter{A(x)}{\pvec a}{\pvec b}).
\end{array}	
\]

In the setting of the bounding interpretations of arithmetic (where one has the bounding interpretation at the ground type $\NN$), first-order bounded quantifications are a natural example of restricted $B$-quantifications which are uniform. Their analysis stems from Example \ref{ex-subsumes-1}, where $B(n,y)$ is the formula $n\leq_\NN y$. Bounded first-order quantifications of the form $\forall x \leq_\NN t(\pvec z)\, A(x, \pvec z)$ and $\exists x \leq_\NN t(\pvec z)\, A(x, \pvec z)$, where $t$ is a term of type $\NN$ in which the variable $x^\NN$ does not occur, are explained as the restricted quantifications $$\forall x \Bbounded_\NN B(t(\pvec z)) \, A(x, \pvec z) \mbox{\,\,\, and \,\,\,} \exists x \Bbounded_\NN B(t(\pvec z)) \, A(x, \pvec z),$$ respectively. Ignoring parameters, by (\ref{B-bounded-forall}) and (\ref{B-bounded-exists}), we have the familiar
\eqleft{
\begin{array}{lcl}
    \forall x \leq t \,A(x)
        & \leftrightarrow & \forall x \,(x\leq t \to A(x)) \\[1mm]
    \exists x \leq t \,A(x)
        & \leftrightarrow & \exists x \, (x\leq t \wedge A(x)).
\end{array}
}
Given that the $\I$-interpretation of the formula $n\leq_\NN y$ is witness-free (since, technically, it is stated as an equality of type $\NN$), the clauses for $\I$-interpretability are:
\eqleft{
\begin{array}{lcl}
		\iInter{\forall x\leq t\, A(x)}{\pvec a}{\pvec b} & :\equiv & \forall x \leq t\, \iInter{A(x)}{\pvec a}{\pvec b} \\[1mm]
		\iInter{\exists x \leq t \,A(x)}{\pvec a}{\pvec B} & :\equiv & \exists x \leq t \forall \pvec b \in \pvec B \, \iInter{A(x)}{\pvec a}{\pvec b},
	\end{array}	
}
where, on the right-hand side, we are using the bounded quantifiers informally -- as the usual abbreviations -- in the target language $\Lang{\TargetT}$.

Within the general framework of bounding interpretations of arithmetic, it is possible to discuss and prove the interpretability of a version of Markov's principle. This version naturally involves a first-order bounded quantification and this is one of the reasons why we chose to discuss it. The other reason for the inclusion of Markov's principle is the comment after Proposition \ref{ub_cb}. 

\begin{proposition}[Markov's principle]\label{Markov} Let be given an extension of $\HAomega$ and a base interpretation which uses the bounding interpretation at the ground type $\NN$. The principle 
\begin{equation}
    (\forall n^\NN A(n) \to B) \to \exists k^\NN (\forall n\leq k \, A(n) \to B)
\end{equation}
is $\I$-interpretable provided that the formula $A$ has no positive $\I$-witnesses and the formula $B$ has no negative $\I$-witnesses. \end{proposition}

\begin{proof} The $\I$-interpretation of the antecedent of the principle is $$\forall r\in R \forall \pvec b \in \pvec B \forall n \leq r \iInter{A(n)}{}{\pvec b} \to \iInter{B}{\pvec c}{}$$ where $R^{\NN^*}$, $\pvec B$ and $\pvec c$ are the positive $\I$-witnesses (there are no negative $\I$-witnesses). On the other hand, the $\I$-interpretation of the consequent is $$\exists k \leq m (\forall \pvec b \in \pvec B \forall n\leq k \iInter{A(n)}{}{\pvec b} \to \iInter{B}{\pvec c}{})$$ where $m^\NN$, $\pvec B$ and $\pvec c$ are the positive $\I$-witnesses (there are no negative $\I$-witnesses).
Therefore, we must exhibit terms $t$, $\pvec q$ and $\pvec s$ such that the target theory shows, for all $R$, $\pvec B$ and $\pvec c$, that we have the implication $$(\forall r\in R \forall \pvec b \in \pvec B \forall n \leq r \iInter{A(n)}{}{\pvec b} \to \iInter{B}{\pvec c}{}) \to $$ $$\exists k \leq t R\pvec B \pvec c \, (\forall \pvec b \in \pvec q R\pvec B \pvec c \, \forall n\leq k \,\iInter{A(n)}{}{\pvec b} \to \iInter{B}{\pvec s R\pvec B \pvec c}{}).$$ We can take $t$, $\pvec q$ and $\pvec s$ such that $t R\pvec B \pvec c =\mathrm{m}(R)$, $\pvec q R\pvec B \pvec c = \pvec B$ and $\pvec s R\pvec B \pvec c = \pvec c$. Recall that $\mathrm{m}(R)$ is the maximum of $R$. It is clear that the consequent is true by choosing $k$ to be $\mathrm{m}(R)$.    \end{proof}

When $B$ is $\bot$, the above principle turns out to be $\neg \forall n  A(n) \to \exists k \neg \forall n\leq k A(n)$. This is, perhaps, more recognizable as a Markov principle.

A second example of restricted quantifiers which are uniform consists of (unrestricted) quantifications over witness-free types (see Notation \ref{type-cases}). A notable case, worth remarking explicitly, concerns quantifications over uniform types (see  Notation \ref{bounded-type} and Definition \ref{Imb-interpretation}) within the framework of the $\Imb$-interpretation. In this case, as we have discussed at the end of Subsection \ref{maj_upon_bounding-section}, we get the clauses of the uniform quantification.

Let us see that the situation described in the beginning of the above paragraph fits Definition \ref{def-bounded-quant}. We saw in Example \ref{notable-case-1} that an unrestricted quantification over a type $\sigma$ is a special case of a $B$-restricted quantification. Such is the case with $B(x^\sigma)$ as the formula $\exists z^\sigma (z \equiv_\sigma x)$. Within this framework, one has 
$$\iInter{B(x^\sigma)}{\pvec c}{\pvec D} :\equiv \exists z^\sigma (\forall \pvec d \in \pvec D \uInter{\witness{z}{\sigma}}{\pvec c}{\pvec d} \,\wedge z = x).$$ 
Therefore, an unrestricted quantification over a type $\sigma$ is uniform (in the sense of Definition \ref{def-bounded-quant}) just in case the type $\sigma$ is witness-free. In particular, this happens when $\witness{x}{\sigma}$ is interpreted uniformly (in the sense of Example \ref{uniform-def}). As an aside, this analysis also shows that an unrestricted quantification is semi-uniform just in case the type $\sigma$ carries no positive witnesses (see Notation \ref{type-cases}).

A last example of uniform restricted quantifiers concerns the intensional bounded quantifications of the bounded functional interpretation, as they were introduced by Ferreira and Oliva in \cite{FerreiraOliva(05)}. We dedicate Subsection \ref{bfi-section} to this example.

\subsection{Semi-uniform restricted quantifiers}
\label{semi-uniform-section}

In this subsection, we consider the case where the $B$-restricted quantifiers are semi-uniform. Looking at Proposition \ref{I-interpretation-prop-ext}, it is clear that the clauses for the inductive presentation of the restricted $B$-quantifiers are, in this (semi-uniform) case, the following: 
\[
\begin{array}{lcl}
	\iInter{\forall x \Bbounded_{\sigma} B(\pvec y) \, A(x)}{\pvec a, \pvec g}{\pvec b} & \mathrm{\,is\,} & \forall x^\sigma (\forall \pvec v \in \pvec g \pvec b \, \iInter{B(x, \pvec y)}{}{\pvec v} \, \to \iInter{A(x)}{\pvec a}{\pvec b}) \\[2mm]
	\iInter{\exists x \Bbounded_{\sigma} B(\pvec y) A(x)}{\pvec a}{\pvec V, \pvec W} & \mathrm{\,is\,} & \exists x^\sigma (\forall \pvec v \in \pvec V \, \iInter{B(x, \pvec y)}{}{\pvec v} \wedge \, \forall \pvec b \in \pvec W \iInter{A(x)}{\pvec a}{\pvec b}).
\end{array}	
\]

In the remainder of this subsection, we assume that $\SourceT$ is an extension of $\HAomega$ and that our base interpretation uses the bounding interpretation at the ground type $\NN$. 

\begin{theorem}[Collection] \label{collection-thm} Suppose that the $B$-restricted quantifier is semi-uniform and that $C(x,n^\NN)$ is an arbitrary formula of the source language. In this case, the following principle
\begin{equation}\label{collection-fml}
    \forall x \Bbounded_\sigma B(\pvec y) \, \exists n^\NN \, C(x,n) \to \exists k^\NN \forall x \Bbounded_\sigma B(\pvec y) \, \exists n\leq k \, C(x,n)
\end{equation}
is $\I$-interpretable.
\end{theorem}

\begin{proof} We must show that the relativization of (\ref{collection-fml}) is $\U$-interpretable, i.e., that the implication $$\hspace{-15mm}\forall x (\rel{B(x,\pvec y)} \to \exists n (\witness{n}{\NN} \wedge \rel{C(x,n)})) \to $$ $$\hspace{25mm}\exists k \, (\witness{k}{\NN} \wedge \forall x (\rel{B(x,\pvec y)} \to \exists n \leq k \, \rel{C(x,n)}))$$ is $\U$-interpretable. To simplify the discussion, we ignore the parameters $\pvec y$. The $\U$-interpretation of the antecedent above asks for witnesses $r^\NN$, $\pvec c$ and $\pvec g$ such that, for all $\pvec D$, $$\forall x \, (\forall \pvec v \in \pvec g \pvec D \iInter{B(x)}{}{\pvec v} \to \exists n\leq r \forall \pvec d \in \pvec D \iInter{C(x,n)}{\pvec c}{\pvec d}).$$ On the other hand, the $\U$-interpretation of the consequent above asks for witnesses $r^\NN$, $\pvec c$ and $\pvec g$ such that, for all $\mathfrak{D}$, $$\exists k \leq r \forall \pvec D \in \mathfrak{D} \forall x (\forall \pvec v \in \pvec g \pvec D \iInter{B(x)}{}{\pvec v} \to \exists n \leq k \forall \pvec d \in \pvec D \iInter{C(x,n)}{\pvec c}{\pvec d}).$$ For the desired $\U$-interpretation of the implication, we must find terms $t$, $\pvec{p}$, $\pvec{q}$ and $\pvec{s}$ such that the target theory proves the following: for all $r$, $\pvec c$, $\pvec g$ and $\mathfrak{D}$, $$\forall \pvec D \in \pvec s r \pvec c \pvec g \mathfrak{D} \, \forall x \, (\forall \pvec v \in \pvec g \pvec D \iInter{B(x)}{}{\pvec v} \to \exists n\leq r \forall \pvec d \in \pvec D \iInter{C(x,n)}{\pvec c}{\pvec d}) \to $$ $$\exists k\leq t r \pvec c \pvec g \forall {\pvec D} \in \mathfrak{D} \, \forall x \, (\forall \pvec v \in \pvec p r \pvec c \pvec g \pvec D \iInter{B(x)}{}{\pvec v} \to \exists n \leq k \forall \pvec d \in \pvec D \iInter{C(x,n)}{\pvec q r \pvec c \pvec g}{\pvec d}).$$ We can just put $t r \pvec c \pvec g = r$, $\pvec p r \pvec c \pvec g = \pvec g$, $\pvec q r \pvec c \pvec g = \pvec c$ and $\pvec s r \pvec c \pvec g \mathfrak{D} = \mathfrak{D}$. The consequent holds by taking $k$ to be $r$.  \end{proof}

The above collection principle has some affinities with the FAN theorem of Brouwerian intuitionistic mathematics. The following is a partial converse, a kind of idealization principle:

\begin{theorem}[Contra-collection]\label{contra-collection-thm} Suppose that the $B$-restricted quantifier is semi-uniform and assume that $C(x,n^\NN)$ is a formula of the source language with no positive  $\I$-witnesses. In this case, the following principle
\begin{equation}\label{contra-collection-fml}
    \forall k^\NN \exists x \Bbounded_\sigma B(\pvec y) \, \forall n\leq k \, C(x,n) \to \exists x \Bbounded_\sigma B(\pvec y) \, \forall n^\NN C(x,n)
\end{equation}
is $\I$-interpretable. \end{theorem}

\begin{proof} We must show that the relativization of (\ref{contra-collection-fml}) is $\U$-interpretable, i.e., that the implication $$\hspace{-5mm} \forall k^\NN (\witness{k}{\NN} \to \exists x \, (\rel{B(x,\pvec y)} \wedge \forall n\leq k \, \rel{C(x,n)})) \to$$ $$\hspace{35mm} \exists x \, (\rel{B(x,\pvec y)} \wedge \forall n^\NN (\witness{n}{\NN} \to \rel{C(x,n)}))$$ is $\U$-interpretable. Let us ignore the parameters $\pvec y$. The $\U$-interpretation of the antecedent above does not ask for (positive) witnesses, but the following must hold for all $r^\NN$, $\pvec V$ and $\pvec D$: $$\forall k \leq r \exists x\, (\forall \pvec v \in \pvec V \iInter{B(x)}{}{\pvec v} \wedge \forall \pvec d \in \pvec D \forall n\leq k \, \iInter{C(x,n)}{}{\pvec d}).$$ The $\U$-interpretation of the consequent does not ask for (positive) witnesses either. Only that the following must hold for all $R^{\NN^*}$, $\pvec V$ and $\pvec D$: $$\exists x \, (\forall \pvec v \in \pvec V \iInter{B(x)}{}{\pvec v} \wedge \forall r \in R \, \forall \pvec d \in \pvec D \, \forall n \leq r \, \iInter{C(x,n)}{}{\pvec d}).$$ 
For the $\U$-interpretation of the implication, we must find terms $t$, $\pvec p$ and $\pvec q$ such that the target theory proves the following: for all $R^{\NN^*}$, $\pvec V$ and $\pvec D$, if for all $r \in t R \pvec V \pvec D$, $\pvec V' \in \pvec p R \pvec V \pvec D$ and $\pvec D' \in \pvec q R \pvec V \pvec D$ one has $$\forall k \leq r \exists x \, (\forall \pvec v \in \pvec V' \iInter{B(x)}{}{\pvec v} \wedge \forall \pvec d \in \pvec D' \forall n\leq k \iInter{C(x,n)}{}{\pvec d}),$$ then $\exists x \, (\forall \pvec v \in \pvec V \iInter{B(x)}{}{\pvec v} \wedge \forall r \in R \, \forall \pvec d \in \pvec D \,\forall n\leq r \iInter{C(x,n)}{}{\pvec d})$. \\[1mm]
We can just put $t R \pvec V \pvec D = R$, $\pvec p R \pvec V \pvec D = \{\pvec V\}$ and $\pvec q R \pvec V \pvec D = \{\pvec D\}$. The element $x$ that works for the consequent is the one that is obtained from the antecedent by taking $k$ the maximum of the elements of $R$. \end{proof}

\subsection{Arithmetical and logical applications}
\label{arithmetical-applications-section}

As in the previous subsections, we assume that our source theory $\SourceT$ is an extension of $\HAomega$ and that our base interpretation uses the \emph{bounding interpretation at the ground type $\NN$}. 

\begin{theorem}\label{numerical_collection} The following collection principle is $\I$-interpretable:
$$\forall n \leq_\NN y \, \exists m^\NN C(n,m) \to \exists k^\NN \, \forall n \leq_\NN y \, \exists m\leq k \, C(n,m),$$ for arbitrary formulas $C(n,m)$ of the source language. \end{theorem}

\begin{proof} This is an immediate consequence of Theorem \ref{collection-thm}, using Example \ref{ex-subsumes-1}. \end{proof}

The above collection principle is, of course, provable by induction on $y$. However, the argument given above shows that it has empty computational content. It can, in fact, be added to theories of weak induction (where the inductive argument fails) without perturbing their provably total functions. As far as we know, this was first observed by Wolfgang Burr in \cite{Burr(00B)}. 

\begin{theorem}\label{numerical_contra_collection} The following contra-collection principle is $\I$-interpretable:
$$\forall k^\NN \exists n \leq_\NN y \, \forall m \leq k \, C(n,m) \to \exists n \leq_\NN y \, \forall m \, C(n,m),$$ 
provided that $C(n,m)$ has no positive $\I$-witnesses. \end{theorem}

\begin{proof} This is an immediate consequence of Theorem \ref{contra-collection-thm}, using Example \ref{ex-subsumes-1}. \end{proof}

This kind of contra-collection was first explicitly discussed in \cite{FerreiraOliva(07)} (see, for more discussions, \cite{Ferreira(20)}). By the same token, its computational content is nil. It was observed in \cite{FerreiraOliva(07)} that the above principle entails a form of the lesser limited principle of omniscience $\mathsf{LLPO}$ of Errett Bishop \cite{Bishop(85)}. However, since we have the Boolean type in our language, we opt for a direct argument using this type:

\begin{theorem}\label{llpo} The following version of $\mathsf{LLPO}$ is $\I$-interpretable:
$$\forall k^\NN (\forall m \leq k A(m) \vee \forall m \leq k B(m)) \to \forall m A(m) \vee \forall m B(m),$$ 
provided that both formulas $A(m)$ and $B(m)$ have no positive $\I$-witnesses. \end{theorem}

\begin{proof} Assume that $\forall k^\NN (\forall m \leq k A(m) \vee \forall m \leq k B(m))$. Therefore, $$ \forall k^\NN \exists i^\BB \forall m \leq k ((i\equiv_\BB 0 \to A(m)) \wedge (i \equiv_\BB 1 \to B(m))).$$ Since the type $\BB$ is interpreted uniformly in bounding interpretations of arithmetic, it is clear that we can apply Theorem \ref{contra-collection-thm} and conclude that $$\exists j^\BB \forall m ((j\equiv_\BB 0 \to A(m)) \wedge (j\equiv_\BB 1 \to B(m))).$$ This is a reformulation of what we want. \end{proof}

\subsection{Analytic applications}
\label{analytic-applications-section}

We consider a concrete example of $B$-restricted quantifiers in $\HAomega$, namely when $B(f^{\NN \to \NN})$ is the formula $\forall n^\NN (f n \leq_{\NN} 1)$. As a matter of convenient notation, in this case we write $\forall f \leq_{\NN \to \NN} \!1 \, A(f)$ instead of the cumbersome $\forall f \Bbounded_{\NN \to \NN}\!B \, A(f)$. By axiom (\ref{B-bounded-forall}), this is equivalent to $\forall f^{\NN \to \NN} (\forall n^\NN (f n \leq_{\NN} 1) \to A(f))$. We use a similar notation for the existential bounded quantifier.

We begin our discussion with a uniform boundedness principle (in the terminology of \cite{Kohlenbach(08)}).

\begin{theorem}\label{ub_cb} The following collection principle is $\Icb$-interpretable:
$$\forall f \leq_{\NN \to \NN} 1 \, \exists n^\NN C(f,n) \to \exists k^\NN \, \forall f \leq_{\NN \to \NN} 1 \, \exists n\leq k \, C(f,n),$$ for arbitrary formulas $C(f,n)$ of the source language. \end{theorem}

\begin{proof} This is an immediate consequence of Theorem \ref{collection-thm}, using Example \ref{ex-subsumes-2}. \end{proof}
 
Note that Proposition \ref{Markov} and Theorem \ref{ub_cb} imply that the $\Icb$-interpretation is able to interpret both a form of Markov's principle and the above form of collection. This only works because we are using the bounding interpretation at the ground type $\NN$ together with the canonical interpretation of the function types. The result, however, does not seem to hold for the $\Imb$-interpretation: see Remark \ref{obstruction} for the obstruction. 

Similarly, using Theorem \ref{contra-collection-thm}, we have a novel contra-collection principle:

\begin{theorem} \label{ideal-canonical} The following contra-collection principle is $\Icb$-interpretable
$$\forall k^\NN \exists f \leq_{\NN \to \NN} 1 \forall n\leq k \, C(f,n) \to \exists f \leq_{\NN \to \NN} 1 \forall n^\NN C(f,n)$$ 
provided $C(f,n)$ has no positive $\I$-witnesses. \end{theorem}

Let us look at an alternative way of dealing with quantifications restricted to the Cantor space within the $\Imb$-interpretation. Instead of considering functions $f$ from $\NN$ to $\NN$ such that $\forall i^\NN (fi\leq 1)$, we consider functions $f$ from $\NN$ to the Booleans $\BB$. An ordinary mathematician would say that we are presenting the Cantor space in two different manners, and that this should not matter. However, it does matter. The like-minded reader knows that there is an extensionality issue operating behind the scenes. We review this issue below.

\begin{proposition}\label{ub-majorizability} The following sentence is $\Imb$-interpretable 
$$\forall f^{\NN \to \BB} \, \exists n^\NN C(f,n) \to \exists k^\NN \, \forall f^{\NN \to \BB} \, \exists n\leq k \, C(f,n)$$
for arbitrary formulas $C(f, n)$ of the source language.
\end{proposition}

\begin{proof} This is a consequence of Theorem \ref{collection-thm}, namely of the notable case described in the second example of Subsection \ref{uniform-section}: we only have to observe that $\NN \to \BB$ is a bounded type (see Notation \ref{bounded-type} and Definition \ref{Imb-interpretation}). Therefore, we are in the uniform case. 
\end{proof}

Similarly,

\begin{proposition}\label{ideal-boolean} The following sentence is $\Imb$-interpretable  
$$\forall k^\NN \exists f^{\NN \to \BB} \forall n\leq k \, C(f,n) \to \exists f^{\NN \to \BB} \forall n^\NN C(f,n)$$ 
provided that $C(f,n)$ has no positive $\I$-witnesses. \end{proposition}

Let $\mathfrak{a}: \BB \to \NN$ be such that $\mathfrak{a} 0^\BB = 0^\NN$ and $\mathfrak{a} 1^\BB = 1^\NN$. Such $\mathfrak{a}$ can be defined as $\lambda i^\BB. \mathrm{Cond}_\NN (i,0^\NN, 1^\NN)$. Let $\mathfrak{b}: \NN \to \BB$ be such that $\mathfrak{b} 0^\NN = 0^\BB$ and $\mathfrak{b} 1^\NN = 1^\BB$. We can take $\mathfrak{b}$ as $\lambda n^\NN. \Rec_\BB n0^\BB (\lambda m^\NN,r^\BB. 1^\BB)$. 

Do look at the following tentative argument. Let $C(f^{\NN \to \NN},n^\NN)$ be a formula of the language of $\HAomega$, and suppose that $\forall f\leq_{\NN \to \NN} 1 \, \exists n \, C(f,n)$. Clearly, we obtain $$\forall g^{\NN \to \BB} \exists n \, C(\lambda r^\NN. \mathfrak{a}(gr),n).$$ By Proposition \ref{ub-majorizability}, there is $k^\NN$ such that $\forall g^{\NN \to \BB} \exists n \leq k\, C(\lambda r^\NN. \mathfrak{a}(gr),n)$. Therefore $$\forall f \leq_{\NN \to \NN} 1 \exists n \leq k \, C(\lambda r^\NN.\mathfrak{a}((\lambda m^\NN. \mathfrak{b}(fm))r),n).$$ Given $f^{\NN \to \NN}$ such that $f \leq_{\NN \to \NN} 1$, let $h^{\NN \to \NN}$ be the function $\lambda r^\NN.\mathfrak{a}((\lambda m^\NN. \mathfrak{b}(fm))r)$. For all $r^\NN$,  $hr = \mathfrak{a}((\lambda m^\NN. \mathfrak{b}(fm))r) = \mathfrak{a} (\mathfrak{b} (fr)) = fr$. So, the functions $h$ and $f$ are extensionally the same. However, we cannot conclude $$\forall f \leq_{\NN \to \NN} 1 \exists n \leq k \, C(f,n).$$ The reason is that the $\Imb$-interpretation does not realize the extensionality principle. This comes from an old observation of Howard in \cite{Howard(73)}. We are not able to $\Imb$-interpret $$\forall g^{\NN \to \NN} \forall h^{\NN \to \NN} (\forall r^\NN (gr \equiv_\NN hr) \to g \equiv_{\NN \to \NN} h).$$ That notwithstanding, a formula $C(f)$ may be extensional with respect to $f$, in the sense that $$\forall g^{\NN \to \NN} \forall h^{\NN \to \NN} (\forall r^\NN (gr \equiv_\NN hr) \wedge C(g) \to C(h))$$ is $\Imb$-interpretable. In this case, the argument above can indeed be concluded. This is the case, for instance, when the only way for the formula $C(f)$ to access $f$ is via its values. Let us briefly discuss the notable case of weak K\"onig's lemma.

Consider the formula
\begin{equation}\label{auxiliary_wkl}
    \forall k^\NN \exists f^{\NN \to \BB} \forall n\leq k (T(\overline{f}(n)) \equiv_\NN 0) \to \exists f^{\NN \to \BB} \forall n^\NN (T(\overline{f}(n)) \equiv_\NN 0)
\end{equation}
where $T$ is a parameter of type $\NN \to \NN$ and $\overline{f}(n)$ stands for the (natural number code of the) finite sequence $f0, f1, \ldots, f(n-1)$. (This code can be taken as a standard coding of the natural number binary sequence $\mathfrak{a}(f0), \mathfrak{a}(f1), \ldots, \mathfrak{a}(f(n-1))$.) This formula has the right form for the application of Proposition \ref{ideal-boolean} because, in this case, $C(f,n)$ is $T(\overline{f}(n)) \equiv_\NN 0$. Therefore, (\ref{auxiliary_wkl}) is $\Imb$-interpretable. Obviously, $C(f,n)$ is extensional (with respect to $f$). Therefore, the formula $$\forall k^\NN \exists f \leq_{\NN \to \NN} 1 \forall n\leq k (T(\overline{f}(n)) \equiv_\NN 0) \to \exists f \leq_{\NN \to \NN} 1 \forall n^\NN (T(\overline{f}(n)) \equiv_\NN 0)$$ is also $\Imb$-interpretable. Suppose that $T$ is a binary tree, i.e., suppose that we are working under the assumptions 
\begin{enumerate}
\item[\mbox{\rm{(}}a\mbox{\rm{)}}] $\forall s^\NN (T(s) \equiv_\NN 0 \to \mathrm{Seq}_2(s))$
\item[\mbox{\rm{(}}b\mbox{\rm{)}}] $\forall s^\NN \forall r^\NN (\mathrm{Seq}_2(s) \wedge \mathrm{Seq}_2(r) \wedge T(s) \equiv_\NN 0 \wedge r\preccurlyeq s \to T(r) \equiv_\NN 0)$
\end{enumerate}
where $\mathrm{Seq}_2 (s)$ says that $s$ codes a binary sequence (of natural numbers 0 and 1) and $r\preccurlyeq s$ says that $r$ is an initial subsequence of $s$. In this case, the above interpretable formula entails $$\forall k^\NN \exists s^\NN (\mathrm{Seq}_2 (s) \wedge \mathrm{lh}(s) \equiv_\NN k \wedge T(s) \equiv_\NN 0) \to \exists f \leq_{\NN \to \NN} 1 \forall n^\NN (T(\overline{f}(n)) \equiv_\NN 0),$$ where $\mathrm{lh}(s) \equiv_\NN k$ says that the binary sequence $s$ has length $k$. This is weak K\"onig's lemma. In sum, we have:

\begin{proposition} Weak K\"onig's lemma is $\Imb$-interpretable. \end{proposition}

Of course, by Theorem \ref{ideal-canonical}, weak K\"onig's lemma is also $\Icb$-interpretable. In this case, however, the extensionality considerations that we have made are not needed to obtain the result.

\begin{remark} In this section we have considered the \emph{interpretability} of collection and contra-collection principles, i.e. interpretations which trivialize these principles. Our work on these principles, here and in \cite{FerreiraOliva(05)}, was inspired by Kohlenbach's seminal work on the \emph{elimination} of similar principles (e.g. the uniform boundedness principle, axiom $F$, and so-called $\Delta$-axioms, \cite{Kohlenbach(96), Kohlenbach(98A)}) via the monotone functional interpretation. It is useful to keep in mind the distinction between \emph{interpreting a principle} and \emph{eliminating a principle}. In our approach, we are considering interpretations such that the interpretation of a principle $P$ becomes provable in some target theory. We say that the principle is interpretable. When using the monotone functional interpretation in proof mining, however, one applies the interpretation to a proof of the implication $P \to A$, where $P$ is the principle in question, and $A$ is some mathematical theorem. The monotone functional interpretation extracts from this proof computational information that leads to a proof of $P' \to A'$. The interesting cases are when $P'$ is much weaker than $P$, while $A'$ is stronger than $A$. See, for instance, \cite{Kohlenbach(96)} where one of the many applications considered is the case when $P$ is \emph{weak K\"onig's lemma} and $A$ is such that $A' \to A$. For a textbook on proof mining see \cite{Kohlenbach(08)}, and for recent case studies see \cite{Kohlenbach(19),KohlenbachPischke(23)}.  
\end{remark}

\subsection{On the bounded functional interpretation}
\label{bfi-section}

The bounded functional interpretation was introduced by the present authors in \cite{FerreiraOliva(05)}. The interpretation is based upon Bezem's strong majorizability notion (cf.\ Subsection \ref{star-majorizability}) and admits very general forms of collection and contra-collection. The main features of these general forms can be discussed within the framework of this paper. Crucially, the principles of collection and contra-collection need intensional notions of majorizability in order to be formulated. Our framework nicely supports these notions. 

We base our account on the $\Imb$-interpretation of Subsection \ref{maj_upon_bounding-section}, where $\uwitness{x}{\sigma}{a}{}$ is $x \Bmaj_\sigma a$. Notice that the $\Imb$-interpretation relies on the Diller-Nahm interpretation of contraction (via finite sets), whereas the bounded functional interpretation relies on majorizability. Even though the Diller-Nahm variant introduces some differences, the two interpretations deal with quantifiers in a similar way and share many properties. With this proviso, we discuss the most conspicuous features of the bounded functional interpretation in the light of the $\Imb$-framework.

We present the $\Imb$-interpretation with the notation of \cite{FerreiraOliva(05)}. We write $A_{\Imb} (\pvec a, \pvec b)$ for $\iInter{A}{\pvec a}{\pvec b}$ and write $A^{\Imb}$ for $\existstilde \pvec a \foralltilde \pvec b A_{\Imb} (\pvec a, \pvec b)$, i.e., $\existstilde \pvec a \foralltilde \pvec b \iInter{A}{\pvec a}{\pvec b}$. The tilde notation in the quantifiers $\foralltilde$ and $\existstilde$ means that the quantifications are over monotone elements (cf.\ Definition \ref{star-monotone}). Note that this is the notation used in \cite{FerreiraOliva(05)} and we keep it here to facilitate the comparison with the bounded functional interpretation (in our account of the $\Imb$-interpretation in Subsection \ref{maj_upon_bounding-section} we do not use the tilde notation because the witnessing elements of the $\Imb$-interpretation are, by definition, monotone). We get:
$$
\begin{array}{lcl}
	(A \wedge B)^{\Imb} 
		& \mathrm{\,is\,} & \existstilde \pvec a, \pvec c \foralltilde \pvec b,\pvec d \, (A_{\Imb} (\pvec a, \pvec b) \wedge B_{\Imb} (\pvec c, \pvec d)) \\[1mm]
    (A \vee B)^{\Imb} 
		& \mathrm{\,is\,} & \existstilde \pvec a, \pvec c \,\foralltilde \pvec F,\pvec G \, (\foralltilde \pvec b \in \pvec F A_{\Imb} (\pvec a, \pvec b) \vee \foralltilde \pvec d \in \pvec G B_{\Imb} (\pvec c, \pvec d)) \\[1mm]
	(A \to B)^{\Imb}
		& \mathrm{\,is\,} & \existstilde \pvec f, \pvec g \foralltilde \pvec a, \pvec d \,(\foralltilde \pvec b \in \pvec g \pvec a \pvec d A_{\Imb} (\pvec a, \pvec b) \,\to 
				B_{\Imb} (\pvec f \pvec a, \pvec d)) \\[1mm]
	(\forall x^\sigma A(x))^{\Imb} & \mathrm{\,is\,} & \existstilde \pvec f \foralltilde c, \pvec b \forall x \Bmaj_\sigma \!c \, A(x)_{\Imb} (\pvec f c, \pvec b) \\[1mm]
	(\exists x^\sigma A(x))^{\Imb} & \mathrm{\,is\,} & \existstilde c, \pvec a \foralltilde \pvec B \, \exists x \Bmaj_\sigma \!c \,\foralltilde \pvec b \in \pvec B \, A(x)_{\Imb} (\pvec a, \pvec b),
\end{array}	
$$
where $\sigma$ is an arithmetical type. The above clauses come from Proposition \ref{I-interpretation-prop}, the simplification of the quantifier clauses discussed at the end of Subsection \ref{maj_upon_bounding-section} and, for the disjunction clause, the discussion at the end of Subsection \ref{uniform-booleans-sec}. Notice that elements of a monotone finite set are necessarily monotone. The reader should compare these clauses with the corresponding clauses of the bounded functional interpretation (cf.\ Section 3 of \cite{FerreiraOliva(05)}). 

A central feature of the bounded functional interpretation is the presence of intensional majorizability relations $\iBmaj_\sigma$. We expand the language of $\HAomega$ with primitive binary relation symbols $\iBmaj_\sigma$, one for each arithmetical type $\sigma$, infixing between elements of type $\sigma$. With new relation symbols, we need to expand our $\Imb$-interpretation. This is done by specifying the information relations associated to the binary relation symbols $\iBmaj_\sigma$. We let $\tau^+_{\iBmaj_\sigma}$ and $\tau^-_{\iBmaj_\sigma}$ be the empty tuples and let $\langle x \iBmaj_\sigma y \rangle$ be $x \Bmaj_\sigma y$.

We accept the universal axioms 
\eqleft{
\begin{array}{l}
		\forall x^\NN, y^\NN (x \iBmaj_{\NN} y  \leftrightarrow  x \leq_\NN y) \\[1mm]
		\forall f^{\sigma \to \tau} \forall g^{\sigma \to \tau} (f \iBmaj_{\sigma \to \tau} g \to \forall x^\sigma \forall y^\sigma (x \iBmaj_\sigma y \to fx \iBmaj_\tau gy \wedge gx \iBmaj_\tau gy))
	\end{array}	
}
as well as the rule
\[
\begin{prooftree}
    \Gamma , x \iBmaj_\sigma y \proves sx \iBmaj_\tau \!ty \wedge tx \iBmaj_\tau \!ty
    \justifies
    \Gamma \proves s \iBmaj_{\sigma \to \tau} t
\end{prooftree}
\]
where $s$ and $t$ are terms of the source language, $\Gamma$ is a tuple of formulas whose interpretations have no negative $\Imb$-witnesses, and $x$ and $y$ are variables that do not appear free in the conclusion. The reader familiar with our paper \cite{FerreiraOliva(05)} will recognize these axioms and the rule. Note that, by Remark \ref{remark-universal} on universal statements and the refinement (2) of Subsection \ref{refinements}, the above (universal) axioms are $\Imb$-interpretable and the above rule preserves $\Imb$-interpretability.

Given an arithmetical type $\sigma$, let $B_\sigma(x,y)$ be the formula $x \iBmaj_\sigma y$, with distinguished variable $x$. We claim that $B_\sigma(x,y)$ subsumes $\witness{x}{\sigma}$. To see this, we must show that $$\forall y^\sigma (\witness{y}{\sigma} \to \forall x^\sigma (x \iBmaj_\sigma y \to \witness{x}{\sigma}))$$ is $\U$-interpretable. This amounts to finding a closed term $t^{\sigma \to \sigma}$ such that the target theory proves that, for all $a^\sigma$, $$\forall y^\sigma (y \Bmaj_\sigma a \to \forall x^\sigma (x \Bmaj_\sigma y \to x \Bmaj_\sigma ta)).$$  Since Bezem's strong majorizability relation is transitive, we can take $ta = a$. We denote these $B$-restricted quantifications by $\forall x \iBmaj_\sigma y \; (\ldots)$ and $\exists x \iBmaj_\sigma y \; (\ldots)$, to be close to the notation of \cite{FerreiraOliva(05)}. Clearly, these $B$-restricted quantifications are uniform (in the sense of Definition \ref{def-bounded-quant}). 

The two clauses at the beginning of Subsection \ref{uniform-section} give the following interpretations of the bounded quantifications:
\eqleft{
\begin{array}{lcl}
	(\forall x \iBmaj_\sigma \!c \, A(x))^{\Imb} & \mathrm{\,is\,} & \existstilde \pvec a \foralltilde \pvec b \forall x \Bmaj_\sigma \!c \, A(x)_{\Imb} (\pvec a, \pvec b) \\[1mm]
	(\exists x \iBmaj_\sigma \!c \,A(x))^{\Imb} & \mathrm{\,is\,} & \existstilde \pvec a \foralltilde \pvec W \, \exists x \Bmaj_\sigma \!c \,\foralltilde \pvec b \in \pvec W \, A(x)_{\Imb} (\pvec a, \pvec b).
\end{array}	
}
The reader is invited to make a comparison between these clauses with the corresponding clauses of the bounded functional interpretation.

It follows, by Theorem \ref{collection-thm}, that we are able to interpret the following form of collection: 
$$\forall x \iBmaj_\sigma w \, \exists n^\NN C(x,n) \to \exists k^\NN \, \forall x \iBmaj_\sigma w \, \exists n\leq k \, C(x,n),$$ 
where $C(x,n)$ can be taken arbitrarily and $\sigma$ is an arithmetical type. By the same token (see Theorem \ref{contra-collection-thm}), the following form of contra-collection is interpretable: 
$$\forall k^\NN \, \exists x \iBmaj_\sigma \!w \, \forall n \leq k \, C(x,n) \to \exists x \iBmaj_\sigma \!w \, \forall n^\NN C(x,n),$$
where $C(x,n)$ has no positive witnesses (and $\sigma$ is an arithmetical type).

As a matter of fact, the above forms of collection and contra-collection can be vastly generalized. They can take the form:
$$\forall x \iBmaj_\sigma \!w \,\exists y^\tau C(x,y) \to \exists z^\tau \forall x \iBmaj_\sigma w \,\exists y\iBmaj_\tau z \, C(x,y),$$ 
where $C(x,y)$ is arbitrary and, if we allow a new construct (to be described below),
$$\forall z^\tau \exists x \iBmaj_\sigma \!w \, \forall y \iBmaj_\tau \!z \, C(x,y) \to \exists x \iBmaj_\sigma \!w \, \forall y^\tau C(x,y),$$
where $C(x,y)$ has no positive witnesses (and $\sigma$ and $\tau$ are  arithmetical types). These very general forms of collection and contra-collection are (essentially) the forms that appear in \cite{FerreiraOliva(05)}.

Let us first look at the collection principle. A simple calculation shows that we need to obtain monotone terms $t$, $\pvec q$ and $\pvec s$ such that, for every monotone $a^\tau$, and all monotone $\pvec c$ and $\mathfrak{D}$, we have the following: if
$$\foralltilde \pvec D \in \pvec q a \pvec c \mathfrak{D} \forall x \Bmaj_\sigma \!w \, \exists y \Bmaj_\tau \!a \, \foralltilde \pvec d \in \pvec D \,\iInter{C(x,y)}{\pvec c}{\pvec d}$$
then
$$\exists z \Bmaj_\tau t a\pvec c \, \foralltilde \pvec D \in \mathfrak{D} \forall x \Bmaj_\sigma \!w \, \exists y \Bmaj_\tau \!z \, \foralltilde \pvec d \in \pvec D \iInter{C(x,y)}{\pvec s a \pvec c}{\pvec d}.$$
We let $ta\pvec c = a$, $\pvec q a \pvec c \mathfrak{D} = \mathfrak{D}$ and $\pvec s a \pvec c = \pvec c$. This clearly works: in the consequent, just take $z$ to be $a$.

Let us now consider the interpretation of the contra-collection principle. Another simple calculation shows that we need to obtain terms $t$ and $\pvec q$ such that, for every monotone $F^{\tau^*}$ and $\pvec D$, we have the following: if 
$$\foralltilde a \in t \pvec D F \,\foralltilde \pvec D'\!\in \pvec q \pvec D F \, \forall z \Bmaj_\tau \!a \,\exists x \Bmaj_\sigma \!w \,\foralltilde \pvec d \in \pvec D' \forall y \Bmaj_\tau \!z \, \iInter{C(x,y)}{}{\pvec d}$$
then
$$\exists x \Bmaj_\sigma \! w \, \foralltilde a \in F \,\foralltilde \pvec d \in \pvec D \forall y \Bmaj_\tau a \,\iInter{C(x,y)}{}{\pvec d}.$$

At this point, we need to have in the target theory a construct $\mathrm{m}_\tau$, of type $\tau^* \to \tau$, such that $\mathrm{m}_\tau (F)$ gives the maximum of the elements of $F$. We are referring here to the maximum, as defined in (\ref{max-functional}) of Subsection \ref{star-majorizability} between two elements of a given arithmetical type. Obviously, it makes sense to take the maximum of $F$ because $F$ is a non-empty and finite set (of monotone elements). Note, also, that $\mathrm{m}_\tau (F)$ is monotone. 

We let $t \pvec D F = \{\mathrm{m}_\tau(F)\}$ and $\pvec q \pvec D F = \{ \pvec D\}$. With these definitions, the antecedent above is
$$\forall z \Bmaj_\tau \!\mathrm{m}_\tau(F) \,\exists x \Bmaj_\sigma \!w \,\foralltilde \pvec d \in \pvec D \forall y \Bmaj_\tau \!z \, \iInter{C(x,y)}{}{\pvec d}.$$
The consequent is true by taking $x^\sigma$ to be the element that works in the antecedent when $z^\tau$ is $\mathrm{m}_\tau(F)$. To see this, note that if $y \Bmaj_\tau \!a$ and $a \in F$, then $a \Bmaj_\tau \mathrm{m}_\tau(F)$ and, by transitivity, $y \Bmaj_\tau \mathrm{m}_\tau(F)$.

\section{Conclusion}
\label{conclusion-section}

We presented a flexible approach to functional interpretations. Against the background uniform interpretation, several choices can be made, leading to different functional interpretations. We discussed, in particular, so-called canonical interpretations of function types over two possible base arithmetical interpretations of the ground type $\NN$. We are referring to the canonical upon the precise interpretation $\Icp$ and to the canonical upon the bounding interpretation $\Icb$. Natural questions can be raised. What kind of choice principles are realized by these interpretations? Or, more ambitiously, can we state and prove characterization theorems for $\Icp$ and for $\Icb$ (as done by Yasugi in \cite{Yasugi(63)} concerning G\"odel's {\em dialectica} interpretation, or by Ferreira and Oliva concerning the bounded functional interpretation \cite{FerreiraOliva(05)}, or by van den Berg, Briseid and Safarik concerning their interpretation of nonstandard arithmetic in \cite{BergBriseidSafarik(12)})? We also saw that these interpretations give rise to models of G\"odel's theory $\GodelT$, namely the typestructures $\mathcal{C^=}$ and $\mathcal{C^\leq}$. Do these typestructures and their hereditary versions have interesting properties? How are the hereditary versions of $\mathcal{C^=}$ and $\mathcal{C^\leq}$ related to typestructures based on the notions of continuous or strongly majorizable functionals, due to Kleene \cite{Kleene(59)} and Kreisel \cite{Kreisel(59)} (see also \cite{Troelstra(73)}), and Bezem \cite{Bezem(85)}, respectively? Are these typestructures models of bar-recursion (see \cite{Spector(62)} and \cite{Troelstra(73)})?

There are other directions of development. The framework presented in this paper can be applied to classical logic (see \cite{Oliva(25)}). It can also be extended to the real numbers (in the manner of \cite{Ferreira(21)}) and to abstract spaces, as used in proof mining studies (see \cite{Kohlenbach(08)}). Principles of collection and contra-collection also hold in the setting of abstract spaces. We want to emphasize these principles because they play an important role, explanatory for sure, but also in the discovery of new results in mathematics. It was observed in \cite{FerreiraLeusteanPinto(19)} that collection principles entail forms of Heine-Borel compactness which can be applied to good effect and used to give abstract (or qualitative) proofs of known convergence results. This was discussed and analyzed in detail in \cite{FerreiraLeusteanPinto(19)} for the proof of a strong convergence result of Browder \cite{Browder(67)}. Even though Heine-Borel compactness is false in the infinite dimensional setting of Hilbert spaces in which Browder works, the (false) principle can indeed be used to obtain the strong convergence result. The quantitative analysis of Browder's result, as done firstly by Kohlenbach in \cite{Kohlenbach(11)}, can be luminously explained as the proof mining analysis of the abstract proof from the false principle of Heine-Borel compactness. Other analyses of this sort have recently been made, and we want to draw attention to the new situation of Dykstra's projection algorithm (see \cite{Pinto(25)} and the careful proof-theoretic analysis in \cite{Pinto(25c)}).

Browder's original proof in \cite{Browder(67)} used a weak compactness argument in Hilbert spaces. Hilbert spaces are naturally generalized to the non-linear setting by so-called Hadamard spaces (complete $\mathrm{CAT(0)}$ spaces). In some of the most interesting recent developments, Pedro Pinto and others have been able to prove new mathematical results by, in fact, using (in general, false) Heine-Borel compactness arguments. This was done in \cite{DinisPinto(23)} and \cite{PintoPischke(26)}, or in \cite{Pinto(26)} where an abstract proof from false principles is presented.

We gave a particular emphasis to a uniform interpretation that operates in the background of computationally meaningful functional interpretations. We believe that the non-set-theoretical principles embodied by the uniform interpretation are worth probing because they furnish not only guidance in some proof mining analyses but also a method for extending results of the mathematical literature to new settings. Furthermore, the existence and role of these set-theoretically false principles are possibly a source of interest for the philosophy and foundations of mathematics.

\paragraph*{Acknowledgments.} We thank Gilda Ferreira for several interesting comments and remarks on an earlier draft of the paper. We also thank the referee for the comments made and the questions posed.

\end{document}